\documentclass[11pt,reqno]{amsart}

\usepackage{amsmath,amssymb,amsthm,mathtools}
\usepackage{tikz-cd}
\usepackage{enumitem}
\usepackage{hyperref}
\usepackage{aliascnt}
\usepackage[margin=1.1in]{geometry}
\usepackage{graphicx}
\usepackage{microtype}
\usetikzlibrary{arrows.meta,calc,decorations.pathreplacing}

\hypersetup{
  colorlinks=true,
  linkcolor=blue,
  citecolor=blue,
  urlcolor=blue
}

\newtheorem{theorem}{Theorem}[section]
\newaliascnt{proposition}{theorem}
\newtheorem{proposition}[proposition]{Proposition}
\aliascntresetthe{proposition}
\newaliascnt{lemma}{theorem}
\newtheorem{lemma}[lemma]{Lemma}
\aliascntresetthe{lemma}
\newaliascnt{corollary}{theorem}
\newtheorem{corollary}[corollary]{Corollary}
\aliascntresetthe{corollary}
\newaliascnt{claim}{theorem}

\aliascntresetthe{claim}
\newaliascnt{definition}{theorem}
\newtheorem{definition}[definition]{Definition}
\aliascntresetthe{definition}
\newaliascnt{remark}{theorem}

\aliascntresetthe{remark}
\newaliascnt{convention}{theorem}
\newtheorem{convention}[convention]{Convention}
\aliascntresetthe{convention}

\usepackage[nameinlink,capitalize,noabbrev]{cleveref}
\crefname{convention}{Convention}{Conventions}
\Crefname{convention}{Convention}{Conventions}

\newcommand{\Mod}{\operatorname{Mod}}
\newcommand{\Diff}{\operatorname{Diff}}
\newcommand{\Homeo}{\operatorname{Homeo}}
\newcommand{\Stab}{\operatorname{Stab}}
\newcommand{\Aut}{\operatorname{Aut}}

\newcommand{\Fix}{\operatorname{Fix}}
\newcommand{\Ax}{\operatorname{Ax}}
\newcommand{\Trans}{\operatorname{Trans}}
\newcommand{\Isom}{\operatorname{Isom}}

\newcommand{\normal}[1]{\left\langle\!\left\langle #1\right\rangle\!\right\rangle}
\newcommand{\Z}{\mathbb Z}

\newcommand{\R}{\mathbb R}
\newcommand{\cE}{\mathcal E}
\newcommand{\cD}{\mathcal D}
\newcommand{\cN}{\mathcal N}

\newcommand{\cA}{\mathcal A}
\newcommand{\cB}{\mathcal B}

\newcommand{\into}{\hookrightarrow}
\newcommand{\onto}{\twoheadrightarrow}

\title[Embedded surfaces with trivial extendable mapping class groups]{Embedded surfaces with trivial extendable mapping class groups
in simply connected \(4\)-manifolds}
\date{}
\author{Weizhe Niu}
\address{Yau Mathematical Sciences Center, Tsinghua University}
\email{weizheniu@mail.tsinghua.edu.cn}
\begin{document}

\begin{abstract}
For every \(g\geq 3\), every closed, connected, oriented, simply connected
smooth \(4\)-manifold \(X\), and every knot \(K\subset S^3\),
we construct infinitely many pairwise topologically inequivalent
smoothly embedded oriented genus-\(g\) surfaces
\(F\subset X\) whose orientation-preserving extendable mapping class
subgroups are trivial in both the topological and smooth categories and
whose first Alexander modules are isomorphic to the Alexander module of
\(K\). In particular, there are infinitely
many such surfaces with vanishing first Alexander module. The
construction is supported in a \(4\)-ball. Although its Alexander data
are prescribed independently, the surfaces are distinguished and their
mapping-class rigidity is detected by the nonabelian centralizer
structure of their exterior groups.
\end{abstract}

\maketitle
\tableofcontents
\section{Introduction}\label{sec:introduction}

A smoothly embedded surface may have intrinsic symmetries that are not
induced by ambient homeomorphisms.  Let \(X\) be a closed, connected,
oriented smooth \(4\)-manifold, and let \(F\subset X\) be a smoothly
embedded, closed, connected, oriented genus-\(g\) surface with a fixed
orientation-preserving identification \(\Sigma_g\cong F\).  Its smooth
and topological orientation-preserving \textit{extendable subgroups} are
\[
E_X^+(F)=
\operatorname{im}\!\left(
\pi_0\Diff^+(X,F)\longrightarrow\Mod(\Sigma_g)
\right)
\]
and
\[
E_{X,\mathrm{TOP}}^+(F)=
\operatorname{im}\!\left(
\pi_0\Homeo^+(X,F)\longrightarrow\Mod(\Sigma_g)
\right),
\]
respectively.  Here every ambient homeomorphism or diffeomorphism is required to
preserve both the orientation of \(X\) and the orientation of \(F\), and we use the canonical identification of the
homeomorphism and diffeomorphism mapping class groups of \(\Sigma_g\)
\cite[Theorem~B]{HatcherTorusTrick}.  Since every diffeomorphism is a homeomorphism,
\[
E_X^+(F)\leq E_{X,\mathrm{TOP}}^+(F).
\]
For \(X=S^4\), write \(E^+(F)\) and \(E_{\mathrm{TOP}}^+(F)\).
The basic realization problem asks how small these extendable subgroups
can be.

Two oriented surfaces in \(X\) are \emph{topologically equivalent} if an
orientation-preserving homeomorphism of \(X\) carries one to the other
and preserves their orientations.  They are \emph{smoothly equivalent}
if such a homeomorphism can be chosen to be a diffeomorphism.  Thus
topological inequivalence implies smooth inequivalence and nonisotopy.

We also record the first Alexander module.  Write
\[
E_F=X\setminus\operatorname{int}\nu F.
\]
When \(H_1(E_F;\Z)\cong\Z\) with positive meridian as preferred
generator, let \(\widetilde E_F\to E_F\) be the corresponding infinite
cyclic cover and set
\[
\cA_1(F)=H_1(\widetilde E_F;\Z).
\]
The positive meridian determines the deck transformation \(t\), making
\(\cA_1(F)\) a left \(\Lambda=\Z[t,t^{-1}]\)-module.  For a knot
\(K\subset S^3\), \(\cA(K)\) denotes its classical Alexander module with
the same convention.

\begin{theorem}\label{thm:main}
For every integer \(g\geq3\) and every classical knot \(K\subset S^3\),
there exist infinitely many pairwise topologically inequivalent smoothly
embedded, closed, connected, oriented genus-\(g\) surfaces
\[
F_{g,n}\subset S^4
\]
such that
\[
E_{\mathrm{TOP}}^+(F_{g,n})=E^+(F_{g,n})=1,
\qquad
\cA_1(F_{g,n})\cong\cA(K)
\]
as \(\Lambda\)-modules.
\end{theorem}

\begin{corollary}\label{cor:simply-connected-ambient}
Let \(X\) be a closed, connected, oriented, simply connected smooth
\(4\)-manifold.  For every integer \(g\geq3\) and every classical knot
\(K\subset S^3\), there exist infinitely many pairwise topologically
inequivalent smoothly embedded, closed, connected, oriented genus-\(g\)
surfaces
\[
F_{g,n}\subset X
\]
such that
\[
E_{X,\mathrm{TOP}}^+(F_{g,n})=E_X^+(F_{g,n})=1,
\qquad
\cA_1(F_{g,n})\cong\cA(K)
\]
as \(\Lambda\)-modules.
\end{corollary}

Taking \(K\) to be the unknot gives infinitely many examples with
vanishing first Alexander module.  Thus this abelian invariant can be
prescribed independently of the rigidity, which is detected by
\textit{nonabelian} information in the exterior group.

The question is part of the realization problem in
\cite[Problem~4.39]{K3Problems}.  For the standard unknotted surface in
\(S^4\), Hirose identified the extendable subgroup with the stabilizer of
the Rokhlin quadratic form, hence a large subgroup of
\(\Mod(\Sigma_g)\) \cite{HiroseUnknotted}.  Related work treats surfaces
for which every mapping class extends, criteria and obstructions for
individual classes, and extensions of prescribed periodic maps
\cite{HiroseYasuhara,LawandeSaha,WangWang}.  Toward small subgroups,
Liu--Ni--Sun--Wang constructed locally flat knotted tori with finite
extendable subgroup \cite[Corollary~5.7]{LiuNiSunWang}, while Q.~Liu
constructed locally flat surfaces in every positive genus whose
extendable subgroup has finite image on first homology
\cite[Theorem~1.1]{QilingLiu}.  The former result does not identify the
finite subgroup, and the latter does not control mapping classes acting
trivially on first homology.

Very recently, Baykur and Sunukjian constructed a genus-\(40\) totally
geodesic surface in a closed hyperbolic \(4\)-manifold whose smooth and
topological extendable mapping class groups are both trivial
\cite{BaykurSunukjian}.  Their ambient manifold is not simply connected.
In the same paper, their iterated rim-surgery constructions in simply
connected \(4\)-manifolds control the projective action of the
extendable subgroup on first homology, rather than the full extendable
subgroup.  To our knowledge, \cref{thm:main} gives the first smoothly
embedded closed oriented surfaces of positive genus in \(S^4\) whose
full orientation-preserving extendable subgroup is trivial even in the
topological category.

The closest direct precursor is the computation in \cite{NiuSingle} for
a single ordinary untwisted rim surgery on the standard unknotted
surface.  There the extendable subgroup is determined by the Rokhlin
quadratic form, the homology class of the rim curve, and a symmetry of
the knot exterior; in particular, the entire Torelli group remains
extendable.  One surgery therefore detects only the homology class of
the rim curve and cannot give a trivial extendable subgroup.  Here an
ordered sequence of rim surgeries retains enough nonabelian information
to recover the individual curves.

Ordinary untwisted \textit{rim surgery} along \(d\) replaces a product annulus
\(S^1\times I\) by \(S^1\times J^\circ\), where
\(J^\circ\subset B^3\) is a knot with a small open arc removed, without
changing the ambient \(4\)-manifold. The same annular replacement
appears in earlier work of Asano and Livingston
\cite{Asano1976,Livingston1988}; see also
Fintushel--Stern \cite{FintushelStern}. We use it
first to make the surface group visible in the exterior and then to
record curves determining every mapping class. In the following, we outline the proof strategy.

Write the standard unknotted genus-\(g\) surface as
\(\partial(P\times I)\) for \[P=\begin{cases}
\Sigma_{r,1}, & g=2r,\\
\Sigma_{r,2}, & g=2r+1.
\end{cases}\]
Rim surgeries along a complete system
of meridian-disk boundaries make the two horizontal copies of
\(\pi_1(P)\) visible, and further surgeries along the vertical boundary
separate their images.  The resulting \(F_0\subset S^4\) has injective
push-off
\[
\rho_0:\pi_1(\Sigma_g,p)\longrightarrow
\pi_1(S^4\setminus\operatorname{int}\nu F_0,*).
\]
This injectivity converts exterior-group information back into surface-
group information; see \cref{sec:disk,sec:seed}.

Choose an ordered collection of nonseparating curves
\[
\cD=(d_1,\ldots,d_N)
\]
that fills \(\Sigma_g\), has trivial label-preserving stabilizer in
\(\Mod(\Sigma_g)\), and satisfies
\[
i(d_i,d_{i+1})=1
\qquad(1\leq i<N).
\]
The labels are the positions in the tuple.  Such a collection is built
in \cref{thm:filling-sequence}. 

Perform successive surgeries along these curves
use hyperbolic knots with pairwise nonisomorphic meridian-marked groups. Let \(F_i\), \(G_i\), and
\[
\rho_i:\pi_1(\Sigma_g,p)\longrightarrow G_i
\]
denote the surface, exterior group, and push-off after stage \(i\).  If
\(\mu\) is the positive meridian, \(h_i=\rho_{i-1}(d_i)\), and \(K_i\)
is the knot group at that stage, then
\[
G_i\cong
G_{i-1}*_{\langle\mu,h_i\rangle}
\bigl(K_i\times\langle h_i\rangle\bigr).
\]
There is a canonical retraction
\[
r_i:G_i\longrightarrow G_{i-1}
\]
with
\[
r_i\circ\rho_i=\rho_{i-1},
\qquad
\ker r_i=\normal{[K_i,K_i]}_{G_i}.
\]
Thus quotienting by the new knot-group commutator recovers the preceding
exterior and push-off.  At each crossing of \(d_i\), the new push-off
acquires a suitably based preferred longitude.  Section~\ref{sec:rimlocal}
gives the local path identity and its one-crossing consequence.

The condition \(i(d_i,d_{i+1})=1\) gives the next curve a controlled
normal form, from which Bass--Serre theory propagates rigid centralizer
behavior; see \cref{sec:powerclean,sec:initialcarriers,sec:latercarriers}.
The resulting classification intrinsically recognizes, up to conjugacy,
\[
B_i=K_i\times\langle h_i\rangle
\]
from the exterior group and positive meridian: it is the relevant
nonabelian centralizer with infinite cyclic center satisfying
\[
Z(B_i)=\langle h_i\rangle,
\qquad
(B_i/Z(B_i),\bar\mu)\cong(K_i,m_i),
\]
where \(m_i\) is the chosen knot meridian.  Pairwise
nonisomorphism of the meridian-marked groups \((K_i,m_i)\)
distinguishes the stages; see \cref{sec:centralizers,sec:recognition}.

The relevant invariant is the exterior group together with the positive
meridian and push-off.  A pair homeomorphism initially acts on these
data up to basepoint and positive normal-section choices; the
normalization of \cref{sec:peripheral} removes both ambiguities for the
groups used here.  Hence an extendable mapping class \(f\) yields
\(\Phi_N\in\Aut(G_N)\) with
\[
\Phi_N(\mu)=\mu,
\qquad
\Phi_N\circ\rho_N=\rho_N\circ f_*.
\]

The reverse induction proceeds at stage \(i\) in the order
\emph{recognize \(B_i\), recover \(d_i\), then descend}.
Recognition preserves the conjugacy class of
\(Z(B_i)=\langle h_i\rangle\). Injectivity of
\(\rho_{i-1}\) alone does not reflect conjugacy from
the exterior group into the surface subgroup.
The surface-subgroup conjugacy results of
\cref{sec:initialcarriers,sec:latercarriers,sec:slope-recovery}
provide the additional step needed to recover the
unoriented curve \(d_i\). Only then does the invariant
quotient by the normal closure of \([B_i,B_i]\) recover
\((G_{i-1},\rho_{i-1})\), so the same mapping class
descends. Repeating from \(N\) to \(1\) fixes every
labelled curve of \(\cD\), and its trivial labelled
stabilizer forces \(f=1\); the quotient step is proved
in \cref{sec:collapse}.

Alexander control is independent of rigidity.  Mayer--Vietoris for the
meridional infinite cyclic cover makes each local knot module a direct
summand (\cref{cor:total-alexander}).  Friedl supplies a hyperbolic knot
with the prescribed module, and Kalfagianni supplies infinitely many
Alexander-trivial hyperbolic knots with pairwise nonisomorphic groups
\cite{FriedlSeifert,KalfagianniVolume}.  Use the prescribed-module knot at one preliminary stage and the
Alexander-trivial knots at all remaining stages. Varying one rank-two factor
changes an intrinsically recognized knot group, hence the surface, while
leaving the module fixed.  All modifications lie in a \(4\)-ball, which
also gives \cref{cor:simply-connected-ambient}.

\medskip
\noindent\textbf{Organization.}
Sections~\ref{sec:peripheral}--\ref{sec:seed} establish normalized
push-offs, the one-step surgery formulas, and the initial injective
surface subgroup.  Sections~\ref{sec:powerclean}--\ref{sec:collapse}
recognize and remove surgeries by Bass--Serre and centralizer methods.
Sections~\ref{sec:filling}--\ref{sec:mainproof} construct the ordered
filling system and prove rigidity.  Section~\ref{sec:homology} computes
the Alexander module, chooses the knots, proves the main results, and
records group and exterior homology.

\medskip
\par\noindent\textbf{Acknowledgement of AI use.}
The author used Prism, Overleaf AI, and Grammarly to assist with LaTeX formatting and English-language editing. DeepSeek was used to assist with the algebraic computations in Propositions~\ref{prop:surface-subtree}, \ref{prop:complete-centralizers}, and~\ref{prop:alexander-one-step}, as well as with proofreading and the preparation of some of the figures.

\section{The push-off homomorphism and extendability}
\label{sec:peripheral}

A homeomorphism of pairs may change the push-off
homomorphism in two ways: changing the exterior basepoint
conjugates its image, while changing the positive
normal-circle section introduces powers of the meridian.
The normalization of normal sections is classical;
see \cite[\S2]{Livingston1985}.
We show that, under the homological conditions satisfied
by our exterior groups, both ambiguities can be removed.
This gives the exact relation between the surface and
exterior automorphisms needed later.

Fix \(g\geq3\), a closed oriented surface \(\Sigma_g\), and
\(p\in\Sigma_g\), and write \(\pi=\pi_1(\Sigma_g,p)\).
For \(f\in\Mod(\Sigma_g)\), choose a representative
\(\bar f:\Sigma_g\to\Sigma_g\) and a path
\(\lambda_f:p\to\bar f(p)\).  They determine
\[
f_*(\gamma)=\lambda_f\,\bar f(\gamma)\,\lambda_f^{-1}.
\]
Changing \(\lambda_f\) changes \(f_*\) by an inner automorphism. Whenever an equation uses \(f_*\), we fix one representative.

A \emph{parametrized surface} is an oriented embedding
\(\Sigma_g\to S^4\) with its image identified with \(\Sigma_g\).
For such an \(F\), the class \([F]=0\in H_2(S^4;\mathbb Z)\).
Its self-intersection, and hence the Euler class of its oriented normal
\(2\)-plane bundle, vanishes, so the normal bundle is trivial.
Choose a section \(\sigma:\Sigma_g\to\partial\nu F\) of the
positively oriented normal circle bundle, and put
\[
 *=\sigma(p),\qquad
 E_F=S^4\setminus\operatorname{int}\nu F,\qquad
 G_F=\pi_1(E_F,*).
\]
Write \(\iota:\partial\nu F\hookrightarrow E_F\) for inclusion. The
\emph{push-off homomorphism} is
\[
 \rho_F=(\iota\circ\sigma)_*:\pi\longrightarrow G_F.
\]
Let
\(
\mu\in G_F
\)
be the positive meridian, namely the positively oriented boundary of a
normal disk to \(F\).  The chosen section gives a product identification
\(\partial\nu F\cong\Sigma_g\times S^1\), under which the section
and the positive meridian determine the two factors of
\(\pi_1(\partial\nu F,*)\cong\pi\times\mathbb Z\). Its image in \(G_F\) is therefore
\[
 \langle\rho_F(\pi),\mu\rangle,
\]
and \(\mu\) commutes with \(\rho_F(\pi)\).

For another section \(\sigma'\), the two sections differ in this
trivialization by a map \(a:\Sigma_g\to S^1\).  Put
\(\eta=a_*:\pi\to\mathbb Z\).  After choosing a boundary path from \(*\) to \(*'\) to identify them, the new based
push-off is
\[
 \rho_F'(\gamma)
 =q\,\rho_F(\gamma)\mu^{\eta(\gamma)}q^{-1},
 \qquad q\in\langle\rho_F(\pi),\mu\rangle.
\]
To see this, write the two sections as $\sigma(x)=(x,1)$ and
$\sigma'(x)=(x,a(x))$ in the chosen trivialization.
Choose a path $\alpha$ from $*=\sigma(p)$ to
$*'=\sigma'(p)$ in the fiber over $p$ of the normal
circle bundle $\partial\nu F\to\Sigma_g$.
For $\gamma\in\pi$, the loop
$\alpha\,\sigma'(\gamma)\,\alpha^{-1}$
projects to $\gamma$ in $\Sigma_g$. Its circle
coordinate has winding number $\eta(\gamma)$,
since the contributions of $\alpha$ and
$\alpha^{-1}$ cancel. Thus
\[
[\alpha\,\sigma'(\gamma)\,\alpha^{-1}]
=(\gamma,\eta(\gamma))
\quad\text{in }\pi_1(\Sigma_g,p)\times\mathbb Z.
\]
Its image in $G_F$ is therefore
$\rho_F(\gamma)\mu^{\eta(\gamma)}$.

For any other boundary path $\beta$ from $*$ to $*'$,
put $q=[\beta\alpha^{-1}]$ in $G_F$. Then
\[
[\beta\,\sigma'(\gamma)\,\beta^{-1}]
=q\,\rho_F(\gamma)\mu^{\eta(\gamma)}q^{-1}.
\]
The same $q$ works for every $\gamma$.
Since $q$ commutes with $\mu$, both choices of boundary path identify
the positive meridian with the same element $\mu\in G_F$.

\begin{proposition}[Naturality of the push-off homomorphism]
\label{prop:peripheral-naturality}
Let \(F,F'\subset S^4\) be smoothly embedded, closed, oriented
parametrized surfaces, and let
\[
 \Psi:(S^4,F)\longrightarrow(S^4,F')
\]
be an orientation-preserving homeomorphism whose restriction preserves
the surface orientation and induces \(f\in\Mod(\Sigma_g)\).
Choose the positive meridians \(\mu,\mu'\) compatibly and fix a
representative \(f_*:\pi\to\pi\).
The induced exterior homeomorphism determines an
outer isomorphism
\(
G_F\longrightarrow G_{F'}.
\)
After choosing a change-of-basepoint path in
\(\partial\nu F'\), this outer isomorphism has a representative
\[
\Phi:G_F\longrightarrow G_{F'}
\]
such that
\(
\Phi(\mu)=\mu'
\) and
\begin{equation}
\label{eq:full-peripheral-naturality}
 \Phi(\rho_F(\gamma))
 =q\,\rho_{F'}(f_*(\gamma))
       (\mu')^{\theta(\gamma)}q^{-1}
 \qquad(\gamma\in\pi),
\end{equation}
where \(q\in\langle\rho_{F'}(\pi),\mu'\rangle\)
and \(\theta:\pi\to\mathbb Z\) work for every
\(\gamma\).
\end{proposition}

\begin{proof}
The surfaces \(F\) and \(F'\) are smoothly embedded so locally flat. Equip \(\Psi(\nu F)\) with the normal disc-bundle structure transported
from \(\nu F\).  Thus \(\Psi(\nu F)\) and \(\nu F'\) are two normal
bundles of \(F'\).  Let
\[
i_1:\Psi(\nu F)\hookrightarrow S^4,
\qquad
i_0:\nu F'\hookrightarrow S^4
\]
denote their inclusions.  By the uniqueness theorem for topological
normal bundles in dimension four
\cite[Theorem~9.3D]{FreedmanQuinn}, there exist a bundle isomorphism
\[
b:\Psi(\nu F)\longrightarrow \nu F'
\]
over \(F'\) and an ambient isotopy \(H_t\) of \(S^4\), fixed pointwise
on \(F'\), such that
\[
H_1\circ i_1=i_0\circ b.
\]
Replacing \(\Psi\) by \(H_1\circ\Psi\), we may therefore assume that
\(\Psi(\nu F)=\nu F'\) and that
\[
\Psi|_{\nu F}:\nu F\longrightarrow\nu F'
\]
is a bundle isomorphism covering \(\Psi|_F\).  Since \(\Psi\) preserves
the orientations of both \(S^4\) and \(F\), this bundle isomorphism
preserves the induced orientation of the normal bundle.

Let $\sigma$ and $\sigma'$ be the chosen normal-circle
sections for $F$ and $F'$, respectively, with basepoints
$*=\sigma(p)$ and $*'=\sigma'(p)$. Under the parametrizations, let
\[
\bar f:\Sigma_g\longrightarrow\Sigma_g
\]
be the surface restriction of \(\Psi\).
Choose a path \(\lambda_f:p\to\bar f(p)\) to define \(f_*\).
Using the trivializations given by the sections, write
\[
 \Psi|_{\nu F}(x,1)=(\bar f(x),a(x)),\qquad
 a:\Sigma_g\longrightarrow S^1.
\]
Composing \(\lambda_f\) with \(\sigma'\) from \(*'\) to
\(\sigma'(\bar f(p))\), then follow the normal circle over \(\bar f(p)\)
to \(\Psi(*)\).  The concatenation
\[
\omega_0:*'\longrightarrow\Psi(*)
\]
lies in \(\partial\nu F'\) and gives a representative
\[
\Phi_0:G_F\longrightarrow G_{F'}
\]
by
\[
\Phi_0([\ell])
=[\omega_0\,\Psi(\ell)\,\omega_0^{-1}]
\]
for every loop $\ell$ in $E_F$ based at $*$.

For \(\gamma\in\pi\),
the loop
\[
 \omega_0\,\Psi(\sigma(\gamma))\,\omega_0^{-1}
\]
projects to \(\lambda_f\bar f(\gamma)\lambda_f^{-1}\).

We now compute the image of the push-off under $\Phi_0$.
For $\gamma\in\pi$, the loop $\sigma\circ\gamma$
represents $\rho_F(\gamma)$. Its image under $\Psi$ has coordinates
\[
t\longmapsto
(\bar f(\gamma(t)),a(\gamma(t))).
\]
After using $\omega_0$ to identify the basepoints,
the surface coordinate represents $f_*(\gamma)$.
The circle coordinate has winding number
$\theta(\gamma)=\deg(a\circ\gamma)$, since the
fiber path in $\omega_0$ is traversed in opposite
directions at the two ends. Therefore
\[
\Phi_0(\rho_F(\gamma))
=
\rho_{F'}(f_*(\gamma))(\mu')^{\theta(\gamma)}.
\]
Since $\Psi$ preserves the ambient ans surface orientations, it also preserves the normal orientation, so it sends a positive
meridian to a positive meridian. Changing the basepoint
along $\omega_0$ preserves the positive generator of the circle
factor. Thus $\Phi_0(\mu)=\mu'$.

For any other boundary path \(\omega:*'\to\Psi(*)\), put
\[
 q=[\omega\omega_0^{-1}]
 \in\operatorname{im}\bigl(\pi_1(\partial\nu F',*')\to G_{F'}\bigr)
 =\langle\rho_{F'}(\pi),\mu'\rangle.
\]
The resulting representative is
\(\Phi(x)=q\Phi_0(x)q^{-1}\) for every \(x\in G_F\), proving
\eqref{eq:full-peripheral-naturality}.
Since \(q\) commutes with \(\mu'\), it still satisfies
\(\Phi(\mu)=\mu'\).
\end{proof}

 For the groups used later,
\cref{prop:early-homology} supplies the hypotheses
\begin{equation}
\label{eq:homology-normalization-hyp}
 H_1(G;\mathbb Z)=\mathbb Z\langle[\mu]\rangle,
 \qquad \rho(\pi)\subset[G,G].
\end{equation}

\begin{corollary}[Normalization of the push-off homomorphism]
\label{cor:rooted-normalization}
Let \(\rho:\pi\to G\) be a homomorphism and let \(\mu\in G\)
commute with \(\rho(\pi)\).  Suppose that
\eqref{eq:homology-normalization-hyp} holds.  Let
\(\Phi\in\Aut(G)\) fix \(\mu\), and suppose, for one fixed
representative \(f_*\) of \(f\in\Mod(\Sigma_g)\), that
\[
 \Phi(\rho(\gamma))
 =q\,\rho(f_*(\gamma))\mu^{\theta(\gamma)}q^{-1}
 \qquad(\gamma\in\pi),
\]
where \(q\in\langle\rho(\pi),\mu\rangle\) and
\(\theta\in\operatorname{Hom}(\pi,\mathbb Z)\).
Then \(\theta=0\).  After composing \(\Phi\) with an inner
automorphism represented by an element of \(\rho(\pi)\), one has
\begin{equation}
\label{eq:normalized-peripheral}
 \Phi\circ\rho=\rho\circ f_*,\qquad \Phi(\mu)=\mu,
\end{equation}
with the same \(f_*\). 
\end{corollary}

\begin{proof}
Abelianize the naturality equation in
\(H_1(G;\mathbb Z)=G/[G,G]\).  Conjugation disappears, and both
\(\rho(f_*(\gamma))\) and \(\Phi(\rho(\gamma))\) have zero class because
\(\rho(\pi)\subset[G,G]\) and automorphisms preserve \([G,G]\).  Hence
\[
0=\theta(\gamma)[\mu]
\qquad
(\gamma\in\pi).
\]
Since \([\mu]\) generates an infinite cyclic group, \(\theta=0\).

Because \(\mu\) centralizes \(\rho(\pi)\), write
\(q=s\mu^n\) for some \(s\in\rho(\pi)\) and \(n\in\mathbb Z\).
The equation then is reduced to
\[
 \Phi(\rho(\gamma))=s\,\rho(f_*(\gamma))s^{-1}.
\]
Thus \(\Phi'(x)=s^{-1}\Phi(x)s\) satisfies
\(\Phi'\circ\rho=\rho\circ f_*\).
Also \(\Phi'(\mu)=\mu\), since \(s\) commutes with \(\mu\).
Replacing \(\Phi\) by \(\Phi'\) proves
\eqref{eq:normalized-peripheral}.
\end{proof}

Under these homological hypotheses, an extendable mapping class gives
an exterior automorphism satisfying \eqref{eq:normalized-peripheral}.
The next section computes
the push-off homomorphism under one rim surgery.

\section{The exterior group and push-off homomorphism for one rim surgery}
\label{sec:rimlocal}

\subsection{The local model and crossing paths}

Let \(F\subset S^4\) be an oriented parametrized surface with positive
normal-circle section \(\sigma\), basepoint \(*=\sigma(p)\), exterior
group \(G=\pi_1(E_F,*)\), and push-off homomorphism \(\rho_F\).
Let \(d\subset\Sigma_g\) be an oriented essential simple closed curve
disjoint from \(p\).  Choose \(q\in d\) and a path
\(\lambda_d:p\to q\).  Also write \(d\in\pi\) for the resulting
based element, and put \(h=\rho_F(d)\).

Choose a product neighborhood of \(d\)
\[
(N,N\cap F)
\cong(S^1_d\times B^3,S^1_d\times I_0),
\qquad
B^3=[-1,1]_u\times D^2_{v,w},
\]
with \(N\) disjoint from \(*\).
Here \(I_0=[-1,1]\times\{0\}\), and
\((\partial_\theta,\partial_u)\) gives the surface orientation.
Choose the coordinates \((v,w)\) so that
\((\partial_\theta,\partial_u,\partial_v,\partial_w)\)
gives the orientation of \(S^4\) and \(\sigma\) points
in the positive \(v\)-direction along \(S^1_d\times I_0\).
Thus \((\partial_v,\partial_w)\) gives the positive normal
orientation, and \((\partial_u,\partial_v,\partial_w)\)
orients \(B^3\).
A crossing in the direction of increasing \(u\) has
sign \(+1\), and the reverse crossing has sign \(-1\).

For an oriented knot \(J\subset S^3\), choose a long-knot
presentation
\[
J^\circ:I_0\longrightarrow B^3
\]
by deleting a trivial open subarc of \(J\) and identifying
the ball containing the remaining arc orientation-preservingly
with \(B^3\).
The arc \(J^\circ\) agrees with \(I_0\) near its endpoints
and is oriented from \(u=-1\) to \(u=1\).
Let \(a\subset\overline{S^3\setminus B^3}\) be the
trivial closing arc, oriented from \(u=1\) to \(u=-1\). Here trivial means boundary-parallel.
Then \(J=J^\circ\cup a\), and \(I_0\cup a\) is an unknot.

Fix a framing along \(a\) that agrees with the positive
\(v\)-direction near its endpoints and completes the
positive \(v\)-framing of \(I_0\) to a zero framing
of \(I_0\cup a\).

The oridinary untwisted rim surgery replaces \(S^1_d\times I_0\) by \(S^1_d\times J^\circ\),
keeping the circle and arc parameters
\cite[\S2]{FintushelStern}.
Denote the resulting surface by \(F'\).
The identity outside \(N\) and the product parametrization
inside it parametrize \(F'\) by the same \(\Sigma_g\).

Retain \(\sigma\) outside \(N\).
Along \(J^\circ\), we choose a nowhere-zero normal vector field
that points in the positive \(v\)-direction near its endpoints.
Choose the field so that its parallel of \(J^\circ\),
completed using the fixed framing of \(a\), has
linking number zero with \(J\).
Such a field exists: full twists away from the endpoints
change the framing number by one.
Use the same field in every \(B^3\)-slice.
After normalization, these vector fields and the unchanged
section outside \(N\) define a normal-circle section
\(\sigma':\Sigma_g\to\partial\nu F'\).

We next describe the change in the exterior group.
The surgery replaces \(S^1_d\times I_0\) by
\(S^1_d\times J^\circ\) inside \(N\), leaving the surface
unchanged outside \(N\). Taking the complement of a tubular
neighborhood of each annulus in \(N\), we obtain the old and
new local exterior pieces
\[
\begin{aligned}
V_0&=S^1_d\times
(B^3\setminus\operatorname{int}\nu I_0),\\
V_J&=S^1_d\times
(B^3\setminus\operatorname{int}\nu J^\circ).
\end{aligned}
\]
Since \(I_0\) and \(J^\circ\) agree near their endpoints,
choose their tubular neighborhoods to agree there.
Each tubular neighborhood meets \(\partial B^3\) in two disks.
Put
\[
A=\partial B^3\setminus
(\mathring D^2\sqcup\mathring D^2).
\]
The old and new local exterior pieces meet \(\partial N\)
in the same region
\[
M=S^1_d\times A\cong T^2\times I
\]
which we refer as the common attaching region.

\begin{figure}[htbp]
\centering
\includegraphics[width=\linewidth]{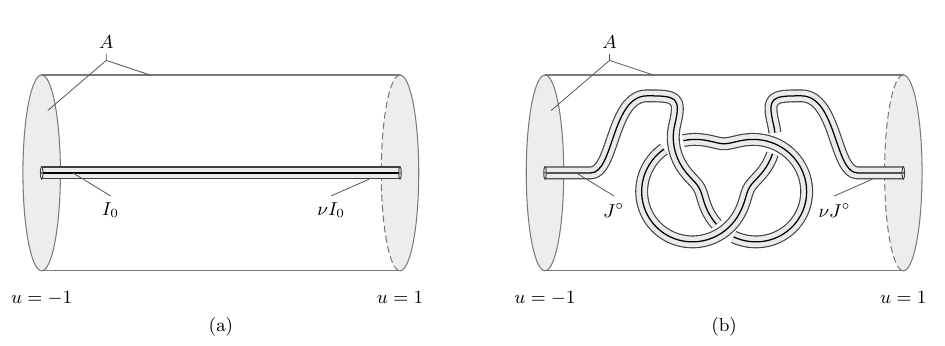}
\caption{The local exterior pieces before (a) and after (b)
rim surgery, with the \(S^1_d\)-factor suppressed.
The tubes indicate the removed tubular neighborhoods of
\(I_0\) and \(J^\circ\).
In both panels,
\(A=\partial B^3\setminus
(\mathring D^2\sqcup\mathring D^2)\)
is the common boundary annulus, and \(M=S^1_d\times A\).
The long trefoil in (b) is illustrative.}
\label{fig:local-rim-surgery}
\end{figure}

Put \(X=\overline{E_F\setminus V_0}\).  Then
\(E_F=X\cup_M V_0\) and \(E_{F'}=X\cup_M V_J\).
The inclusion \(M\hookrightarrow V_0\) induces an isomorphism on
fundamental groups: both groups are generated by the circle direction
and the positive meridian. Applying the Seifert--van Kampen theorem to
\(E_F=X\cup_M V_0\), we obtain an isomorphism
\[
\pi_1(X,*)\xrightarrow{\cong}\pi_1(E_F,*)=G
\]
induced by the inclusion \(X\hookrightarrow E_F\).
Under this isomorphism, the inclusion
\(X\hookrightarrow E_{F'}\) induces a homomorphism
\(G\to G'=\pi_1(E_{F'},*)\).

Choose \(*_{\mathrm{loc}}\in M\) on the positive
\(v\)-axis at \(u=0\), in the \(B^3\)-slice over \(q\).
Follow the push-off \(\sigma(\lambda_d)\) from \(*\)
to \(\sigma(q)\), then continue along the positive
\(v\)-axis to \(*_{\mathrm{loc}}\).
Since \(M\hookrightarrow V_0\) induces an isomorphism
on fundamental groups, we may homotope the portions
of this path in \(V_0\) into \(M\), keeping the
endpoints fixed. Denote the resulting path in \(X\) by
\(\omega:*\to *_{\mathrm{loc}}\).
Note that this choice depends on \(\lambda_d\) and is not canonical.

We use \(\omega\) to base loops in \(M\), \(V_0\),
and \(V_J\) at \(*\). The circle loop in \(M\),
running in the positive \(S^1_d\)-direction, is
homotopic in \(V_0\) to the push-off of \(d\).
With this choice, the \(S^1_d\)-circle and positive meridian
in \(M\) represent \(h=\rho_F(d)\) and \(\mu\) in \(G\),
respectively.

The \(B^3\)-slice of \(V_J\) containing
\(*_{\mathrm{loc}}\) is the exterior of \(J^\circ\)
in \(B^3\), with torus boundary.
For a knot \(L\subset S^3\), write
\(E(L)=S^3\setminus\operatorname{int}\nu L\).
Completing \(J^\circ\) with the fixed closing arc \(a\)
induces an isomorphism
\[
\pi_1(B^3\setminus\operatorname{int}\nu J^\circ)
\cong K_J:=\pi_1(E(J)).
\]
Choose loops \(m_{\mathrm{loc}}\) and
\(\ell_{\mathrm{loc}}\) on this boundary torus,
based at \(*_{\mathrm{loc}}\), representing the
positive meridian \(m_J\) and preferred longitude
\(\ell_J\) of \(J\), respectively.
Using \(\omega\) to base these loops at \(*\), we have
\[
\mu=[\omega m_{\mathrm{loc}}\omega^{-1}],
\qquad
\ell=[\omega\ell_{\mathrm{loc}}\omega^{-1}]
\quad\text{in }G'.
\]
Below, equalities between paths mean homotopies
relative to endpoints.

\begin{lemma}[Change in a crossing path]
\label{lem:local-crossing-coordinates}
Fix \(\theta\in S^1_d\) and a direction across the
surgery annulus.
Let \(\kappa'\) be the normal push-off of the new
crossing arc in \(\{\theta\}\times J^\circ\).
Its crossing sign \(\epsilon\) is \(+1\) if the arc
parameter runs from \(-1\) to \(1\), and \(-1\)
in the reverse direction.

Move the corresponding old normal push-off radially
into \(M\), keeping its endpoints fixed, and denote
the resulting path by \(\kappa\).
It has the same endpoints as \(\kappa'\).
Choose a path \(b\subset M\) from the endpoint of
\(\kappa\) to \(*_{\mathrm{loc}}\).
Then, in \(V_J\),
\begin{equation}
\label{eq:local-crossing-connectors}
\kappa'=\kappa b\ell_{\mathrm{loc}}^\epsilon b^{-1}.
\end{equation}
\end{lemma}

\begin{proof}
First suppose \(\epsilon=+1\) and take \(\theta\) to be
the circle coordinate of \(q\).
Choose \(b\subset A\) from the endpoint of \(\kappa\)
to \(*_{\mathrm{loc}}\).

Recall the trivial closing arc \(a\), with its fixed
framing, and put
\[
W=\overline{S^3\setminus B^3}
\setminus\operatorname{int}\nu a.
\]
Since \(a\) is trivial, the inclusion
\(A\hookrightarrow W\) induces an isomorphism
on fundamental groups.
Represent the framed push-off of \(a\) by a path
\(a^+\subset A\), running from the endpoint of
\(\kappa\) back to its starting point.
This is possible because \(A\hookrightarrow W\)
induces an isomorphism on fundamental groups.

The inclusion
\[
A\hookrightarrow
B^3\setminus\operatorname{int}\nu I_0
\]
also induces an isomorphism on fundamental groups.
By van Kampen, so does
\[
A\hookrightarrow E(I_0\cup a).
\]
The loop \(a^+\kappa\) represents the zero-framed
parallel of the unknot \(I_0\cup a\).
It is therefore null-homotopic in its exterior.
The preceding isomorphism gives
\[
a^+\simeq\kappa^{-1}
\quad\text{in }A,\text{ relative to endpoints}.
\]

The framing of \(J^\circ\), together with the same
framing of \(a\), gives a zero-framed parallel of \(J\).
Thus \(a^+\kappa'\) represents the preferred longitude
of \(J\).
The knot exterior decomposes as
\[
E(J)=
\bigl(B^3\setminus\operatorname{int}\nu J^\circ\bigr)
\cup_A W.
\]
Since \(A\hookrightarrow W\) induces an isomorphism
on fundamental groups, van Kampen shows that the inclusion
\[
B^3\setminus\operatorname{int}\nu J^\circ
\hookrightarrow E(J)
\]
also induces an isomorphism on fundamental groups.
Hence
\[
[\kappa^{-1}\kappa']
=[a^+\kappa']
=[b\ell_{\mathrm{loc}}b^{-1}]
\]
in the fundamental group of the \(B^3\)-slice over \(q\),
based at the endpoint of \(\kappa\).
This proves the formula for a positive crossing
in the chosen slice.

For a negative crossing, reverse both crossing paths.
Taking inverses of the positive-crossing formula gives
\[
(\kappa')^{-1}
=b\ell_{\mathrm{loc}}^{-1}b^{-1}\kappa^{-1}
=\kappa^{-1}(\kappa b)\ell_{\mathrm{loc}}^{-1}
(\kappa b)^{-1}.
\]
Here the second equality follows by cancelling
\(\kappa^{-1}\kappa\).
Since \(\kappa b\subset M\) runs from the endpoint of
\(\kappa^{-1}\) to \(*_{\mathrm{loc}}\), this proves
the formula for \(\epsilon=-1\).

Now take any \(\theta\in S^1_d\).
The same calculation applies in its \(B^3\)-slice.
First choose \(b\) by following a path in \(A\)
in that slice and then a path in the \(S^1_d\)-direction
to \(*_{\mathrm{loc}}\).
The product structure of \(V_J\) identifies the
longitude in this slice with \(\ell_{\mathrm{loc}}\).
Thus the formula holds for this choice of \(b\).

Finally, let \(b,b'\subset M\) be two paths with
the same endpoints. Then \(b^{-1}b'\) is a loop
in \(M\) based at \(*_{\mathrm{loc}}\).
The image of \(\pi_1(M,*_{\mathrm{loc}})\) in
\(\pi_1(V_J,*_{\mathrm{loc}})\) is generated by
the \(S^1_d\)-circle and \(m_{\mathrm{loc}}\).
The circle factor commutes with
\(\ell_{\mathrm{loc}}\) by the product structure,
and \(m_{\mathrm{loc}}\) commutes with
\(\ell_{\mathrm{loc}}\) on the boundary torus.
Changing \(b\) therefore does not change
\(b\ell_{\mathrm{loc}}^\epsilon b^{-1}\).
\end{proof}

The longitude-insertion formula was established in
\cite[Proposition~3.2]{NiuSingle} in the single-surgery, standard-unknotted setting. The formulation above records the basing and crossing orientation
needed for later surgeries.

\subsection{The exterior group and the canonical retraction}

For a subgroup \(L\le Q\), write \(\normal{L}_Q\) for its
normal closure in \(Q\). We omit the subscript when the
ambient group is clear.

For injective homomorphisms
\[
\iota_A:C\longrightarrow A,
\qquad
\iota_B:C\longrightarrow B,
\]
write \(A*_C B\) for their group pushout. It is the quotient
of \(A*B\) obtained by imposing
\[
\iota_A(c)=\iota_B(c)
\qquad(c\in C).
\]
We call \(A,B\) the vertex groups and \(C\) the edge group.
Any homomorphisms from \(A\) and \(B\) to a common group
that agree on \(C\) extend uniquely to \(A*_C B\).

We use the notation of the preceding subsection.
Thus \(G=\pi_1(E_F,*)\) and \(G'=\pi_1(E_{F'},*)\),
with push-off homomorphisms \(\rho_F\) and \(\rho_{F'}\).
The element \(\mu\in G\) is the positive surface meridian,
and \(h=\rho_F(d)\).
The knot group is \(K_J=\pi_1(E(J))\), with positive
meridian \(m_J\).
Also, \(\pi_1(V_J)\cong K_J\times\langle s\rangle\),
where \(s\) represents the \(S^1_d\)-circle.

In the amalgam
\[
G*_{\langle\mu,h\rangle}
\bigl(K_J\times\langle s\rangle\bigr)
\]
below, the hypothesis \(\langle\mu,h\rangle\cong\mathbb Z^2\)
allows us to take the abstract edge group to be
\(\mathbb Z^2\) with ordered basis \((\bar\mu,\bar h)\).
The edge maps send this basis to \((\mu,h)\) in \(G\)
and to \((m_J,s)\) in \(K_J\times\langle s\rangle\).
Thus \(\mu=m_J\) and \(h=s\) are identifications
in the amalgam, not in the separate vertex groups.

The amalgamated-product calculation in the following theorem is due to \cite[\S1]{Livingston1988}. Here we also specifies extra the canonical
retraction and its compatibility with the
push-off homomorphisms.

\begin{theorem}[Exterior group and retraction for one rim surgery]
\label{thm:one-step}
Suppose that \(\langle\mu,h\rangle\cong\mathbb Z^2\).  Then
\begin{equation}
\label{eq:one-step-amalgam}
 G'\cong G*_{\langle\mu,h\rangle}
              \bigl(K_J\times\langle s\rangle\bigr),
\end{equation}
where the edge maps identify \(\mu\) with \(m_J\) and \(h\) with
\(s\).  Equivalently,
\begin{equation}
\label{eq:one-step-quotient}
 G'\cong
 \frac{G*K_J}
 {\normal{\,m_J\mu^{-1},\ [h,k]\;:\;k\in K_J\,}_{G*K_J}}.
\end{equation}
There is a canonical retraction \(r:G'\onto G\) satisfying
\begin{equation}
\label{eq:retraction-properties}
 r\circ\rho_{F'}=\rho_F,
 \qquad
 \ker r=\normal{[K_J,K_J]}_{G'}.
\end{equation}
In particular, injectivity of \(\rho_F\) implies injectivity of
\(\rho_{F'}\).
\end{theorem}

\begin{proof}
By van Kampen applied to \(E_{F'}=X\cup_M V_J\),
the group \(G'\) is the pushout of
\[
G\longleftarrow\pi_1(M,*_{\mathrm{loc}})
\longrightarrow K_J\times\langle s\rangle,
\]
using the basing path \(\omega\).
The two maps send the circle and meridian generators
of \(\pi_1(M)\) to \((h,\mu)\) and \((s,m_J)\),
respectively. The first map is injective because
\(\langle\mu,h\rangle\cong\mathbb Z^2\).
The second is injective because \(m_J\) has infinite
order and \(s\) generates an independent
infinite cyclic factor. This proves
\eqref{eq:one-step-amalgam}.

The product structure of \(V_J\) gives
\([s,k]=1\) for every \(k\in K_J\).
Since the amalgamation identifies \(s=h\) and
\(m_J=\mu\), we obtain the presentation
\eqref{eq:one-step-quotient}.

Since \(H_1(K_J;\mathbb Z)\cong\mathbb Z\) is generated
by the positive meridian \(m_J\), abelianization gives
a homomorphism
\[
\phi_J:K_J\longrightarrow\mathbb Z,
\qquad \phi_J(m_J)=1.
\]
In particular, \(\phi_J(\ell_J)=0\), since the
preferred longitude is null-homologous.
Define \(r\) on the vertex groups by
\[
 r|_G=\operatorname{id}_G,
 \qquad
 r(k)=\mu^{\phi_J(k)}\quad(k\in K_J),
 \qquad
 r(s)=h.
\]
Since \([\mu,h]=1\), the elements \(h\) and
\(\mu^{\phi_J(k)}\) commute for every \(k\in K_J\).
Thus the assignments define a homomorphism
\(K_J\times\langle s\rangle\to G\).
This homomorphism agrees with the identity on \(G\)
along the edge group, since
\[
r(m_J)=\mu,\qquad r(s)=h.
\]
By the universal property of the amalgam,
these maps define a homomorphism \(r:G'\to G\).
Since \(r|_G=\operatorname{id}_G\), it is a retraction.

The normal closure of \([K_J,K_J]\) lies in \(\ker r\).
After quotienting by it, the new vertex group is
\(\langle m_J,s\rangle\cong\mathbb Z^2\), and its edge map is an
isomorphism.  Since the edge map now identifies the entire new vertex
group with \(\langle\mu,h\rangle\le G\), the quotient is
isomorphic to \(G\).
Under this identification, the quotient map fixes \(G\),
sends \(s\) to \(h\), and sends each \(k\in K_J\) to
\(\mu^{\phi_J(k)}\).
It therefore agrees with \(r\), proving
\(
\ker r=\normal{[K_J,K_J]}_{G'}
\)
(Compare \cite[\S2.3]{FintushelStern}
and \cite[proof of Theorem~8.1]{Cochran2004}).

Let \(\gamma\in\pi\), represented by a based surface loop.
Divide it into arcs outside the surgery annulus
and arcs crossing it. The push-offs of the outside
arcs are unchanged by the surgery. At each crossing, \cref{lem:local-crossing-coordinates}
gives
\[
\kappa'=\kappa b\ell_{\mathrm{loc}}^\epsilon b^{-1},
\]
where \(\kappa\subset M\) represents the old pushed-off
crossing arc. Thus the new push-off of \(\gamma\)
is obtained from the old one by inserting conjugates
of the preferred longitude.

Since \(r\) fixes \(G\) and
\(r(\ell)=\mu^{\phi_J(\ell_J)}=1\), applying \(r\)
removes these insertions. Therefore
\[
r\circ\rho_{F'}=\rho_F.
\]
In particular, injectivity of \(\rho_F\) implies
injectivity of \(\rho_{F'}\).
\end{proof}

\subsection{The change in the push-off homomorphism}

The preferred longitude \(\ell_J\) lies in \([K_J,K_J]\).
Thus every conjugate of its image in \(G'\) has zero homology
class.  By \cref{lem:local-crossing-coordinates}, the new
push-off differs from the old one by insertions of these
longitudes.  Hence, for every \(\gamma\in\pi\),
\[
[\rho_{F'}(\gamma)]=[\rho_F(\gamma)]
\quad\text{in }H_1(G';\mathbb Z).
\]
The right-hand side uses the homomorphism \(G\to G'\)
induced by the unchanged part \(X\).

\begin{figure}[htbp]
\centering
\includegraphics[width=\linewidth]{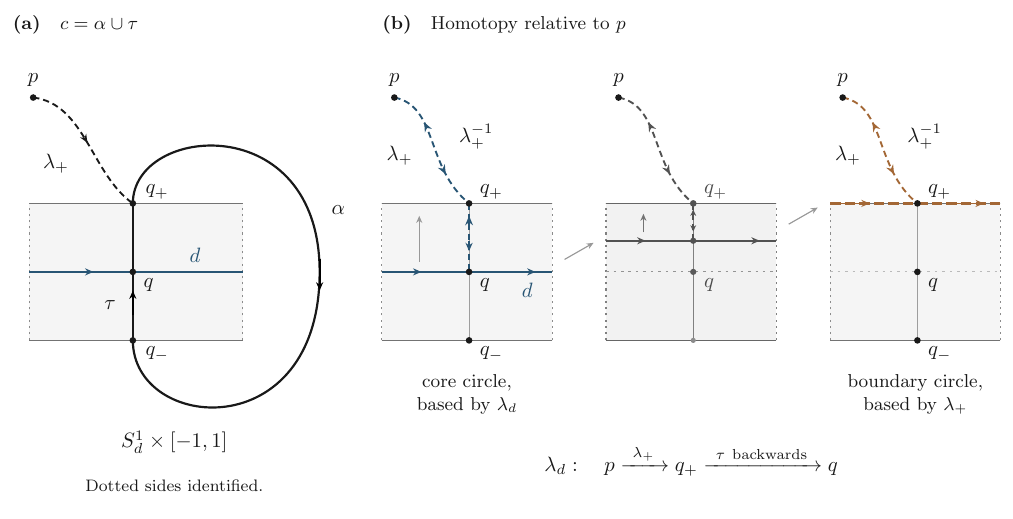}
\caption{Basing the surgery curve.
(a) The curve \(c=\alpha\cup\tau\) and the path
\(\lambda_+\). The dotted vertical sides are identified.
(b) Sliding the core circle toward \(q_+\) shortens
the two traversals of the subarc of \(\tau\) from
\(q_+\) to \(q\). This gives a homotopy relative to
\(p\) between the core circle based by \(\lambda_d\)
and the boundary circle based by \(\lambda_+\).}
\label{fig:one-crossing-basing}
\end{figure}

\begin{corollary}[A curve meeting the surgery curve once]
\label{cor:one-crossing-push-off}
Under the hypotheses of \cref{thm:one-step}, let
\(c\subset\Sigma_g\) be an oriented simple closed curve
with geometric intersection number \(i(c,d)=1\).
We may choose basing paths for \(c\) and \(d\), together
with a path \(\omega:*\to *_{\mathrm{loc}}\) in \(X\)
for the based inclusion of \(K_J\) into \(G'\), so that
the resulting elements \(c,d\in\pi\) satisfy
\[
\rho_{F'}(c)=\rho_F(c)\ell^\epsilon,
\qquad
\rho_{F'}(d)=\rho_F(d)=h.
\]
For these choices, \(h=\rho_F(d)\) and
\(\ell=[\omega\ell_{\mathrm{loc}}\omega^{-1}]\).
The sign \(\epsilon\) is \(+1\) for a crossing in the
direction of increasing \(u\), and \(-1\) in the reverse
direction.  The same \(\omega\) is used in
\eqref{eq:one-step-amalgam}, and \(G\) is identified
with its vertex subgroup in \(G'\).
The based representative of \(d\) may be taken to be an
oriented boundary parallel, based by a path outside the
interior of the surgery annulus.
\end{corollary}

\begin{proof}
Choose a representative of \(c\), disjoint from \(p\),
whose intersection with the surgery annulus
\(S^1_d\times[-1,1]\) is the arc \(\{q\}\times[-1,1]\).
Let \(q_-\) and \(q_+\) be its initial and terminal
endpoints in the orientation of \(c\).
Write \(\tau:q_-\to q_+\) for this crossing arc and
\(\alpha:q_+\to q_-\) for the remaining arc of \(c\).

Since \(i(c,d)=1\), the curve \(d\) is nonseparating.
Choose a path \(\lambda_+:p\to q_+\) in the connected complement
\(\Sigma_g\setminus(S^1_d\times(-1,1))\), and put
\(\lambda_-=\lambda_+\alpha\).
Base \(c\) at \(q_+\) using \(\lambda_+\).  Then
\begin{equation}
\label{eq:based-one-crossing-curve}
 c=[\lambda_-\tau\lambda_+^{-1}]
  =[\lambda_+\alpha\tau\lambda_+^{-1}].
\end{equation}

Choose \(\lambda_d\) to follow \(\lambda_+\) and then
the subarc of \(\tau\) from \(q_+\) back to \(q\).
By definition, the based element \(d\in\pi\) is
represented by following \(\lambda_d\), traveling once
around the oriented curve \(d\), and returning along
\(\lambda_d^{-1}\).
The product structure of the surgery annulus allows us
to move this loop to the positively oriented boundary
circle through \(q_+\), based using \(\lambda_+\).
This representative lies outside the interior of the
surgery annulus, where both the surface and the section
are unchanged. The two based representatives are shown in
\cref{fig:one-crossing-basing}. Therefore
\(\rho_{F'}(d)=\rho_F(d)=h\).

Write \(\widehat\lambda_+=\sigma\circ\lambda_+\) and
\(\widehat\alpha=\sigma\circ\alpha\).
These paths lie in \(X\) and are unchanged by surgery.
Let \(\kappa'\) be the \(\sigma'\)-push-off of \(\tau\).
Move the \(\sigma\)-push-off of \(\tau\) radially into
\(M\), keeping its endpoints fixed, and denote the
resulting path by \(\kappa\), as in
\cref{lem:local-crossing-coordinates}.
Both crossing paths end at \(\sigma(q_+)\).
For our choice of \(\lambda_d\), we may take
\(\omega=\widehat\lambda_+b\), where \(b\subset M\)
runs from \(\sigma(q_+)\) to \(*_{\mathrm{loc}}\).
This is a choice of \(\omega\) as constructed
in Section~3.1.
Since \(\pi_1(M)\) is abelian, using \(b\) to base its
positive circle and meridian at \(\sigma(q_+)\) gives
the corresponding loops there.  Basing them at \(*\)
using \(\widehat\lambda_+\) gives \(h\) and \(\mu\)
in \(G\).  We use this same \(\omega\) to base the knot-group
factor in \eqref{eq:one-step-amalgam}, so the edge
identifications are \(s=h\) and \(m_J=\mu\).

By \cref{lem:local-crossing-coordinates},
\(\kappa'=\kappa b\ell_{\mathrm{loc}}^\epsilon b^{-1}\).
Consequently, in \(G'\),
\[
\begin{aligned}
\rho_{F'}(c)
&=[\widehat\lambda_+\widehat\alpha\kappa'
   \widehat\lambda_+^{-1}]\\
&=[\widehat\lambda_+\widehat\alpha\kappa
   b\ell_{\mathrm{loc}}^\epsilon b^{-1}
   \widehat\lambda_+^{-1}]\\
&=[\widehat\lambda_+\widehat\alpha\kappa
   \widehat\lambda_+^{-1}]
  [\omega\ell_{\mathrm{loc}}^\epsilon\omega^{-1}]\\
&=\rho_F(c)\ell^\epsilon.
\end{aligned}
\]
\end{proof}
\section{Preliminary rim surgeries along meridian-disk boundaries}
\label{sec:disk}

Let
\[
P=\Sigma_{r,1}
\quad\text{if }g=2r,
\qquad
P=\Sigma_{r,2}
\quad\text{if }g=2r+1.
\]
Put \(H=P\times I\), so \(\partial H\cong\Sigma_g\), and embed it as a
standard handlebody in an equatorial \(S^3\subset S^4\).  The standard
unknotted genus-\(g\) surface then can be viewed as
\[
\Sigma_g^0=\partial(P\times I)\subset S^3\subset S^4.
\]
It is decomposed as two pieces
\[
P^+=P\times\{1\},
\qquad
P^-=P\times\{0\},
\]
joined by \(\partial P\times I\).

Fix the parametrization \(\Sigma_g\cong\Sigma_g^0\).  Choose
\(q_0\in\partial P\) and set
\(
p=(q_0,1)\in P^+
\)
and \(p_-=(q_0,0)\). Let
\(
\lambda_-:p\longrightarrow p_-
\)
be the path \(q_0\times I\), directed from \(P^+\) to \(P^-\).  
Let \(B_+^4\) and \(B_-^4\) be the two \(4\)-balls
of \(S^4\) bounded by the equatorial \(S^3\).
Choose a unit normal vector field to
\(S^3\subset S^4\) that points into \(B_+^4\). Its restriction
to \(\Sigma_g^0\) is also a normal vector field.
Pushing each point of \(\Sigma_g^0\) along this field defines a normal-circle
section \(\sigma\). Set \(*=\sigma(p)\).

Since each \(B_\pm^4\setminus\Sigma_g^0\) is
contractible, and their intersection
\(S^3\setminus\Sigma_g^0\) has two connected
components, van Kampen gives
\[
\pi_1(E_{\Sigma_g^0},*)
=\langle\mu\rangle\cong\mathbb Z,
\]
where
\(E_{\Sigma_g^0}
=S^4\setminus\operatorname{int}\nu\Sigma_g^0\)
and \(\mu\) is the positive meridian.

The pushed-off surface \(\sigma(\Sigma_g^0)\)
lies in the interior of \(B_+^4\), which is
contractible and disjoint from \(\Sigma_g^0\).
Thus the initial push-off homomorphism is trivial.

In the following, we first perform rim surgeries
along the boundaries of a complete meridian-disk
system of \(P\times I\). These surgeries will help us make the
push-off maps on \(P^+\) and \(P^-\) injective,
with the same
image \(U_0\) in the resulting exterior group. Further surgeries along
\(\partial P\times\{1/2\}\) in \cref{sec:seed}
will separate the two images.

Choose \(g=2r+b-1\) pairwise disjoint, oriented,
properly embedded arcs
\[
\alpha_1,\ldots,\alpha_g\subset P,
\]
with distinct endpoints away from \(q_0\), such that
cutting \(P\) along them gives a disk. Such arcs exist
because \(P\) admits a handle decomposition with one
\(0\)-handle and \(g\) \(1\)-handles. Their complement is a disk
because
\(
2r+b-1=g.
\)
For \(g=3\), \cref{fig:complete-arc-system-g3}
illustrates such a system of arcs on \(P=\Sigma_{1,2}\).

\begin{figure}[htbp]
\centering
\begin{tikzpicture}[
  x=0.88cm,
  y=0.88cm,
  line cap=round,
  line join=round,
  every node/.style={font=\small}
]

  \draw[line width=0.8pt] (0,0) rectangle (6,6);

  \draw[-{Latex[length=1.8mm,width=1.2mm]},line width=0.7pt]
    (0.14,4.45) -- (0.14,5.15);
  \draw[-{Latex[length=1.8mm,width=1.2mm]},line width=0.7pt]
    (5.86,4.45) -- (5.86,5.15);
  \draw[-{Latex[length=1.8mm,width=1.2mm]},line width=0.7pt]
    (4.15,0.14) -- (4.75,0.14);
  \draw[-{Latex[length=1.8mm,width=1.2mm]},line width=0.7pt]
    (4.92,0.14) -- (5.52,0.14);
  \draw[-{Latex[length=1.8mm,width=1.2mm]},line width=0.7pt]
    (4.15,5.86) -- (4.75,5.86);
  \draw[-{Latex[length=1.8mm,width=1.2mm]},line width=0.7pt]
    (4.92,5.86) -- (5.52,5.86);

  \coordinate (BOne) at (2,3);
  \coordinate (BTwo) at (4,3);
  \draw[line width=0.8pt] (BOne) circle[radius=0.48];
  \draw[line width=0.8pt] (BTwo) circle[radius=0.48];

  \draw[line width=1.0pt] (2.48,3) -- (3.52,3);                         % alpha_1
  \draw[line width=1.0pt] (1.52,3) -- (0,3);                           % alpha_2 left
  \draw[line width=1.0pt] (6,3) -- (4.48,3);                           % alpha_2 right
  \draw[line width=1.0pt] (2,3.48) -- (2,6);                           % alpha_3 top
  \draw[line width=1.0pt]
    (2,0) .. controls (2.25,0.82) and (3.55,1.20) .. (4,2.52);         % alpha_3 bottom

  \node at (3,3.34) {$\alpha_1$};
  \node at (0.78,3.34) {$\alpha_2$};
  \node at (3.10,0.72) {$\alpha_3$};

  \node[anchor=east] at (1.28,4.02) {$\partial P$};
  \draw[line width=0.45pt] (1.34,3.90) -- (1.67,3.36);
  \node[anchor=west] at (4.72,4.02) {$\partial P$};
  \draw[line width=0.45pt] (4.66,3.90) -- (4.33,3.36);

  \fill (1.66,2.66) circle[radius=1.15pt];
  \node[anchor=east] at (1.18,2.23) {$q_0$};
  \draw[line width=0.45pt] (1.24,2.30) -- (1.60,2.60);
\end{tikzpicture}
\caption{A complete arc system on \(P=\Sigma_{1,2}\).  The broken
pieces of \(\alpha_2\) and \(\alpha_3\) continue across the indicated
side identifications, and the quotient complement is a disk.}
\label{fig:complete-arc-system-g3}
\end{figure}

The product disks
\[
\Delta_j=\alpha_j\times I\subset H
\]
form a complete meridian-disk system, since cutting along them gives a
\(3\)-ball.  Give each its product orientation and set
\[
c_j=\partial\Delta_j\subset\Sigma_g^0
\]
with boundary orientation. \cref{fig:product-disk-boundary} shows its
four pieces.

\begin{figure}[htbp]
\centering
\begin{tikzpicture}[
  x=0.86cm,
  y=0.86cm,
  line cap=round,
  line join=round,
  every node/.style={font=\small}
]

  \coordinate (Aminus) at (0,0);
  \coordinate (Bminus) at (6,0);
  \coordinate (Cminus) at (7.2,1.4);
  \coordinate (Dminus) at (1.2,1.4);
  \coordinate (Aplus) at (0,3.4);
  \coordinate (Bplus) at (6,3.4);
  \coordinate (Cplus) at (7.2,4.8);
  \coordinate (Dplus) at (1.2,4.8);

  \draw[line width=0.75pt] (Aminus) -- (Bminus) -- (Cminus);
  \draw[densely dashed,line width=0.65pt] (Cminus) -- (Dminus);
  \draw[line width=0.75pt] (Dminus) -- (Aminus);
  \draw[line width=0.75pt]
    (Aplus) -- (Bplus) -- (Cplus) -- (Dplus) -- cycle;

  \draw[line width=0.75pt] (Aminus) -- (Aplus);
  \draw[line width=0.75pt] (Bminus) -- (Bplus);
  \draw[densely dashed,line width=0.65pt] (Cminus) -- (Cplus);
  \draw[line width=0.75pt] (Dminus) -- (Dplus);

  \coordinate (Lminus) at (0.6,0.7);
  \coordinate (Rminus) at (6.6,0.7);
  \coordinate (Rplus) at (6.6,4.1);
  \coordinate (Lplus) at (0.6,4.1);
  \draw[line width=1.15pt]
    (Lminus) -- (Rminus) -- (Rplus) -- (Lplus) -- cycle;
  \draw[densely dotted,line width=0.55pt] (0.6,1.83) -- (6.6,1.83);
  \draw[densely dotted,line width=0.55pt] (0.6,2.97) -- (6.6,2.97);

  \node at (5.78,4.55) {$P^+$};
  \node at (1.45,0.23) {$P^-$};
  \node at (3.58,2.40) {$\Delta_j=\alpha_j\times I$};

  \node at (3.62,4.55) {$\alpha_j\times\{1\}$};
  \draw[line width=0.45pt] (3.62,4.40) -- (3.62,4.14);

  \node at (4.20,0.23) {$\alpha_j\times\{0\}$};
  \draw[line width=0.45pt] (4.20,0.37) -- (4.20,0.66);

  \node[anchor=east] at (0.25,4.62) {$c_j=\partial\Delta_j$};
  \draw[line width=0.45pt] (0.30,4.51) -- (0.60,4.14);

  \node[anchor=east] at (0.10,2.42) {$\partial\alpha_j\times I$};
  \draw[line width=0.45pt] (0.17,2.34) -- (0.55,2.34);

  \node[rotate=90] at (7.02,2.45) {$\partial P\times I$};

  \node at (3.60,-1.15)
    {$c_j=(\alpha_j\times\{1\})\cup(\partial\alpha_j\times I)
      \cup(\alpha_j\times\{0\})$};
\end{tikzpicture}
\caption{A local product model for \(c_j=\partial\Delta_j\).}
\label{fig:product-disk-boundary}
\end{figure}

The curves \(c_1,\ldots,c_g\) are essential and disjoint: cutting along
them gives a sphere with \(2g\) boundary components, so the complement is
connected.  For each \(j\), choose an oriented nontrivial
hyperbolic knot \(J_j^D\), where the superscript
\(D\) stands for the meridian-disk surgeries.
Perform ordinary untwisted rim surgery along each
\(c_j\) using \(J_j^D\), in disjoint neighborhoods
avoiding \(p\) and \(\lambda_-\). Choose the normal
section as in \cref{sec:rimlocal}, and identify
each knot meridian \(m_j^D\) with the positive
surface meridian \(\mu\). Continue to denote the
corresponding subsurfaces by \(P^+\) and \(P^-\).

\begin{proposition}[The exterior group after the preliminary rim surgeries]
\label{prop:disk-group}
After the \(g\) rim surgeries along \(c_1,\ldots,c_g\), the exterior
group is the amalgam
\begin{equation}
\label{eq:disk-amalgam}
A_0=
K_1^D*_{\langle\mu\rangle}
K_2^D*_{\langle\mu\rangle}\cdots
*_{\langle\mu\rangle}K_g^D,
\end{equation}
where the meridian of every knot-group factor is identified with
\(\mu\).

Set
\[
U_0=\langle\ell_1^D,\ldots,\ell_g^D\rangle\le A_0.
\]
Then \(U_0\) is a free group with free basis
\(\ell_1^D,\ldots,\ell_g^D\), and
\begin{equation}
\label{eq:disk-meridian-centralizer}
C_{A_0}(\mu)
=
\langle\mu\rangle\times U_0.
\end{equation}
Moreover,
\begin{equation}
\label{eq:F-characterization}
U_0
=
C_{A_0}(\mu)\cap[A_0,A_0].
\end{equation}
\end{proposition}

\begin{proof}
Consider the surgery along \(c_j\) with old exterior group \(G\).
Let \(V_j^0\) be the local annulus exterior and
\(X_j=\overline{E_{\mathrm{old}}\setminus V_j^0}\).
The attaching region \(M_j\cong T^2\times I\) has generators
\(\mu_M,c_M\), and \(M_j\hookrightarrow V_j^0\) induces an
isomorphism on fundamental groups.  Thus
\(\pi_1(X_j,*)\cong G\).
Choose a path \(\omega_j:*\to *_j\in M_j\) by pushing off a
surface path from \(p\) that avoids the other surgery annuli.
Such paths exist because the complement of the \(c_k\) is connected.
Use \(\omega_j\) for van Kampen and for the based inclusion of
\(K_j^D\).

If the old push-off of \(c_j\) is trivial, the old edge map sends
\(\mu_M\) to \(\mu\) and \(c_M\) to \(1\).
The new local group is \(K_j^D\times\langle s_j\rangle\), with
\(\mu_M\mapsto m_j^D\) and \(c_M\mapsto s_j\).
The old edge map is not injective.  Its van Kampen pushout is
\[
 \frac{G*(K_j^D\times\langle s_j\rangle)}
 {\normal{m_j^D\mu^{-1},s_j}}
 \cong G*_{\langle\mu\rangle}K_j^D.
\]
In particular, the circle \(s_j\) is killed.
Initially \(G=\langle\mu\rangle\) and every surface push-off is
trivial.  At each later stage the remaining \(c_k\), with their
basing paths, lie in the unchanged outside region, so their push-offs
remain trivial.  Both maps from \(\langle\mu\rangle\) are injective. 
Normal form of amalgam therefore embeds each old group and the new knot group, and repeat iteratively gives \eqref{eq:disk-amalgam}.  Abelianization identifies all
knot meridians, giving
\(H_1(A_0;\mathbb Z)=\mathbb Z\langle[\mu]\rangle\).

The peripheral subgroup of each hyperbolic knot group is malnormal
\cite[Corollary~2]{HarpeWeber}. If an element of \(K_j^D\)
centralizes \(\mu\), then its conjugate of the peripheral
subgroup intersects that subgroup nontrivially. Malnormality
forces the element to lie in the peripheral subgroup, which
is abelian. Hence
\[
C_{K_j^D}(\mu)=\langle\mu,\ell_j^D\rangle.
\]

We compute \(C_{A_0}(\mu)\) using reduced-word
normal forms in the successive cyclic amalgams. Put
\[
A^{(0)}=\langle\mu\rangle,\qquad
A^{(j)}=A^{(j-1)}*_{\langle\mu\rangle}K_j^D
\quad(1\le j\le g).
\]
Write an element in reduced form as
\(\mu^r k_1\cdots k_n\), where
\(k_i\in K_{j_i}^D\setminus\langle\mu\rangle\). If it centralizes \(\mu\), then
\(k_1^{-1}\mu k_1\in\langle\mu\rangle\). Otherwise
\[
 k_n^{-1}\cdots k_2^{-1}(k_1^{-1}\mu k_1)k_2\cdots k_n
\]
is a reduced word outside \(\langle\mu\rangle\), a contradiction.
Meridional abelianization forces \(k_1^{-1}\mu k_1=\mu\).
Cancel this conjugation and repeat to see that every \(k_i\) belongs
to \(\langle\mu,\ell_{j_i}^D\rangle\).
Conversely, these peripheral subgroups centralize \(\mu\).
Their own reduced words remain reduced in \(A_0\), so
\[
 C_{A_0}(\mu)
 =\langle\mu,\ell_1^D\rangle*_{\langle\mu\rangle}\cdots
   *_{\langle\mu\rangle}\langle\mu,\ell_g^D\rangle
 =\langle\mu\rangle\times
   \bigl(\langle\ell_1^D\rangle*\cdots*\langle\ell_g^D\rangle\bigr).
\]
This proves freeness of \(U_0\) and
\eqref{eq:disk-meridian-centralizer}.
Every longitude is null-homologous.  Thus \(\mu^n u\), with
\(u\in U_0\), has homology class \(n[\mu]\), which vanishes exactly
when \(n=0\).  This proves \eqref{eq:F-characterization}.
\end{proof}

Let
\[
\rho:\pi_1(\Sigma_g,p)\longrightarrow A_0
\]
be the resulting push-off homomorphism.  For each \(j\), choose a point in the interior of
\(\alpha_j\). Cutting \(P\) along the arcs produces a disk \(D_P\) with
two copies of this point on \(\partial D_P\).
Choose an embedded tree in \(D_P\) whose endpoints
are \(q_0\) and these \(2g\) points, and whose
interior is disjoint from \(\partial D_P\). Identify the two boundary copies of each
\(\alpha_j\) to recover \(P\). For each \(j\),
the two paths in the tree from \(q_0\) to the
copies of the chosen point on \(\alpha_j\)
form a based loop \(y_j\) in \(P\). These loops
form a free basis of \(\pi_1(P,q_0)\).

Taking their copies in
\(P^+=P\times\{1\}\) and
\(P^-=P\times\{0\}\), we obtain free bases
\[
y_j^+\in\pi_1(P^+,p),
\qquad
y_j^-\in\pi_1(P^-,p_-)
\quad(1\le j\le g).
\]
Orient \(y_j^+\) so that its single crossing of
\(\alpha_j\times\{1\}\) is positive, and orient \(y_j^-\) so the
based loop \(\lambda_-y_j^-\lambda_-^{-1}\) has the same positive sign
at \(\alpha_j\times\{0\}\). Each \(y_j^\pm\) meets \(c_j\) once and is disjoint
from \(c_k\) for \(k\ne j\).

\begin{proposition}[Push-offs of \(P^+\) and \(P^-\)]
\label{prop:horizontal-peripheral}
With the chosen basings,
\[
\rho(y_j^+)
=\rho(\lambda_-y_j^-\lambda_-^{-1})
=\ell_j^D
\qquad(1\le j\le g).
\]
Consequently, \(\rho\) restricts to an isomorphism
onto \(U_0\) on each of
\(\pi_1(P^+,p)\) and
\(\lambda_-\pi_1(P^-,p_-)\lambda_-^{-1}\).
\end{proposition}

\begin{proof}
For the \(j\)-th surgery, take \(\omega_j\) to be the path that follows the pushed-off
tree path to the exit of the positive crossing of \(y_j^+\), then a short
segment to \(*_j\).  
Use the same path \(\omega_j\) in the van Kampen calculation
and to base the knot-group factor \(K_j^D\) in the new
exterior. Thus a loop \(k\in K_j^D\), based at \(*_j\),
represents
\[
\omega_j k\omega_j^{-1}
\]
at the global basepoint \(*\). In particular,
\[
\ell_j^D
=\omega_j\ell_{j,\mathrm{loc}}^D\omega_j^{-1}.
\]
The crossing-path formula of
\cref{lem:local-crossing-coordinates} is a local identity
in the new exterior piece, so it remains valid even though
\(c_M\) maps to the identity in the old exterior group.

The old push-off of \(y_j^+\) is trivial, initially and through all
other disjoint surgeries.  Its unique crossing is positive and is with
\(c_j\), so its entire new based word is
\[
\rho(y_j^+)=\ell_j^D.
\]

The loop \(\lambda_-y_j^-\lambda_-^{-1}\) also has
trivial push-off before surgery and crosses \(c_j\)
once positively. Let \(\omega_j^+=\omega_j\), and
choose \(\omega_j^-\) using \(\lambda_-\), the corresponding
tree path in \(P^-\), and a final path in \(M_j\) to \(*_j\).

The upper and lower tree paths come from the same
tree in the disk obtained by cutting \(P\) along
the arcs \(\alpha_k\). Their product rectangles in
\(P\times I\) can be pushed into the positive normal
side of \(S^3\), away from the surgery neighborhoods.
The remaining difference between \(\omega_j^+\)
and \(\omega_j^-\) lies in \(M_j\). Thus, for some
\(a,b\in\mathbb Z\),
\[
\omega_j^+(\omega_j^-)^{-1}
=\omega_j^+c_M^a\mu_M^b(\omega_j^+)^{-1}.
\]
The circle generator \(c_M\) is trivial in the old
exterior, and its image \(s_j\) is also trivial after
surgery by \cref{prop:disk-group}. Hence the
difference is \(\mu^b\). Since \(\mu\) commutes with
the preferred longitude,
\[
\begin{aligned}
\rho(\lambda_-y_j^-\lambda_-^{-1})
&=\omega_j^-\ell_{j,\mathrm{loc}}^D
  (\omega_j^-)^{-1}\\
&=\omega_j^+\ell_{j,\mathrm{loc}}^D
  (\omega_j^+)^{-1}
=\ell_j^D.
\end{aligned}
\]
Both families \(y_j^\pm\) are free bases, and
the longitudes \(\ell_j^D\) freely generate \(U_0\)
by \cref{prop:disk-group}. Thus both horizontal
push-off maps are isomorphisms onto \(U_0\).
\end{proof}

\section{Injectivity after rim surgery along the vertical boundary}\label{sec:seed}

The surgeries in \cref{sec:disk} make the push-offs of
\(P^+\) and \(P^-\) injective, but both have the same
image \(U_0\). Thus the push-off of the closed surface
is not yet injective.

In this section, we perform further rim surgeries along
the core circles of the vertical annuli in
\(\partial P\times I\). There is one such circle in
even genus and two in odd genus. These surgeries insert
knot longitudes into the push-offs of paths joining
\(P^+\) to \(P^-\). The first surgery conjugates the
image of \(P^-\) while leaving the image of \(P^+\)
unchanged. In odd genus, the second surgery also changes
the image of the loop passing through both vertical
annuli.

We compute these changes and prove that the resulting
push-off homomorphism
\[
\rho_0:\pi_1(\Sigma_g,p)\longrightarrow G_0
\]
is injective. The proof uses the decomposition of the
surface group as an amalgam in even genus case and as an
amalgam followed by an HNN extension in odd genus case.

\subsection{Boundary components and based paths}

The annuli \(\partial P\times I\) join \(P^+\) to \(P^-\) in
\(\Sigma_g=\partial(P\times I)\). 
Write
\[
\partial P=
\begin{cases}
B_1, & g=2r,\\
B_1\sqcup B_2, & g=2r+1,
\end{cases}
\]
with \(q_0\in B_1\). To compare the push-offs of \(P^+\) and \(P^-\),
we use a path from \(p=(q_0,1)\) to \(p_-=(q_0,0)\) through each
annulus. For the first annulus, set \(\Lambda_1=\lambda_-\). The rim surgeries we perform in this section
are along their core circles \(B_\nu\times\{1/2\}\).

In odd genus, choose \(q_2\in B_2\) away from the endpoints of the
arcs \(\alpha_j\), and an embedded path \(\beta_2:q_0\to q_2\)
that avoids these arcs and meets \(\partial P\) only at its endpoints.
Such a path exists because cutting \(P\) along the \(\alpha_j\)
gives a disk. Let \(\beta_2^\pm\) be its copies in \(P^\pm\),
and let \(\lambda_2=q_2\times I\) run from \((q_2,1)\) to \((q_2,0)\).
The path through the second annulus is
\begin{equation}
\label{eq:vertical-path-Lambda}
\Lambda_2=\beta_2^+\lambda_2(\beta_2^-)^{-1}.
\end{equation}
Each \(\Lambda_\nu\) runs from \(p\) to \(p_-\), crossing
\(B_\nu\times\{1/2\}\) once and avoiding the other core circles.
Choose disjoint product neighborhoods of the
surgery curves $B_\nu\times\{1/2\}$, disjoint from
$P^+$ and $P^-$.
Each $\Lambda_\nu$ meets only the corresponding
neighborhood, along a single subarc of $\lambda_\nu$.

In odd genus, going down the second annulus and returning through
the first gives the loop
\begin{equation}
\label{eq:odd-surface-cycle}
\tau
=
\Lambda_2\Lambda_1^{-1}
=
\beta_2^+\lambda_2(\beta_2^-)^{-1}\lambda_-^{-1}.
\end{equation}
The choice of \(\beta_2\) makes \(\tau\) disjoint from every \(c_j\).
Its initial push-off is trivial by \cref{sec:disk}, and the
preliminary surgeries we performed in \cref{sec:disk} leave it unchanged because it misses their
neighborhoods. Thus
\begin{equation}
\label{eq:old-odd-cycle-trivial}
\rho(\tau)=1
\qquad\text{in }A_0.
\end{equation}

\begin{figure}[htbp]
\centering
\includegraphics[width=0.78\linewidth]{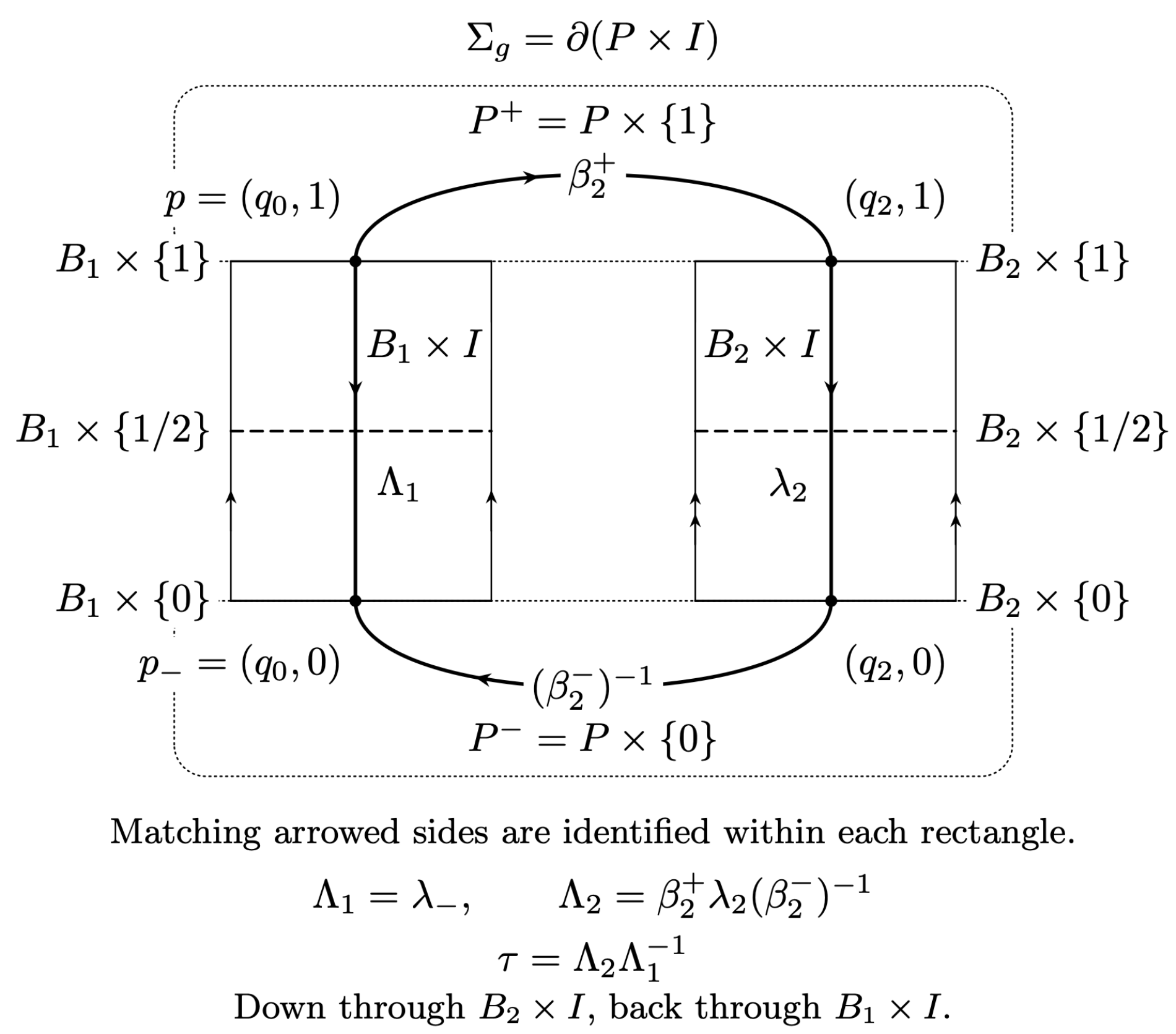}
\caption{The decomposition of
$\Sigma_g=\partial(P\times I)$ in odd genus. The horizontal surfaces $P^+$ and $P^-$
are indicated schematically by dotted lines.
Each rectangle represents a vertical annulus
$B_\nu\times I$, with its two vertical sides identified.
Thus each horizontal segment represents a circle,
and the dashed segments indicate the surgery cores
$B_\nu\times\{1/2\}$.}
\label{fig:boundary-components-based-paths}
\end{figure}

Orient each core $B_\nu\times\{1/2\}$ so that
$\Lambda_\nu$, directed from $p$ to $p_-$,
crosses it positively in the convention of
\cref{lem:local-crossing-coordinates}.
Give $B_\nu\times\{1\}$ and $B_\nu\times\{0\}$
the same orientation. Let
\[
b_\nu^+\in\pi_1(P^+,p),
\qquad
b_\nu^-\in\pi_1(P^-,p_-)
\]
be the loops obtained by traversing the oriented circles
$B_\nu\times\{1\}$ and $B_\nu\times\{0\}$, respectively.
For $B_1$, these circles are already based at $p$ and $p_-$.
For $B_2$, use $\beta_2^+$ and $\beta_2^-$ as basing paths.
The product annulus $B_\nu\times I$ gives a based
homotopy between the upper boundary circle and the
loop obtained by following $\lambda_\nu$ to the lower
boundary circle, traversing it once, and returning
along $\lambda_\nu^{-1}$. Therefore we have
\[
\Lambda_\nu b_\nu^-\Lambda_\nu^{-1}=b_\nu^+
\qquad\text{in }\pi_1(\Sigma_g,p).
\]
Under the projections $P^\pm\to P$, both
$b_\nu^+$ and $b_\nu^-$ represent the same loop
in $\pi_1(P,q_0)$. For $\nu=1$, this loop traverses
$B_1$ once, starting at $q_0$. For $\nu=2$, it follows
$\beta_2$ to $B_2$, traverses $B_2$ once, and returns
along $\beta_2^{-1}$.

Thus $b_\nu^+$ and $b_\nu^-$ are expressed by the same
word in the free bases $y_j^+$ and $y_j^-$ chosen in
\cref{prop:horizontal-peripheral}. Since that proposition
gives
\[
\rho(y_j^+)
=\rho(\lambda_-y_j^-\lambda_-^{-1})
=\ell_j^D,
\]
we obtain
\begin{equation}
\label{eq:vertical-boundary-common-image}
\rho(b_\nu^+)
=
\rho(\lambda_-b_\nu^-\lambda_-^{-1})
=:
w_\nu\in U_0.
\end{equation}

To base the core $B_\nu\times\{1/2\}$ at $p$,
we follow $\Lambda_\nu$ to its intersection with the core,
traverse the core once in the chosen orientation,
and return along the same path.
The annulus $B_\nu\times[1/2,1]$ gives a homotopy
of this based loop to $b_\nu^+$.
Thus its pre-surgery push-off is
\[
\rho(b_\nu^+)=w_\nu\in U_0,
\]
which commutes with $\mu$ by
\eqref{eq:disk-meridian-centralizer}.
In even genus, write $w=w_1$.

For each core, choose an oriented nontrivial hyperbolic knot
\(J_\nu^V\) and write
\[
K_\nu^V=K_{J_\nu^V},
\qquad
m_\nu^V=m_{J_\nu^V},
\qquad
\ell_\nu^V=\ell_{J_\nu^V}.
\]
We perform rim surgery along
\(B_\nu\times\{1/2\}\), using the section-compatible product framing
of \cref{sec:rimlocal} and identifying \(m_\nu^V\) with the positive
surface meridian. The edge maps required by \cref{thm:one-step}
are injective, as we verify below.

\subsection{The change in the push-off homomorphism}

Let $N_\nu\cong S^1\times B^3$ be the ambient product
neighborhood of $B_\nu\times\{1/2\}$ used for rim
surgery, and let $M_\nu$ be the common attaching
region of the old and new local exteriors. The path $\Lambda_\nu$ meets $N_\nu$ in a single
subarc of $\lambda_\nu$, directed from $P^+$ to $P^-$.
In the old exterior, homotope the positive push-off
of this subarc into $M_\nu$, keeping its endpoints
fixed, as in \cref{lem:local-crossing-coordinates}.
Leave the rest of the push-off unchanged and denote
the resulting path by $\widehat\Lambda_\nu$.
Let $\widehat\Lambda_\nu'$ be the positive push-off
after surgery. Both paths run from $*$ to the
push-off of $p_-$, since the surface and normal
section are unchanged near their endpoints.

Choose $*_\nu\in M_\nu$ on the boundary torus of
the new local exterior. Define
$\omega_\nu^V:*\to *_\nu$ by following
$\widehat\Lambda_\nu$ to the end of its subarc in
$M_\nu$ on the $P^-$ side, and then continuing
inside $M_\nu$ to $*_\nu$.

Use $\omega_\nu^V$ to base the knot group $K_\nu^V$
at $*$. In particular, its preferred longitude
becomes
\[
\ell_\nu^V
=\omega_\nu^V\ell_{\nu,\mathrm{loc}}^V
(\omega_\nu^V)^{-1}.
\]
The circle generator of $M_\nu$ is a parallel of
the surgery core. Using $\omega_\nu^V$ to base loops at $*$,
the loop in $M_\nu$ parallel to
$B_\nu\times\{1/2\}$ represents $w_\nu$,
while a positive surface meridian represents $\mu$.

By \cref{lem:local-crossing-coordinates}, the positive crossing
inserts one preferred longitude. Basing this longitude at
$*$ using $\omega_\nu^V$ gives $\ell_\nu^V$.
Hence, in the new exterior and relative to endpoints,
\begin{equation}
\label{eq:vertical-path-insertion}
\widehat\Lambda_\nu'=\ell_\nu^V\widehat\Lambda_\nu.
\end{equation}

After these surgeries, let \(F_0\) be the resulting surface,
\(G_0\) be its exterior group, and
\[
\rho_0:\pi_1(\Sigma_g,p)\longrightarrow G_0
\]
be its push-off homomorphism. The map \(\rho\) still denotes the
push-off before the vertical-boundary surgeries. In formulas in
\(G_0\), its values are understood via the natural homomorphism
\(A_0\to G_0\).
Loops in \(P^+\) avoid the surgery neighborhoods, so
\begin{equation}
\label{eq:vertical-plus-restriction}
\rho_0(\gamma)=\rho(\gamma)
\qquad
(\gamma\in\pi_1(P^+,p)).
\end{equation}
Loops in \(P^-\) also avoid these neighborhoods, but their basing
path \(\lambda_-=\Lambda_1\) acquires \(\ell_1^V\).
Thus \eqref{eq:vertical-path-insertion} gives
\begin{equation}
\label{eq:vertical-minus-restriction}
\rho_0(\lambda_-\gamma\lambda_-^{-1})
=
\ell_1^V\,
\rho(\lambda_-\gamma\lambda_-^{-1})\,
(\ell_1^V)^{-1}
\qquad
(\gamma\in\pi_1(P^-,p_-)).
\end{equation}
Consequently,
\[
\rho_0(\pi_1(P^+,p))=U_0,
\qquad
\rho_0(\lambda_-\pi_1(P^-,p_-)\lambda_-^{-1})
=
\ell_1^VU_0(\ell_1^V)^{-1}.
\]
In odd genus, \(\tau\) follows \(\Lambda_2\) and returns along
\(\Lambda_1\). Each path changes only at its own crossing, so
\eqref{eq:vertical-path-insertion} and
\eqref{eq:old-odd-cycle-trivial} give
\begin{equation}
\label{eq:vertical-cycle-image}
\rho_0(\tau)
=
\ell_2^V\,\rho(\tau)\,(\ell_1^V)^{-1}
=
\ell_2^V(\ell_1^V)^{-1}.
\end{equation}

Each $B_\nu$ is essential in $P$, so
\cref{prop:horizontal-peripheral} gives $w_\nu\ne1$.
Put
\[
E_\nu=\langle\mu,w_\nu\rangle,
\qquad C_\nu^V=K_\nu^V\times\langle z_\nu\rangle.
\]
By \eqref{eq:disk-meridian-centralizer},
\[
E_\nu\cong\mathbb Z^2,
\qquad U_0\cap E_\nu=\langle w_\nu\rangle.
\]
In $C_\nu^V$, the direct-product structure and
the peripheral injection
$\langle m_\nu^V,\ell_\nu^V\rangle\cong\mathbb Z^2$
give
\[
\langle z_\nu,\ell_\nu^V\rangle
\cap\langle m_\nu^V,z_\nu\rangle
=\langle z_\nu\rangle.
\]
Thus the edge maps required by \cref{thm:one-step}
are injective. In the resulting amalgam,
$m_\nu^V$ is identified with $\mu$ and
$z_\nu$ with $w_\nu$. In odd genus, the first surgery embeds
$A_0$ in the intermediate exterior group and leaves the
push-off of the second core equal to $w_2$.
Hence the rank-two hypothesis also holds for the second surgery.

\subsection{Algebraic preliminaries}
We first record two algebraic lemmas used to prove the injectivity in both
the even- and odd-genus cases.
Let \(U\) be a free group.  Two nontrivial elements are
\emph{commensurable} if some nonzero powers are conjugate.
\begin{lemma}[Boundary words in free groups]\label{lem:free-boundary-words}
\begin{enumerate}[label=\textup{(\roman*)},leftmargin=2.3em]
\item A cyclically reduced word in a free basis is
not a proper power if some generator or its inverse
occurs exactly once.
\item If \(w\ne1\) is not a proper power, then
\(C_U(w)=\langle w\rangle\), and \(\langle w\rangle\) is maximal
cyclic and malnormal.
\item In \(\pi_1(\Sigma_{r,2})\), with \(r\ge1\), the two based
boundary elements are nontrivial, are not proper powers, and are
noncommensurable.  Every conjugate of one boundary cyclic subgroup
has trivial intersection with the other.
\end{enumerate}
\end{lemma}

\begin{proof}
A proper power repeats a cyclic word, so each oriented-letter
multiplicity is divisible by the exponent.  This proves~\textup{(i)}.

For~\textup{(ii)}, use the free action on the Cayley tree.
The stabilizer of the axis of \(w\) acts faithfully on that line.
It contains no reflection: a reflection would have square fixing the
line pointwise and hence would have order two.  Its translation image
is a nonzero subgroup of \(\mathbb Z\), so the stabilizer is generated
by a shortest positive translation \(u\).  Since \(w\) is not a
proper power, \(\langle u\rangle=\langle w\rangle\).
Any conjugator producing a nontrivial intersection of
\(\langle w\rangle\) with its conjugate preserves the same axis and
therefore belongs to \(\langle w\rangle\).  This proves both claims.

For~\textup{(iii)}, choose a free basis
\(a_1,b_1,\ldots,a_r,b_r,c\) of
\(\pi_1(\Sigma_{r,2})\) so that the two
boundary elements are represented, up to
conjugacy and inversion, by
\[
c,\qquad
w=\left(\prod_{j=1}^r[a_j,b_j]c\right)^{-1}.
\]
Both elements are nontrivial. The element \(c\)
is a free generator, while \(c^{-1}\) occurs
exactly once in the cyclically reduced word
for \(w\). Thus neither is a proper power
by~\textup{(i)}.

We next show that no nonzero power of \(c\)
is conjugate to a nonzero power of \(w\).
Cap off the boundary component represented
by \(c\) with a disk. The induced homomorphism
to the free group on \(a_1,b_1,\ldots,a_r,b_r\)
sends \(c\) to \(1\) and \(w\) to
\[
\left(\prod_{j=1}^r[a_j,b_j]\right)^{-1},
\]
which is nontrivial. If \(c^m\) were conjugate
to \(w^n\) for some nonzero integers \(m,n\),
applying this homomorphism would make a
nonzero power of the displayed word trivial.
This is impossible in a free group.
Hence the boundary elements are
noncommensurable. The same argument shows
that every conjugate of \(\langle c\rangle\)
intersects \(\langle w\rangle\) trivially.
\end{proof}

\begin{lemma}[A subgroup normal-form criterion]
\label{lem:subgroup-amalgam}
Let \(G=G_1*_{A}G_2\) have injective edge maps
\(\iota_i:A\to G_i\).  If \(C\le A\), \(U_i\le G_i\), and
\[
U_i\cap\iota_i(A)=\iota_i(C)
\qquad(i=1,2),
\]
then, with the maps \(C\to U_i\) given by
\(\iota_i|_C\), the vertex inclusions induce an injection
\[
U_1*_{C}U_2\longrightarrow G.
\]
\end{lemma}

\begin{proof}
By the intersection hypothesis, each syllable of a reduced word outside
\(\iota_i(C)\) also lies outside \(\iota_i(A)\).  The same word is
therefore reduced and nontrivial in \(G_1*_{A}G_2\).
\end{proof}

\subsection{Even genus}

Let $g=2r$ and $P=\Sigma_{r,1}$. Choose a standard
free basis
\[
\bar a_1,\bar b_1,\ldots,\bar a_r,\bar b_r
\]
of $\pi_1(P^+,p)$ such that the oriented boundary
loop $b_1^+$ satisfies
\[
b_1^+=\prod_{j=1}^r[\bar a_j,\bar b_j].
\]
Take the copies of these loops in $P^-=P\times\{0\}$,
using the product structure $P\times I$.
They are based at $p_-$ and form a free basis
of $\pi_1(P^-,p_-)$. By
\cref{prop:horizontal-peripheral}, the two homomorphisms from
\(\pi_1(P^+,p)\) and \(\pi_1(P^-,p_-)\) send corresponding loops
to the same elements of \(U_0\).  Set
\[
a_j=\rho(\bar a_j),
\qquad
b_j=\rho(\bar b_j).
\]
Because \(\rho|_{\pi_1(P^+,p)}\) is an isomorphism onto \(U_0\),
\[
U_0=\langle a_1,b_1,\ldots,a_r,b_r\rangle
\]
and the push-off of \(b_1^{+}\) is
\begin{equation}\label{eq:even-seam-word}
w=\rho(b_1^+)=\prod_{j=1}^r[a_j,b_j].
\end{equation}

The word \(w\) is nonempty and cyclically reduced. Since $w\ne1$ in the free group $U_0$, it has
infinite order. By \eqref{eq:disk-meridian-centralizer},
$\mu$ commutes with $w$ and
$\langle\mu\rangle\cap U_0=1$. Hence
\[
E:=E_1=\langle\mu,w\rangle
=\langle\mu\rangle\times\langle w\rangle
\cong\mathbb Z^2.
\]
The one-step calculation of \cref{thm:one-step} applies, with edge maps
\[
\mu\longmapsto m_1^V,
\qquad
w\longmapsto z_1,
\]
and gives
\begin{equation}\label{eq:even-seed-group}
G_0=A_0*_{E}C_1^V.
\end{equation}
After forming the amalgam, \(z_1\) is identified with \(w\).

The following embedding uses the standard
centralizer-extension construction
\cite[Corollary~3.6 and Sections~4.1--4.2]{ChampetierGuirardel}.
We verify it inside the rim-surgery exterior with
the specified peripheral identifications.

\begin{lemma}\label{lem:even-L-embedding}
The subgroup generated by \(U_0\) and \(\ell_1^V\) in \(G_0\) is
naturally isomorphic to
\begin{equation}\label{eq:even-HNN}
L
=
U_0*_{\langle w\rangle}\langle w,\ell_1^V\rangle
\cong
\left\langle
U_0,\ell_1^V
\ \middle|\
[\ell_1^V,w]=1
\right\rangle.
\end{equation}
Moreover, \(L\cap A_0=U_0\).
\end{lemma}

\begin{proof}
The common intersection calculation gives
\(U_0\cap E=\langle w\rangle\) and
\(\langle w,\ell_1^V\rangle\cap E=\langle w\rangle\).
Thus \cref{lem:subgroup-amalgam} embeds \eqref{eq:even-HNN} in
\(G_0=A_0*_{E}C_1^V\).
A reduced word using a syllable from
\(\langle w,\ell_1^V\rangle\setminus\langle w\rangle\)
remains reduced in this ambient amalgam and cannot lie in \(A_0\).
Hence \(L\cap A_0=U_0\).
\end{proof}

\begin{proposition}[Injectivity of the push-off homomorphism in even genus]
\label{prop:even-seed-injection}
The word \(w\) in \eqref{eq:even-seam-word} is not a proper power,
\(\langle w\rangle\) is maximal cyclic and malnormal in \(U_0\), and
\begin{equation}\label{eq:even-intersection}
U_0\cap
\ell_1^VU_0(\ell_1^V)^{-1}
=
\langle w\rangle
\end{equation}
in \(L\).  Consequently the push-off homomorphism induces an embedding
\begin{equation}\label{eq:even-surface-embedding}
\rho_0:\pi_1(\Sigma_{2r},p)
\cong
U_0*_{\langle w\rangle}U_0
\into
L
\into
G_0,
\end{equation}
where the two displayed copies of \(U_0\) arise from
\(\pi_1(P^+,p)\) and from \(\pi_1(P^-,p_-)\) using \(\lambda_-\),
respectively.
\end{proposition}

\begin{proof}
Each oriented letter
\(a_j,b_j,a_j^{-1},b_j^{-1}\) occurs exactly once in the cyclically
reduced word for \(w\).  By \cref{lem:free-boundary-words}, \(w\) is
not a proper power and \(\langle w\rangle\) is maximal cyclic and
malnormal.

Suppose that
\[
\ell_1^Vu(\ell_1^V)^{-1}\in U_0
\]
for some \(u\in U_0\).  If \(u\notin\langle w\rangle\), then the
word \(\ell_1^Vu(\ell_1^V)^{-1}\) is reduced in the HNN presentation
\eqref{eq:even-HNN}; Britton's lemma \cite{LyndonSchupp} says that it does not lie in the
base group \(U_0\).  Thus \(u\in\langle w\rangle\).  Conversely,
\(\ell_1^V\) centralizes \(w\), proving
\eqref{eq:even-intersection}.

The decomposition
\[
\Sigma_{2r}=P^+\cup_{B_1\times I}P^-
\]
and the path \(\lambda_-\) give an amalgamated-product decomposition
of the based surface group.  Under the two isomorphisms supplied by
\cref{prop:horizontal-peripheral}, both boundary generators map to
\(w\), so this decomposition becomes
\[
\pi_1(\Sigma_{2r},p)
\cong
U_0*_{\langle w\rangle}U_0.
\]
By \eqref{eq:vertical-plus-restriction} and
\eqref{eq:vertical-minus-restriction}, \(\rho_0\) maps the first
vertex group identically to \(U_0\) and maps the second by
\[
u\longmapsto\ell_1^Vu(\ell_1^V)^{-1}.
\]
The two edge maps agree because \([\ell_1^V,w]=1\).

Let a nontrivial element of the surface amalgam be represented by a
reduced word.  If it has more than one vertex-group syllable, every
syllable lies outside \(\langle w\rangle\).  Replacing each syllable
from the second copy of \(U_0\) by its \(\ell_1^V\)-conjugate produces
a word in the HNN presentation \eqref{eq:even-HNN}.  A Britton pinch
could occur only in a subword
\[
\ell_1^Vu(\ell_1^V)^{-1}
\quad\text{or}\quad
(\ell_1^V)^{-1}u\ell_1^V
\]
with \(u\in\langle w\rangle\), contrary to reducedness of the
surface-amalgam word.  The image is therefore HNN-reduced and
nontrivial.  A word contained in one vertex group is nontrivial because
each vertex-group map is injective.  Hence \(\rho_0\) is injective into
\(L\), and \cref{lem:even-L-embedding} embeds \(L\) in \(G_0\).
\end{proof}

We denote the embedded surface subgroup by
\begin{equation}\label{eq:even-S0}
S_0=\rho_0(\pi_1(\Sigma_{2r},p))
=
\left\langle
U_0,
\ell_1^VU_0(\ell_1^V)^{-1}
\right\rangle.
\end{equation}

\subsection{Odd genus}

Let \(g=2r+1\) and \(P=\Sigma_{r,2}\). Recall that
$b_1^+\in\pi_1(P^+,p)$ traverses the oriented circle
$B_1\times\{1\}$ once, starting at $p$. The loop
$b_2^+\in\pi_1(P^+,p)$ follows $\beta_2^+$ from $p$
to $(q_2,1)$, traverses the oriented circle
$B_2\times\{1\}$ once, and returns along
$(\beta_2^+)^{-1}$.

Choose a standard free basis
\[
\bar a_1,\bar b_1,\ldots,\bar a_r,\bar b_r,\bar c
\]
of $\pi_1(P^+,p)$ such that
\[
b_2^+=\bar c,
\qquad
b_1^+=
\left(
\prod_{j=1}^r[\bar a_j,\bar b_j]\bar c
\right)^{-1}.
\]

Take the copies of these loops in $P^-=P\times\{0\}$
under the product identification
$(x,1)\mapsto(x,0)$. They are based at $p_-$
and form a free basis of $\pi_1(P^-,p_-)$. As in the even case,
\cref{prop:horizontal-peripheral} sends corresponding loops in
\(P^+\) and \(P^-\) to the same elements of \(U_0\).  Set
\[
a_j=\rho(\bar a_j),
\qquad
b_j=\rho(\bar b_j),
\qquad
c=\rho(\bar c).
\]
Then
\[
U_0=\langle a_1,b_1,\ldots,a_r,b_r,c\rangle
\]
and
\begin{equation}\label{eq:odd-seam-words}
w_2=c,
\qquad
w_1=
\left(
\prod_{j=1}^r[a_j,b_j]c
\right)^{-1}.
\end{equation}

Both \(w_1\) and \(w_2\) are nontrivial, so a similar calculation as before
gives \(E_\nu=\langle\mu,w_\nu\rangle\cong\mathbb Z^2\).
Perform rim surgery associated to \(B_1\).
By \cref{thm:one-step},
\[
G_0^{(1)}
=
A_0*_{E_1}C_1^V,
\qquad
\mu\longmapsto m_1^V,
\quad
w_1\longmapsto z_1.
\]
In particular, the canonical homomorphism
$A_0\to G_0^{(1)}$ is injective. The based representative of $B_2\times\{1/2\}$
defined in Subsection~5.1 is disjoint from the
first surgery neighborhood. Hence its push-off
remains $w_2\in A_0\subset G_0^{(1)}$.  Thus the second edge group is still
\(E_2\cong\mathbb Z^2\), and the second rim surgery
gives
\begin{equation}\label{eq:odd-seed-group}
G_0
=
\bigl(A_0*_{E_1}C_1^V\bigr)*_{E_2}C_2^V,
\end{equation}
where the second pair of edge maps sends
\[
\mu\longmapsto m_2^V,
\qquad
w_2\longmapsto z_2.
\]
After forming the two amalgams, \(z_\nu\) is identified with
\(w_\nu\).

The subgroup generated by \(U_0,\ell_1^V,\ell_2^V\) has the relative
presentation
\begin{equation}\label{eq:odd-L}
L=
\left\langle
U_0,\ell_1^V,\ell_2^V
\ \middle|\
[\ell_1^V,w_1]=1,
[\ell_2^V,w_2]=1
\right\rangle.
\end{equation}
The proof of \cref{lem:even-L-embedding}, with \(w,E\) replaced by
\(w_1,E_1\), embeds
\[
 L_1=\left\langle U_0,\ell_1^V\ \middle|\ [\ell_1^V,w_1]=1\right\rangle
 \le G_0^{(1)}
\]
and gives \(L_1\cap A_0=U_0\).
Therefore
\[
E_2\cap L_1
=
E_2\cap U_0
=
\langle w_2\rangle.
\]
The intersection calculation in Subsection~5.2 gives
\[
E_2\cap\langle w_2,\ell_2^V\rangle
=\langle w_2\rangle.
\]
Applying \cref{lem:subgroup-amalgam} again
therefore gives an embedding
\[
L_1*_{\langle w_2\rangle}\langle w_2,\ell_2^V\rangle
\cong L
\]
into $G_0$.

By \cref{lem:free-boundary-words}, neither $w_1$ nor $w_2$
is a proper power, and no nonzero power of one is conjugate
to a nonzero power of the other. Thus each generates a
maximal cyclic subgroup that is malnormal in $U_0$,
and every conjugate of $\langle w_1\rangle$ intersects
$\langle w_2\rangle$ trivially.

We first use the annulus $B_1\times I$ to form
the amalgamated product of the fundamental groups
of $P^+$ and $P^-$. The remaining annulus
$B_2\times I$ then gives an HNN extension.
By \eqref{eq:vertical-plus-restriction} and
\eqref{eq:vertical-minus-restriction}, the two vertex-group images are
\[
U_0
\quad\text{and}\quad
\ell_1^VU_0(\ell_1^V)^{-1}.
\]
Set
\[
H=
\left\langle
U_0,
\ell_1^VU_0(\ell_1^V)^{-1}
\right\rangle
\le L_1.
\]
The even-genus normal-form argument gives
\[
H
\cong
U_0*_{\langle w_1\rangle}
\ell_1^VU_0(\ell_1^V)^{-1}.
\]
The based paths satisfy
\[
\tau
\bigl(\Lambda_1b_2^-\Lambda_1^{-1}\bigr)
\tau^{-1}
=
b_2^+,
\]
because \(\tau=\Lambda_2\Lambda_1^{-1}\) and
\(\Lambda_2b_2^-\Lambda_2^{-1}=b_2^+\).  The two images of the second
boundary subgroup in \(H\) are therefore
\[
W_2^+=\langle w_2\rangle,
\qquad
W_2^-=
\ell_1^V\langle w_2\rangle(\ell_1^V)^{-1}.
\]
The loop \(\tau\) from \eqref{eq:odd-surface-cycle} is the stable
letter in the resulting HNN presentation of the surface group, with
\[
\tau W_2^-\tau^{-1}=W_2^+.
\]

\begin{proposition}[Injectivity of the push-off homomorphism in odd genus]
\label{prop:odd-seed-injection}
For the HNN extension described above, the stable letter
$\tau=\Lambda_2\Lambda_1^{-1}$ satisfies
\begin{equation}\label{eq:odd-stable-letter}
\rho_0(\tau)
=
\ell_2^V(\ell_1^V)^{-1}.
\end{equation}
The resulting push-off homomorphism
\begin{equation}\label{eq:odd-surface-injection}
\rho_0:\pi_1(\Sigma_{2r+1},p)\into L\into G_0
\end{equation}
is injective.
\end{proposition}

\begin{proof}
The stable-letter formula is \eqref{eq:vertical-cycle-image}.  It also
respects the HNN relation, because
\[
\begin{aligned}
&\ell_2^V(\ell_1^V)^{-1}
\bigl(\ell_1^Vw_2(\ell_1^V)^{-1}\bigr)
\ell_1^V(\ell_2^V)^{-1}\\
&\hspace{4em}
=
\ell_2^Vw_2(\ell_2^V)^{-1}
=
w_2.
\end{aligned}
\]
Thus the surface HNN presentation maps to \(L\).

Write a reduced word in that surface HNN presentation as
\[
u_0\tau^{\epsilon_1}u_1\cdots
\tau^{\epsilon_n}u_n,
\qquad
u_j\in H,
\quad
\epsilon_j\in\{\pm1\}.
\]
Reducedness means that, whenever
\((\epsilon_j,\epsilon_{j+1})=(1,-1)\), one has
\(u_j\notin W_2^-\), and whenever
\((\epsilon_j,\epsilon_{j+1})=(-1,1)\), one has
\(u_j\notin W_2^+\).

Substitute \(\tau=\ell_2^V(\ell_1^V)^{-1}\) and view the result in
\[
L
=
\left\langle
L_1,\ell_2^V
\ \middle|\
[\ell_2^V,w_2]=1
\right\rangle.
\]
A subword \(\tau u\tau^{-1}\) becomes
\[
\ell_2^V
\bigl((\ell_1^V)^{-1}u\ell_1^V\bigr)
(\ell_2^V)^{-1}.
\]
It is a Britton pinch in the \(\ell_2^V\)-HNN extension exactly when
\[
(\ell_1^V)^{-1}u\ell_1^V\in\langle w_2\rangle,
\]
equivalently when \(u\in W_2^-\).  Similarly, a subword
\(\tau^{-1}u\tau\) contains
\[
(\ell_2^V)^{-1}u\ell_2^V
\]
and pinches exactly when \(u\in W_2^+\).  These are precisely the two
conditions excluded by reducedness of the surface HNN word.  Adjacent
stable letters with the same sign do not form a Britton pinch.  Hence
every word containing a \(\tau\)-letter maps to an
\(\ell_2^V\)-HNN-reduced, and therefore nontrivial, word.  A word with
no \(\tau\)-letter lies in \(H\) and is nontrivial because
\(H\into L_1\into L\).  Britton's lemma proves injectivity into
\(L\), and the two chronological applications of
\cref{lem:subgroup-amalgam} embed \(L\) in \(G_0\).
\end{proof}

In odd genus we denote
\begin{equation}\label{eq:odd-S0}
S_0=\rho_0(\pi_1(\Sigma_{2r+1},p))
=
\left\langle
U_0,
\ell_1^VU_0(\ell_1^V)^{-1},
\ell_2^V(\ell_1^V)^{-1}
\right\rangle.
\end{equation}

\subsection{Choice of knots and homological consequences}
The one or two knots $J_\nu^V$ used for the
vertical-boundary surgeries, and all knots used
for later surgeries, are chosen to be hyperbolic,
with pairwise nonisomorphic meridian-marked groups. Choices also meeting the Alexander requirements are taken care of in \cref{prop:hyperbolic-alexander-supply}.

The following result supplies the homological normalization and
separates meridian powers from later push-off images.

\begin{proposition}[First homology and push-off images]
\label{prop:early-homology}
For the preliminary group \(A_0\), for the exterior group \(G_0\) after
the vertical-boundary surgeries, and for every exterior group \(G_i\)
obtained by the subsequent rim surgeries,
\[
H_1(A_0;\mathbb Z)=\mathbb Z\langle[\mu]\rangle,
\qquad
H_1(G_i;\mathbb Z)=\mathbb Z\langle[\mu]\rangle
\qquad(i\ge0),
\]
and
\[
\rho(\pi_1(\Sigma_g,p))\subset[A_0,A_0],
\qquad
\rho_i(\pi_1(\Sigma_g,p))\subset[G_i,G_i]
\qquad(i\ge0).
\]
The same homology and push-off conclusions hold at every intermediate
preliminary and vertical-boundary stage.  In particular, every later
pre-surgery element \(h_i=\rho_{i-1}(d_i)\) is null-homologous.
\end{proposition}

\begin{proof}
We first consider a single rim surgery, using the notation
of \cref{sec:rimlocal}. Let $E_F$ and $E_{F'}$ be the
exteriors before and after surgery. Recall that
\[
E_F=X\cup_M V_0,
\qquad
E_{F'}=X\cup_M V_J,
\]
where $X$ is the part of the exterior unchanged by surgery,
$V_0$ and $V_J$ are the old and new local exteriors, and
$M\cong T^2\times I$ is their common attaching region.

The group $H_1(M;\mathbb Z)\cong\mathbb Z^2$ is generated
by the circle direction and the positive surface meridian.
The inclusion $M\hookrightarrow V_0$ induces an
isomorphism on first homology. The same holds for
$M\hookrightarrow V_J$, since
\[
H_1(V_J;\mathbb Z)
\cong H_1(K_J;\mathbb Z)\oplus\mathbb Z\langle[s]\rangle
\cong\mathbb Z^2,
\]
and the two generators of $H_1(M;\mathbb Z)$ map to
$[s]$ and $[m_J]$, respectively.

Apply the Mayer--Vietoris sequence to
$E_{F'}=X\cup_M V_J$. Since all three spaces are
connected, the relevant part is
\[
H_1(M;\mathbb Z)
\longrightarrow
H_1(X;\mathbb Z)\oplus H_1(V_J;\mathbb Z)
\longrightarrow
H_1(E_{F'};\mathbb Z)
\longrightarrow 0.
\]
Because $H_1(M;\mathbb Z)\to H_1(V_J;\mathbb Z)$
is an isomorphism, every class from $V_J$ can be
replaced by a class from $X$, and no nonzero class
from $X$ becomes trivial. Thus inclusion induces
an isomorphism
\[
H_1(X;\mathbb Z)\xrightarrow{\cong}
H_1(E_{F'};\mathbb Z).
\]
The same argument applied to $E_F=X\cup_M V_0$
gives
\[
H_1(X;\mathbb Z)\xrightarrow{\cong}
H_1(E_F;\mathbb Z).
\]
Consequently, first homology is unchanged by
the surgery. These isomorphisms preserve the
positive surface meridian, which can be represented
in the unchanged part $X$. Notice that the argument
does not require the homomorphism
$\pi_1(M)\to\pi_1(X)$ to be injective.

We next compare the push-off homomorphisms.
Let $\gamma$ be a based loop on the surface.
Divide $\gamma$ into arcs outside the surgery annulus
and arcs crossing it. Outside the surgery annulus,
the surface and the chosen normal section are unchanged,
so the old and new push-offs agree.

For each crossing, move the old pushed-off arc into
$M$, keeping its endpoints fixed. By
\cref{lem:local-crossing-coordinates}, the corresponding
new pushed-off arc differs from it by the insertion
of a conjugate of a preferred knot longitude
$\ell_J^{\pm1}$. The preferred longitude has zero
class in $H_1(K_J;\mathbb Z)$, and hence also in
$H_1(V_J;\mathbb Z)$ and $H_1(E_{F'};\mathbb Z)$.
Conjugation does not change homology classes.
Therefore these insertions do not change the
homology class of the pushed-off loop.

It follows that the old and new push-offs of
$\gamma$ determine the same class under the
identification
\[
H_1(E_F;\mathbb Z)
\cong H_1(X;\mathbb Z)
\cong H_1(E_{F'};\mathbb Z).
\]

Finally, the standard unknotted surface has
exterior group $\langle\mu\rangle\cong\mathbb Z$
and trivial push-off homomorphism, as established
in \cref{sec:disk}. Every subsequent rim surgery
preserves first homology and the homology classes
of all pushed-off surface loops. By induction,
at every preliminary, vertical-boundary, and
ordered surgery stage, the exterior group has
first homology $\mathbb Z$ generated by $[\mu]$,
and the image of the push-off homomorphism has
zero homology class.

Since the kernel of abelianization
$G\to H_1(G;\mathbb Z)$ is $[G,G]$, the push-off
image lies in the commutator subgroup at every
stage. In particular,
\[
h_i=\rho_{i-1}(d_i)\in[G_{i-1},G_{i-1}],
\]
so $h_i$ is null-homologous.
\end{proof}

\section{Centralizers, transporters, and Bass--Serre axes}
\label{sec:powerclean}

In this section, we establish two algebraic criteria using Bass–Serre theory: one controlling centralizers and conjugation of elements in rank-two surgery edge groups, and the other controlling conjugacy of cyclic subgroups in the surface group. These criteria will be applied in Sections 7 and 8 to the successive rim surgery curves.

We study these algebraic properties through the Bass--Serre trees associated with the amalgamated-product decompositions of our exterior groups. The key idea is to control centralizers and conjugacy by examining the stabilizers of axes and invariant subtrees. We recall the relevant terminology.

Recall that an amalgamated product \(H=G*_A B\) acts on a tree
\(T\) whose vertices are the cosets \(H/G\) and \(H/B\), and whose
edges are the cosets \(H/A\). The edge \(hA\) joins the vertices
\(hG\) and \(hB\). The group \(H\) acts by left multiplication,
and the stabilizers of \(hG\), \(hB\), and \(hA\) are
\(hGh^{-1}\), \(hBh^{-1}\), and \(hAh^{-1}\), respectively.

Throughout, trees are simplicial, and group actions are by simplicial automorphisms without inversions. For an element \(x\) acting on a tree \(T\), its translation length is
\[
\|x\|_T=\min_{v\in V(T)}d_T(v,xv).
\]
An element is \textit{elliptic} if it fixes a vertex and \textit{hyperbolic} otherwise. A hyperbolic element \(x\) has a unique invariant
bi-infinite geodesic, called its \emph{axis}
\(\Ax_T(x)\). Its vertex set is
\[
V(\Ax_T(x))
=
\{v\in V(T):d_T(v,xv)=\|x\|_T\}.
\]
The element \(x\) acts on this axis by translation
through \(\|x\|_T>0\) edges.

For a subtree \(L\), we write
\[
\Fix_G^{\mathrm{pt}}(L)=\{g\in G:gv=v\text{ for every }v\in L\},
\qquad
\Stab_G(L)=\{g\in G:gL=L\}.
\]
When \(L\subset T\) is a bi-infinite geodesic, a
\emph{reflection} of \(L\) is an element \(g\in G\)
such that \(gL=L\) and \(g\) exchanges the two ends
of \(L\). Since the action is without inversions,
such an element fixes a vertex of \(L\) and
interchanges the two rays from that vertex.

\begin{definition}
\label{def:transporter}
For \(A\le G\) and \(x\in G\), the \emph{transporter of \(x\)
into \(A\)} is
\[
 \Trans_G(x,A)=\{g\in G:g^{-1}xg\in A\}.
\]
\end{definition}

\begin{definition}[Centralizer--transporter property]
\label{def:power-clean}
Let \(A=\langle\mu,h\rangle\cong\mathbb Z^2\).
We say it has the \emph{centralizer--transporter property in \(G\) relative
to \(\mu\)} if, for every \(a,b\in\mathbb Z\) with \(b\ne0\),
\begin{equation}
\label{eq:power-clean-transporter}
 \Trans_G(\mu^ah^b,A)=A.
\end{equation}
Since \(A\le C_G(x)\le\Trans_G(x,A)\) for \(x\in A\), this also gives
\begin{equation}
\label{eq:power-clean-centralizer}
 C_G(\mu^ah^b)=A\qquad(b\ne0).
\end{equation}
Since every \(x\in A\setminus\langle\mu\rangle\) has the form
\(\mu^ah^b\) with \(b\ne0\), and
\[
A\le C_G(x)\le\Trans_G(x,A),
\]
the property is equivalent to
\[
C_G(x)=\Trans_G(x,A)=A
\qquad(x\in A\setminus\langle\mu\rangle).
\]
In particular, it is independent of the choice of \(h\) and
is preserved when \(A\) and \(\langle\mu\rangle\) are
conjugated simultaneously in \(G\).
\end{definition}

A \emph{graph of groups} \(\mathcal G\) consists of a connected
graph \(\Gamma\), a group \(G_v\) for each vertex \(v\), and
a group \(G_e\) for each edge \(e\), together with injective
homomorphisms
\[
\iota_{e,-}:G_e\hookrightarrow G_{v_-},
\qquad
\iota_{e,+}:G_e\hookrightarrow G_{v_+},
\]
where \(v_-\) and \(v_+\) are the endpoints of \(e\), in a
chosen order.

Choose a maximal subtree \(\Gamma_0\subset\Gamma\).
The \emph{fundamental group of the graph of groups},
denoted \(\pi_1(\mathcal G)\), is generated by the
vertex groups and an additional generator \(t_e\),
called a \emph{stable letter}, for each edge
\(e\notin\Gamma_0\). Its defining relations are
\[
\iota_{e,-}(a)=
\begin{cases}
\iota_{e,+}(a), & e\in\Gamma_0,\\
t_e\iota_{e,+}(a)t_e^{-1}, & e\notin\Gamma_0,
\end{cases}
\qquad a\in G_e.
\]
This is distinct from the ordinary fundamental group of
the underlying graph \(\Gamma\). For example, a graph
of groups consisting of two vertex groups \(G,B\)
joined by an edge group \(A\) has fundamental group
\(G*_A B\).

\begin{lemma}[Torsion-freeness]
\label{lem:tower-torsionfree}
If every vertex group of a graph of groups
\(\mathcal G\) is torsion-free, then its fundamental
group \(\pi_1(\mathcal G)\) is torsion-free.
In particular, all exterior groups in our
construction are torsion-free.
\end{lemma}

\begin{proof}
Put \(H=\pi_1(\mathcal G)\), and let \(T\) be the
associated Bass--Serre tree. The action of \(H\)
on \(T\) is without inversions, and the stabilizer
of each vertex is conjugate to a vertex group
of \(\mathcal G\).

Suppose \(x\in H\) has finite order. The finite
cyclic group \(\langle x\rangle\) fixes a vertex
\(v\in T\)
\cite[Example~I.6.3.1]{SerreTrees}.
Hence
\[
x\in\Stab_H(v)=gG_ug^{-1}
\]
for some vertex group \(G_u\) and \(g\in H\).
Since \(g^{-1}xg\) has finite order and \(G_u\)
is torsion-free, we have \(g^{-1}xg=1\).
Thus \(x=1\), proving that \(H\) is torsion-free.

In our construction, the preliminary group \(A_0\)
is an iterated amalgam of torsion-free knot groups
along infinite cyclic subgroups. Each subsequent
exterior group is obtained by an amalgam
\[
G'=G*_{\langle\mu,h\rangle}
\bigl(K_J\times\langle s\rangle\bigr)
\]
with injective edge maps, where \(K_J\) is a
torsion-free knot group and
\(\langle s\rangle\cong\mathbb Z\).
The conclusion follows inductively.
\end{proof}

\begin{lemma}[Reflection exclusion]
\label{lem:reflection-exclusion}
Let a torsion-free group \(G\) act on a tree \(T\), and suppose
\[
 H_1(G;\mathbb Z)=\mathbb Z\langle[\mu]\rangle,
 \qquad \Fix_G^{\mathrm{pt}}(L)=\langle\mu\rangle
\]
for a line \(L\subset T\).
Then \(\Stab_G(L)\) centralizes \(\mu\) and contains no reflection
of \(L\).
\end{lemma}

\begin{proof}
The pointwise stabilizer is the kernel of
\(\Stab_G(L)\to\Isom(L)\).
Every \(g\in\Stab_G(L)\) therefore sends \(\mu\) to
\(\mu^{\pm1}\) by conjugation. Since conjugation preserves the homology class of \(\mu\),
and \([\mu]\) has infinite order, the sign must be positive.
Hence \(g\) commutes with \(\mu\).
If \(g\) reflected \(L\), its square would fix \(L\) pointwise,
giving \(g^2=\mu^k\).
Write \([g]=n[\mu]\).  Abelianization gives \(k=2n\), and hence
\[
 (\mu^{-n}g)^2=\mu^{-2n}g^2=1.
\]
But \(\mu^{-n}g\) still reflects \(L\), so it is nontrivial.
This contradicts torsion-freeness.
\end{proof}

For related centralizer arguments using translations
along tree axes, see \cite[Lemma~7]{CrispArtin}.

\begin{lemma}[Axis criterion for the centralizer--transporter property]
\label{lem:clean-axis-criterion}
Let a torsion-free group \(G\) act on a tree \(T\).
Suppose that \(h\in G\) is hyperbolic with axis \(L\), and
\[
H_1(G;\mathbb Z)=\mathbb Z\langle[\mu]\rangle,
\qquad
\Fix_G^{\mathrm{pt}}(L)=\langle\mu\rangle.
\]
Assume that at least one of the following holds:
\begin{enumerate}[label=\textup{(\roman*)}]
\item For every \(g\in\Stab_G(L)\), there exists
\(n\in\mathbb Z\) such that
\[
gv=h^nv\qquad\text{for every }v\in V(L).
\]
\item The element \(h\) is null-homologous and is not
a proper power in \(G\).
\end{enumerate}
Then
\[
\Stab_G(L)=A:=\langle\mu,h\rangle\cong\mathbb Z^2,
\]
and \(A\) has the centralizer--transporter property
in \(G\) relative to \(\mu\).
\end{lemma}

\begin{proof}
By \cref{lem:reflection-exclusion}, every
\(g\in\Stab_G(L)\) commutes with \(\mu\) and
preserves each of the two ends of \(L\).
Thus every such \(g\) acts on \(L\) by a
translation, possibly trivial.  Thus \([\mu,h]=1\).
If \(\mu^ah^b=1\), its action on \(L\) gives \(b=0\), and the
infinite order of \([\mu]\) then gives \(a=0\).
Hence \(A\cong\mathbb Z^2\).

For the second alternative, orient \(L\) so that
\(h\) acts by translation in the positive direction.
By \cref{lem:reflection-exclusion}, every element
of \(\Stab_G(L)\) acts on \(L\) by a translation.
Choose \(u\in\Stab_G(L)\) whose positive translation
length is the smallest among all such elements. Then the translation induced by \(h\) is a positive integer multiple of it.
Thus, for some \(n\ge1\), the element \(u^nh^{-1}\) fixes \(L\)
pointwise, so \(u^nh^{-1}=\mu^k\) for some \(k\in\mathbb Z\).
Since \(H_1(G;\mathbb Z)\) is generated by \([\mu]\),
we can also write \([u]=r[\mu]\) for some \(r\in\mathbb Z\). Taking homology classes in \(u^nh^{-1}=\mu^k\), we obtain
\[
n[u]-[h]=k[\mu].
\]
Since \([u]=r[\mu]\) and \([h]=0\), this becomes
\[
nr[\mu]=k[\mu].
\]
As \([\mu]\) has infinite order, we conclude that \(k=nr\). The element \(u\) commutes with \(\mu\), so
\(
 (\mu^{-r}u)^n=h.
\)
If \(n>1\), the equality
\[
h=(\mu^{-r}u)^n
\]
expresses \(h\) as a proper power of an element of \(G\),
contrary to our assumption. Therefore \(n=1\).
Thus in either case the translation of \(h\) generates the
image of the axis stabilizer.  Its kernel is \(\langle\mu\rangle\),
so every element of \(\Stab_G(L)\) lies in \(A\).

Write \(x=\mu^ah^b\), where \(b\ne0\).  It translates \(L\)
nontrivially, so \(\Ax_T(x)=L\).
If \(g^{-1}xg=\mu^ch^d\in A\), this conjugate is hyperbolic. Therefore \(d\ne0\), and its axis is also \(L\).
But that axis is \(g^{-1}L\).  Hence \(g\in\Stab_G(L)=A\).
Conversely, since \(A\) is abelian and \(x\in A\),
every \(g\in A\) satisfies
\(
g^{-1}xg=x\in A.
\)
Hence \(A\subset\Trans_G(x,A)\), leading to
\[
\Trans_G(x,A)=A.
\]
Moreover, since every element centralizing \(x\)
also belongs to its transporter, we have
\[
A\subset C_G(x)\subset\Trans_G(x,A)=A.
\]
Thus \(C_G(x)=A\), as required.
\end{proof}

For a vertex \(v\in T\), write
\(G\cdot v=\{gv:g\in G\}\) for its \(G\)-orbit.

\begin{lemma}[Conjugacy via a subtree]
\label{lem:surface-subtree-conjugacy}
Let \(G\) act on a tree \(T\), and let \(S\le G\).
Suppose that \(\mathcal T\subset T\) is a nonempty
\(S\)-invariant subtree and that \(x\in S\) is
hyperbolic with axis \(L\). Assume also that \(\mu\in C_G(S)\). Suppose there is a vertex \(v\in L\) satisfying
the following conditions:
\begin{enumerate}[label=\textup{(\roman*)}]
\item \(\mathcal T\cap(G\cdot v)=S\cdot v\).
\item Exactly two edges of \(\mathcal T\) meet at \(v\).
\item The subgroup of \(G\) that preserves these two edges as an unordered pair is
\(\langle\mu,h_v\rangle\) for some \(h_v\in S\).
\end{enumerate}
Then, for any \(g\in G\),
\[
g\langle x\rangle g^{-1}\le S
\quad\Longrightarrow\quad
g\langle x\rangle g^{-1}
=s\langle x\rangle s^{-1}
\]
for some \(s\in S\).
\end{lemma}

\begin{proof}
We first recall that every nonempty subtree invariant
under a hyperbolic element contains its axis.
Indeed, let \(y\) be hyperbolic, and let \(q\) be
a vertex of a \(y\)-invariant subtree.
If \(p\) is the closest vertex on the axis of \(y\)
to \(q\), then the geodesic joining \(y^{-n}q\)
to \(y^nq\) contains the segment joining
\(y^{-n}p\) to \(y^np\).
As \(n\) increases, these segments exhaust the
axis of \(y\). Since the invariant subtree
contains all these geodesics, it contains the
entire axis.

Now suppose that \(g\langle x\rangle g^{-1}\le S\).
Both \(x\) and \(gxg^{-1}\) belong to \(S\) and
are hyperbolic, with axes \(L\) and \(gL\),
respectively. Since \(\mathcal T\) is
\(S\)-invariant, the preceding observation gives
\[
L\subset\mathcal T,
\qquad
gL\subset\mathcal T.
\]

In particular, \(gv\in\mathcal T\).
Since \(gv\) also belongs to the \(G\)-orbit of
\(v\), condition~\textup{(i)} gives an element
\(s_1\in S\) such that
\(
s_1v=gv.
\)
Set \(b=s_1^{-1}g\). Then \(bv=v\), and
\[
bL=s_1^{-1}gL\subset\mathcal T,
\]
because \(s_1\in S\) preserves \(\mathcal T\).

Thus \(L\) and \(bL\) are both bi-infinite
geodesics in \(\mathcal T\) passing through \(v\).
By condition~\textup{(ii)}, exactly two edges
of \(\mathcal T\) meet at \(v\), so both lines
must pass through these same two edges.
Since \(b\) fixes \(v\) and carries \(L\) to
\(bL\), it preserves this pair of edges.
Condition~\textup{(iii)} therefore gives
\(
b\in\langle\mu,h_v\rangle.
\)

Since \(h_v\in S\) and \(\mu\) centralizes \(S\), \(\mu\) and \(h_v\) commute.
We may therefore write
\(
b=\mu^m h_v^n
\)
for some \(m,n\in\mathbb Z\).
Put \(s=s_1h_v^n\in S\).
As \(\mu\) also commutes with \(s_1\), we have
\[
g=s_1b=s_1\mu^mh_v^n=\mu^ms.
\]
Finally, \(\mu\) centralizes \(x\in S\), so
\[
gxg^{-1}
=\mu^msxs^{-1}\mu^{-m}
=sxs^{-1}.
\]
Consequently,
\[
g\langle x\rangle g^{-1}
=s\langle x\rangle s^{-1},
\]
as desired.
\end{proof}

\section{The first surgery curve and its Bass--Serre axis}
\label{sec:initialcarriers}
Recall that \(F_0\) is the surface constructed in
\cref{sec:seed}, with exterior group
\(G_0=\pi_1(S^4\setminus\operatorname{int}\nu F_0,*)\)
and push-off homomorphism \(\rho_0\).
By \cref{prop:even-seed-injection,prop:odd-seed-injection},
\(\rho_0\) is injective.
Write \(S_0=\rho_0(\pi_1(\Sigma_g,p))\).
By \cref{prop:even-seed-injection,prop:odd-seed-injection},
the push-off homomorphism \(\rho_0\) is injective.
Write \(S_0=\rho_0(\pi_1(\Sigma_g,p))\).

In this section, we construct a nonseparating curve \(d_1\) and establish
two properties of its push-off \(h_1=\rho_0(d_1)\).
First, \(A_1=\langle\mu,h_1\rangle\) satisfies the
centralizer--transporter property of \cref{def:power-clean},
using the axis criterion of \cref{lem:clean-axis-criterion}.
Second, every \(G_0\)-conjugate of \(\langle h_1\rangle\)
contained in \(S_0\) is already conjugate to it within
\(S_0\), using \cref{lem:surface-subtree-conjugacy}.
These properties provide the starting point for the
induction in \cref{sec:latercarriers}.
We treat even and odd genus separately.

The following condition is a cyclic-subgroup version
of conjugacy stability; compare
\cite[Lemma~1]{FernandezMartinezSerbin}.

\begin{definition}[Cyclic-subgroup conjugacy condition]
\label{def:rooted-isolation}
Let \(S\le G\) and let \(1\ne h\in S\).  We say that the pair
\((S,\langle h\rangle)\) satisfies the cyclic-subgroup conjugacy
condition in \(G\) if
\[
a\langle h\rangle a^{-1}\le S
\quad\Longrightarrow\quad
 a\langle h\rangle a^{-1}
 =s\langle h\rangle s^{-1}
\text{ for some }s\in S.
\]
Thus every ambient conjugate of \(\langle h\rangle\) contained in
\(S\) is already conjugate to \(\langle h\rangle\) by an element of
\(S\).
\end{definition}

For a simplicial tree \(T\), we write \(\operatorname{sd}(T)\) for its
barycentric subdivision, obtained by inserting one vertex in the
interior of every edge.

\subsection{Even genus}

Assume \(g=2r\), so \(P=\Sigma_{r,1}\).  Fix the basis
\[
U_0=\langle a_1,b_1,\ldots,a_r,b_r\rangle
\]
and the boundary word
\[
w=\prod_{j=1}^r[a_j,b_j]
\]
from \cref{eq:even-seam-word}.  Choose distinct points
\(q_a,q_b\in B_1\setminus\{q_0\}\).  Choose embedded boundary paths
\(\zeta:q_0\to q_a\) and 
\(\xi:q_a\to q_b\), with disjoint interiors.  Denote their copies
in \(P^+\) and \(P^-\) by \(\zeta^\pm\) and \(\xi^\pm\) respectively, and put
\[
\lambda_a=q_a\times I:q_a\times\{1\}\longrightarrow
q_a\times\{0\},
\qquad
\lambda_b=q_b\times I:q_b\times\{1\}\longrightarrow
q_b\times\{0\}.
\]

Choose embedded proper arcs
\[
\alpha_+:q_a\times\{1\}\longrightarrow q_b\times\{1\},
\qquad
\alpha_-:q_b\times\{0\}\longrightarrow q_a\times\{0\},
\]
such that, together with the boundary paths fixed above, the based
loops defined below correspond to the basis elements
\(a_1\) and \(b_1\), respectively, under the push-off homomorphisms
from \(P^+\) and \(P^-\).  Choose \(\alpha_+\) to be
nonseparating.  Define
\begin{equation}
\label{eq:xy-arcs}
\begin{aligned}
\bar x
&=
\zeta^+\alpha_+(\xi^+)^{-1}(\zeta^+)^{-1}
\in\pi_1(P^+,p),\\
\bar y
&=
\zeta^-\xi^-\alpha_-(\zeta^-)^{-1}
\in\pi_1(P^-,p_-),\\
x
&=\rho_0(\bar x)=a_1,\\
y
&=\rho(\lambda_-\bar y\lambda_-^{-1})=b_1.
\end{aligned}
\end{equation}
Here \(\bar y\) corresponds to \(\bar b_1\) under the
product identification \(P^-\cong P^+\).
By \cref{prop:horizontal-peripheral}, their based
push-offs agree:
\[
\rho(\lambda_-\bar y\lambda_-^{-1})
=\rho(\bar b_1)=b_1.
\] In particular,
\[
x,y\in U_0\setminus\langle w\rangle,
\]
because \(w\) is trivial in the abelianization of \(U_0\), whereas
\(a_1\) and \(b_1\) are not. The arcs and their boundary closures are illustrated in
\cref{fig:even-generator-arcs}.

\begin{figure}[htbp]
  \centering
  \includegraphics[width=\linewidth]{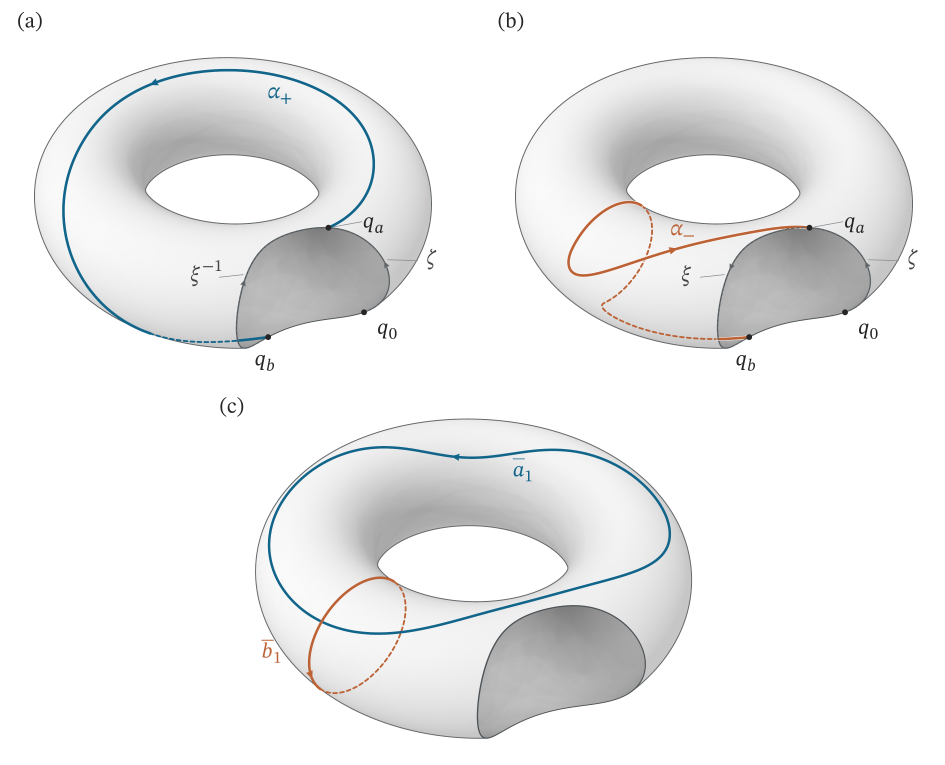}
  \caption{Generator arcs on the first handle of $P$, shown for $r=1$.
  (a) The closure $\alpha_+(\xi^+)^{-1}$, based along $\zeta^+$,
  represents $\bar a_1$.
  (b) The closure $\xi^-\alpha_-$, based along $\zeta^-$,
  represents $\bar b_1$.
  (c) Representatives of these loops intersect transversely once.}
  \label{fig:even-generator-arcs}
\end{figure}

Define the oriented curve
\begin{equation}
\label{eq:even-d1}
d_1=
\bigl(\alpha_+\cup\lambda_b\cup\alpha_-
\cup\lambda_a^{-1}\bigr)_{\mathrm{sm}},
\end{equation}
where the subscript denotes smoothing the four corners. Choose the smoothing so that
\(q_a\times\{1\}\) remains a point of the smoothed curve.  Since
\(q_a\ne q_0\), the resulting curve is disjoint from \(p\).  We use
\(\zeta^+\) as the path from \(p\) to
\(q_a\times\{1\}\), thereby fixing the based element denoted by
\(d_1\). The construction of $d_1$ is shown in
\cref{fig:even-first-curve}.

\begin{figure}[htbp]
  \centering
  \includegraphics[width=\linewidth]{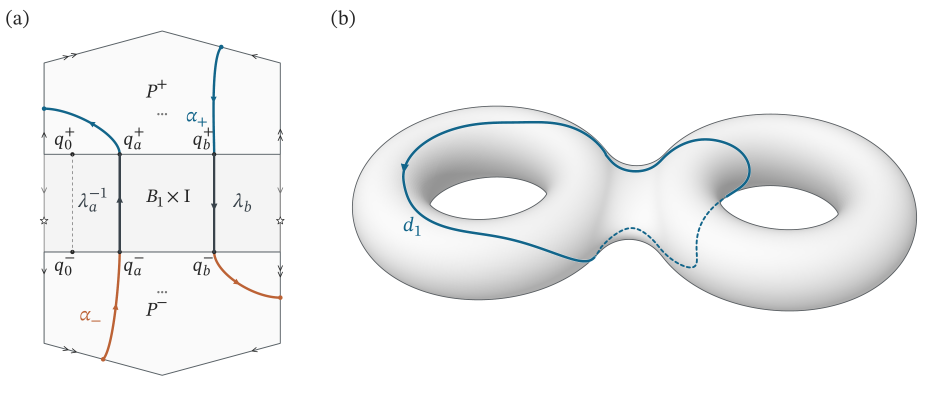}
  \caption{The first surgery curve in even genus.
  (a) The decomposition $P^+\cup(B_1\times I)\cup P^-$. Matching arrows indicate the polygon side identifications. The two stared sides are glued together to form \(B_1\times I\).
(b) The curve \(d_1\) is shown on
\(\Sigma_2=\partial(\Sigma_{1,1}\times I)\). For \(r>1\), the
additional handles are omitted because they do
not meet \(d_1\).}
  \label{fig:even-first-curve}
\end{figure}

\begin{lemma}
\label{lem:even-d1-geometry}
The curve \(d_1\) in \eqref{eq:even-d1} is embedded, essential, and
nonseparating.  Its push-off image is 
\begin{equation}
\label{eq:even-h1-word}
 h_1=\rho_0(d_1)
 =x\,\ell_1^V\,y\,(\ell_1^V)^{-1}.
\end{equation}
\end{lemma}

\begin{proof}
The interiors of \(\alpha_+\), \(\alpha_-\),
\(\lambda_a\), and \(\lambda_b\) are pairwise disjoint:
the first two lie in the distinct subsurfaces \(P^+\) and \(P^-\),
and the last two lie in distinct fibers of \(B_1\times I\).  Their
union is therefore an embedded circle.
Because \(\alpha_+\) is a nonseparating proper arc, there is a simple
closed curve in the interior of \(P^+\) meeting \(\alpha_+\) once and
disjoint from the other three pieces.  It meets \(d_1\) once, so
\(d_1\) is nonseparating and hence essential.

Define paths from \(p\) to \(p_-\) by
\[
\Lambda_a
=
\zeta^+\lambda_a(\zeta^-)^{-1}
\]
and
\[
\Lambda_b
=
\zeta^+\xi^+\lambda_b
(\xi^-)^{-1}(\zeta^-)^{-1}.
\]
The rectangles \(\zeta\times I\) and
\((\zeta\xi)\times I\) give homotopies relative to endpoints from
both of these paths to \(\lambda_-\).  Therefore
\[
 d_1
 =
 \bar x\,\Lambda_b\,\bar y\,\Lambda_a^{-1}
 =
 \bar x\,\lambda_-\,\bar y\,\lambda_-^{-1}
\]
in \(\pi_1(\Sigma_{2r},p)\).  Applying
\eqref{eq:vertical-plus-restriction} and
\eqref{eq:vertical-minus-restriction}, together with
\eqref{eq:xy-arcs}, gives
\[
\rho_0(d_1)
=
 x\,\ell_1^V\,y\,(\ell_1^V)^{-1}.
\]
\end{proof}

Recall that
\begin{equation}
\label{eq:even-initial-tree-splitting}
G_0=A_0*_{E}C_1^V,
\qquad
E=\langle\mu,w\rangle,
\qquad
C_1^V=K_1^V\times\langle z_1\rangle,
\end{equation}
where \(z_1=w\) in the amalgam.  Let \(T_0\) be the Bass--Serre tree associated with
the splitting \(G_0=A_0*_{E}C_1^V\). Explicitly,
\[
V(T_0)=G_0/A_0\sqcup G_0/C_1^V,
\qquad
E(T_0)=G_0/E.
\]
For each \(g\in G_0\), the edge \(gE\) joins the
vertices \(gA_0\) and \(gC_1^V\).
The group \(G_0\) acts on \(T_0\) by left multiplication.
The stabilizers of these vertices and their connecting
edge are, respectively,
\[
gA_0g^{-1},\qquad
gC_1^Vg^{-1},\qquad
gEg^{-1}.
\]
The quotient graph \( T_0/G_0\) consists
of a single edge joining two vertices.

\begin{lemma}[The axis of \(h_1\) in even genus]
\label{lem:even-clean-axis}
The word \(h_1=x\ell_1^Vy(\ell_1^V)^{-1}\) is cyclically reduced
of syllable length four in \eqref{eq:even-initial-tree-splitting},
and is not a proper power in \(G_0\).  The subgroup
\(A_1=\langle\mu,h_1\rangle\cong\mathbb Z^2\) has the
centralizer--transporter property in \(G_0\) relative to \(\mu\).
\end{lemma}

\begin{proof}
Since \(U_0\cap E=\langle w\rangle\) and
\(x,y\notin\langle w\rangle\), we have
\(x,y\in A_0\setminus E\).
Under the projection
\(C_1^V=K_1^V\times\langle z_1\rangle\to K_1^V\),
the edge group \(E=\langle m_1^V,z_1\rangle\)
maps to \(\langle m_1^V\rangle\).
Since
\[
\langle m_1^V,\ell_1^V\rangle\cong\mathbb Z^2,
\]
neither \(\ell_1^V\) nor its inverse belongs to \(E\).
Thus the four syllables of
\(h_1=x\ell_1^Vy(\ell_1^V)^{-1}\)
alternate between the two vertex groups, with
the first and last belonging to different groups.
Hence \(h_1\) is cyclically reduced of syllable
length four.

Suppose \(h_1=v^n\) with \(n>1\).  Then \(v\) is hyperbolic on
\(L_1\) and \(4=n\|v\|_{T_0}\).
The action preserves both vertex types, so positive translation lengths
are even.  Thus \(n=2\) and \(\|v\|_{T_0}=2\).

The axis of
\(h_1=x\ell_1^Vy(\ell_1^V)^{-1}\)
passes through two \(C_1^V\)-vertices in each
period of length four. At the first vertex, we
may represent the incoming and outgoing edges
by \(E\) and \(\ell_1^VE\), respectively.
At the second vertex, they are represented by
\(E\) and \((\ell_1^V)^{-1}E\).
These two pairs are therefore determined,
up to translation by \(C_1^V\), by the
double cosets
\[
E\ell_1^VE
\quad\text{and}\quad
E(\ell_1^V)^{-1}E.
\]

If \(h_1=v^2\), then \(v\) translates the
axis by two edges and carries the first
pair to the second, preserving their order.
The two double cosets must therefore agree:
\[
E\ell_1^VE=E(\ell_1^V)^{-1}E.
\]
Since \(\ell_1^V\) commutes with
\(E=\langle\mu,w\rangle\), this implies
\((\ell_1^V)^2\in E\). Projecting
\(C_1^V=K_1^V\times\langle z_1\rangle\)
onto \(K_1^V\) gives
\[
(\ell_1^V)^2\in\langle m_1^V\rangle,
\]
contrary to
\(\langle m_1^V,\ell_1^V\rangle
\cong\mathbb Z^2\).
Thus \(h_1\) is not a proper power in \(G_0\).

At the base \(A_0\)-vertex on \(L_1\), consecutive axis edges have
stabilizers \(E\) and \(xEx^{-1}\).  Since \(x\in U_0\) and
\(\langle w\rangle\) is malnormal in \(U_0\),
\[
 E\cap xEx^{-1}
 =\langle\mu\rangle\times
   \bigl(\langle w\rangle\cap x\langle w\rangle x^{-1}\bigr)
 =\langle\mu\rangle.
\]
Every syllable of \(x\ell_1^Vy(\ell_1^V)^{-1}\) commutes with
\(\mu\).  Thus \(\mu\) fixes every edge along a fundamental
axis segment and all its \(h_1\)-translates.  The preceding
intersection gives the reverse inclusion, proving
\begin{equation}
\label{eq:even-pointwise-axis}
 \Fix_{G_0}^{\mathrm{pt}}(L_1)=\langle\mu\rangle.
\end{equation}
Torsion-freeness and meridional homology are
\cref{lem:tower-torsionfree,prop:early-homology}; the latter also gives
\([h_1]=0\).  The rootless alternative of
\cref{lem:clean-axis-criterion} now gives
\begin{equation}
\label{eq:even-full-axis}
 \Stab_{G_0}(L_1)=\langle\mu,h_1\rangle
\end{equation}
and the stated edge property.
\end{proof}

\begin{proposition}[Conjugacy in the even-genus surface subgroup]
\label{prop:even-rooted-isolation}
The pair \((S_0,\langle h_1\rangle)\) satisfies the cyclic-subgroup
conjugacy condition in \(G_0\).
\end{proposition}

\begin{proof}
By \eqref{eq:even-S0}, the surface subgroup has the splitting
\[
S_0
=
U_0*_{\langle w\rangle}
\ell_1^VU_0(\ell_1^V)^{-1}.
\]
Let \(T_{\Sigma}\) be its Bass--Serre tree.  There is an
\(S_0\)-equivariant simplicial embedding
\[
\operatorname{sd}(T_{\Sigma})\longrightarrow T_0
\]
obtained by mapping a chosen subdivided edge to the two-edge path
represented by
\[
E,
\qquad
\ell_1^VE.
\]
Indeed, the two endpoints map to the vertices \(A_0\) and
\(\ell_1^VA_0\), and the inserted barycentric vertex maps to
\(C_1^V\).  At an \(A_0\)-type endpoint, local injectivity is exactly
\(U_0\cap E=\langle w\rangle\); the conjugate calculation applies at
the other endpoint.  At the inserted vertex it follows from
\(\ell_1^V\notin E\).  Denote the image by \(\mathcal T_0\).
A locally injective simplicial map between trees is injective: the
geodesic between two vertices with the same image would give a
nontrivial closed reduced path in a tree.  Thus \(\mathcal T_0\) is an
\(S_0\)-invariant embedded subtree.

At each \(C_1^V\)-type vertex of \(\mathcal T_0\), exactly two edges
of \(\mathcal T_0\) are incident.  At the base vertex they are
represented by \(E\) and \(\ell_1^VE\).  Their setwise stabilizer in
\(C_1^V\) is exactly \(E\).  Every element of \(E\) fixes both edge
cosets because \(\ell_1^V\) centralizes \(E\).  Conversely, an
element fixing \(E\) belongs to \(E\).  If an element exchanged the
two cosets, it would belong to \(\ell_1^VE\) and would carry
\(\ell_1^VE\) to \((\ell_1^V)^2E\); exchange would therefore imply
\((\ell_1^V)^2\in E\), which is impossible after projection to the rank-two knot peripheral
subgroup.

Since \(h_1\in S_0\) is hyperbolic, \(L_1\subset\mathcal T_0\).
Choose \(v=s_vC_1^V\in L_1\), with \(s_v\in S_0\).  The
\(C_1^V\)-type vertices form one \(S_0\)-orbit, and at \(v\) the
preceding calculation gives valence two and pair stabilizer
\[
s_vEs_v^{-1}=\langle\mu,s_vws_v^{-1}\rangle\cong\mathbb Z^2,
\qquad s_vws_v^{-1}\in S_0.
\]
The meridian centralizes \(S_0\).  The orbit, valence, and pair-stabilizer
conditions of \cref{lem:surface-subtree-conjugacy} now hold, finishing the proof.
\end{proof}

\subsection{Odd genus}

Assume \(g=2r+1\), so \(P=\Sigma_{r,2}\).
Recall the based loop
\[
\tau=\beta_2^+\lambda_2(\beta_2^-)^{-1}\lambda_-^{-1}
\]
from \eqref{eq:odd-surface-cycle}.
This loop traverses the boundary of the product disk
\[
\Delta=\beta_2\times I\subset P\times I,
\]
starting at \(p\). Isotope \(\partial\Delta\) slightly
off \(p\), and denote the resulting oriented curve by
\(d_1\). Using the path traced by the point initially
at \(p\) as the basing path, we have
\(d_1=\tau\) in \(\pi_1(\Sigma_g,p)\).

The arc $\beta_2$ and product disk $\Delta$ are illustrated in
\cref{fig:odd-product-disk}.

\begin{figure}[htbp]
  \centering
  \includegraphics[width=\linewidth]{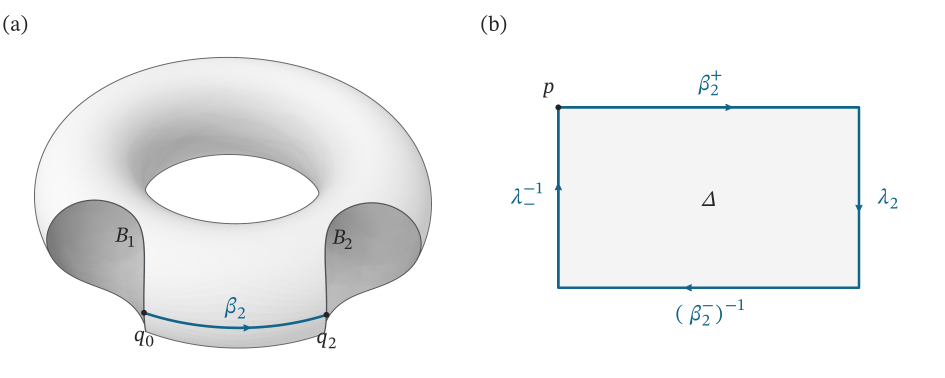}
  \caption{The product disk that defines the first odd-genus surgery curve.
  (a) The arc $\beta_2$ joins the two boundary components $B_1,B_2$
  of $P$, shown for $P=\Sigma_{1,2}$. 
  (b) The parameter square of $\Delta=\beta_2\times I$, with its
  four oriented boundary sides. A small isotopy of
  $\partial\Delta$ off $p$ gives $d_1$.}
  \label{fig:odd-product-disk}
\end{figure}

Since \(\beta_2\) joins the two boundary components
of \(P\), cutting \(P\) along \(\beta_2\) leaves it
connected. Hence cutting \(P\times I\) along \(\Delta\)
gives the connected manifold
\((P\setminus\!\setminus\beta_2)\times I\).
Thus \(\Delta\) is nonseparating, so its boundary
and hence \(d_1\) are nonseparating in \(\Sigma_g\)
and therefore essential.

The resulting curve on $\partial(P\times I)$ is shown in
\cref{fig:odd-first-curve}.

\begin{figure}[htbp]
  \centering
  \includegraphics[width=\linewidth]{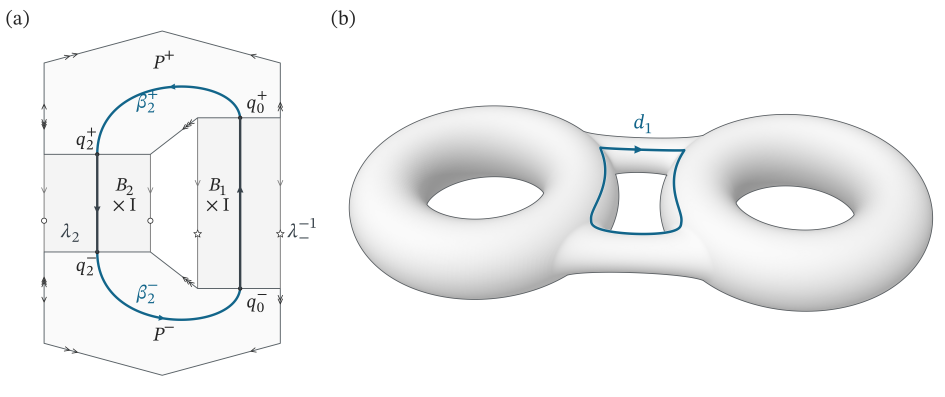}
  \caption{The first surgery curve in odd genus, drawn for $r=1$.
  (a) The two copies $P^\pm$ joined by the annuli $B_1\times I$
  and $B_2\times I$. Matching arrows indicate the polygon side identifications.
The two sides marked with circles are glued together
to form \(B_2\times I\), and the two sides marked
with stars are glued together to form \(B_1\times I\). Panel (b) shows \(d_1\) on
\(\Sigma_3=\partial(\Sigma_{1,2}\times I)\).
For higher odd genus, the additional handles are
omitted.}
  \label{fig:odd-first-curve}
\end{figure}

By \cref{prop:odd-seed-injection}, the based push-off image is 
\begin{equation}
\label{eq:odd-h1}
 h_1=\rho_0(d_1)
 =\rho_0(\tau)
 =\ell_2^V(\ell_1^V)^{-1}.
\end{equation}
By \eqref{eq:odd-seed-group}, the exterior group
after the two vertical-boundary surgeries has the
amalgamated-product decomposition
\begin{equation}
\label{eq:odd-initial-tree-splitting}
G_0=
\bigl(A_0*_{E_1}C_1^V\bigr)*_{E_2}C_2^V,
\qquad
E_i=\langle\mu,w_i\rangle.
\end{equation}

We regard this as the fundamental group of the three-vertex tree of
groups with central vertex group \(A_0\), leaf vertex groups
\(C_1^V,C_2^V\), and edge groups \(E_1,E_2\).  Let \(T_0\) be its
Bass--Serre tree.  The quotient graph \(T_0/G_0\) has three vertices,
corresponding to \(A_0,C_1^V,C_2^V\), and two edges,
corresponding to \(E_1,E_2\). Each edge joins the
\(A_0\)-vertex to the corresponding \(C_i^V\)-vertex.
The stabilizers of vertices and edges in \(T_0\)
are conjugates of the corresponding vertex and
edge groups.

\begin{lemma}[The axis of \(h_1\) in odd genus]
\label{lem:odd-clean-axis}
The element \(h_1\) is hyperbolic on \(T_0\), has translation length
four, and is not a proper power in \(G_0\).  The subgroup
\(A_1=\langle\mu,h_1\rangle\cong\mathbb Z^2\) has the
centralizer--transporter property in \(G_0\) relative to \(\mu\).
\end{lemma}

\begin{proof}
Recall that \(T_0\) is the Bass--Serre tree of the
graph of groups with central vertex group \(A_0\),
leaf vertex groups \(C_1^V,C_2^V\), and edge groups
\(E_1,E_2\). To represent
\(
h_1=\ell_2^V(\ell_1^V)^{-1},
\)
we start at the central vertex of the quotient graph.
First cross the edge of type \(E_2\) to the
\(C_2^V\)-vertex, apply \(\ell_2^V\), and return
to the central vertex along the same quotient edge.
Then cross the edge of type \(E_1\) to the
\(C_1^V\)-vertex, apply \((\ell_1^V)^{-1}\),
and return to the central vertex.

This path has four edges and is reduced.
Indeed, the two traversals through each leaf
vertex cannot cancel because
\(\ell_i^V\notin E_i\). At the central vertex,
the consecutive edges have different types,
so they cannot cancel either. The same is true
when the path is repeated, hence the word is
cyclically reduced. It follows that \(h_1\) acts hyperbolically
on \(T_0\), with translation length four.
Write \(L_1=\Ax_{T_0}(h_1)\). Each fundamental
segment of length four contains one vertex
of type \(C_1^V\) and one of type \(C_2^V\).

If \(h_1=v^n\) with \(n>1\), then \(v\) is hyperbolic on the same
axis and
\[
4=n\|v\|_{T_0}.
\]
The action preserves the central and the two leaf vertex types, so
every positive translation length is even.  Hence \(n=2\) and
\(\|v\|_{T_0}=2\).  The axis of a length-two translation passes from
a central vertex through a leaf vertex and back to a central vertex.
Because the group action preserves each leaf-vertex orbit, every leaf
vertex on that axis has the same type.  This contradicts the presence
of both \(C_1^V\)- and \(C_2^V\)-type vertices on \(L_1\).  Thus
\(h_1\) is not a proper power.

At the base \(A_0\)-vertex on \(L_1\), the consecutive axis edges
have stabilizers \(E_1\) and \(E_2\).  The direct-product description of the meridian
centralizer gives
\[
E_1\cap E_2
=
\langle\mu\rangle\times
\bigl(\langle w_1\rangle\cap\langle w_2\rangle\bigr)
=
\langle\mu\rangle.
\]
Here the intersection is trivial by the boundary noncommensurability
proved in \cref{sec:seed}. At the base $C_i^V$-vertex on $L_1$,
the two consecutive axis edges have stabilizers
\[
E_i
\quad\text{and}\quad
\ell_i^VE_i(\ell_i^V)^{-1}=E_i.
\]
The meridian commutes with \(\ell_1^V\) and \(\ell_2^V\) and belongs
to both \(E_1\) and \(E_2\), so it fixes every edge of a fundamental
segment and every \(h_1\)-translate of that segment.  Hence it fixes
\(L_1\) pointwise. Conversely, an element fixing $L_1$ pointwise
fixes the two axis edges meeting the base
$A_0$-vertex, so it belongs to
$E_1\cap E_2=\langle\mu\rangle$.

The preceding paragraphs prove that \(h_1\) is hyperbolic, is not a
proper power, and has pointwise axis stabilizer \(\langle\mu\rangle\).
Together with torsion-freeness and
\(H_1(G_0;\mathbb Z)=\mathbb Z\langle[\mu]\rangle\), \([h_1]=0\)
from \cref{lem:tower-torsionfree,prop:early-homology}, these are the
rootless alternative of \cref{lem:clean-axis-criterion}.  Hence
\(\Stab_{G_0}(L_1)=\langle\mu,h_1\rangle\), and the stated edge
property follows.
\end{proof}

\begin{proposition}[Conjugacy in the odd-genus surface subgroup]
\label{prop:odd-rooted-isolation}
The pair \((S_0,\langle h_1\rangle)\) satisfies the cyclic-subgroup
conjugacy condition in \(G_0\).
\end{proposition}

\begin{proof}
Let \(T_{\Sigma}\) be the Bass--Serre tree of the two-vertex,
two-edge graph-of-groups decomposition of the surface group determined
by
\[
\Sigma_{2r+1}=P^+\cup_{(B_1\sqcup B_2)\times I}P^-.
\]
Through the injective homomorphism \(\rho_0\), its two vertex groups
are represented by \(U_0\) and
\(\ell_1^VU_0(\ell_1^V)^{-1}\), and its edge groups of types
\(B_i\) are represented by \(\langle w_i\rangle\).  There is an \(S_0\)-equivariant
simplicial embedding
\[
\operatorname{sd}(T_{\Sigma})\longrightarrow T_0
\]
defined on chosen edge lifts as follows.  The edge associated to
\(B_1\times I\) maps to the two-edge path represented by
\[
E_1,
\qquad
\ell_1^VE_1,
\]
through the vertex \(C_1^V\).  The edge associated to
\(B_2\times I\) maps to the path represented by
\[
E_2,
\qquad
\ell_2^VE_2,
\]
through \(C_2^V\).  The endpoint of the second path is
\(\ell_2^VA_0\), which equals
\(\rho_0(\tau)\ell_1^VA_0\) by
\eqref{eq:odd-stable-letter}; hence the two chosen paths have exactly
the endpoints prescribed by the surface graph of groups.

At a lift of the first surface vertex, local injectivity for edges of
type \(i\) follows from
\[
U_0\cap E_i=\langle w_i\rangle.
\]
At a lift of the second surface vertex, the same assertion is conjugated
by \(\ell_1^V\).  For the edge of type \(B_2\), this uses
\[
h_1^{-1}E_2h_1
=
\ell_1^VE_2(\ell_1^V)^{-1},
\]
which follows from
\(h_1=\ell_2^V(\ell_1^V)^{-1}\) and the fact that
\(\ell_2^V\) centralizes \(E_2\).  Edges of different types cannot
coincide because they belong to different \(G_0\)-orbits in
\(T_0\).  At a newly inserted vertex,
local injectivity follows from \(\ell_i^V\notin E_i\).  Thus the map is an embedding.  Denote its image by
\(\mathcal T_0\).  The image is an \(S_0\)-invariant subtree.

At the base $C_i^V$-vertex of $\mathcal T_0$,
the two incident edges are $E_i$ and
$\ell_i^VE_i$. Their setwise stabilizer in
$C_i^V$ is exactly $E_i$. Indeed, \(E_i\) fixes both cosets
because \(\ell_i^V\) centralizes \(E_i\).  An element exchanging them
would imply \((\ell_i^V)^2\in E_i\); projection to \(K_i^V\) would
then give \((\ell_i^V)^2\in\langle m_i^V\rangle\), contrary to
\(\langle m_i^V,\ell_i^V\rangle\cong\mathbb Z^2\).

Since \(h_1\in S_0\) is hyperbolic, \(L_1\subset\mathcal T_0\)
and a fundamental segment contains both leaf types.  Fix
\(i\in\{1,2\}\) and choose \(v=s_vC_i^V\in L_1\), with \(s_v\in S_0\).
Vertices of this fixed type form one \(S_0\)-orbit; at \(v\),
\(\mathcal T_0\) has valence two and pair stabilizer
\[
s_vE_is_v^{-1}=\langle\mu,s_vw_is_v^{-1}\rangle\cong\mathbb Z^2,
\qquad s_vw_is_v^{-1}\in S_0.
\]
The meridian centralizes \(S_0\), so
\cref{lem:surface-subtree-conjugacy} applies and finishes the proof.
\end{proof}

\section{The inductive step for subsequent surgery curves}
\label{sec:latercarriers}

The aim of this section is to establish the inductive step
for successive surgery curves. Suppose that two nonseparating
curves \(d,c\subset\Sigma_g\) intersect once. After rim
surgery along \(d\), the push-off of \(c\) acquires a
preferred knot longitude. We show that this gives the
centralizer--transporter and cyclic-subgroup conjugacy
properties needed to use \(c\) as the next surgery curve.
The argument does not require the preceding surgery edge
group to have the centralizer--transporter property.

Let \(F'\) be obtained from \(F\) by ordinary untwisted rim surgery
along a nonseparating curve \(d\subset\Sigma_g\), and let
\(c\subset\Sigma_g\) meet \(d\) once.  Use the basing choices of
\cref{cor:one-crossing-push-off} for both curves and the local knot
group.  We also write \(d,c\in\pi=\pi_1(\Sigma_g,p)\) for the
resulting based elements and put
\[
 \rho:\pi\longrightarrow G,
 \qquad S=\rho(\pi),\qquad h=\rho(d).
\]
Assume that \(\rho\) is injective, \(G\) is torsion-free, and
\[
 H_1(G;\mathbb Z)=\mathbb Z\langle[\mu]\rangle,
 \qquad S\le[G,G].
\]
The positive surface meridian \(\mu\) centralizes \(S\).
Suppose \(\mu^ah^b=1\). Taking homology classes gives
\[
a[\mu]+b[h]=0.
\]
Since \(h\in S\le[G,G]\), we have \([h]=0\).
As \([\mu]\) has infinite order, it follows that \(a=0\).
Thus \(h^b=\rho(d^b)=1\). Injectivity of \(\rho\)
gives \(d^b=1\). Since \(d\) represents an essential
simple closed curve, it has infinite order in
\(\pi_1(\Sigma_g,p)\), so \(b=0\).  Thus
\(A=\langle\mu,h\rangle\cong\mathbb Z^2\), and
\cref{thm:one-step} gives
\begin{equation}\label{eq:generic-rim-stage}
 G'=G*_{A}B,\qquad B=K\times\langle h\rangle,
 \qquad r:G'\onto G.
\end{equation}
Here \(K=\pi_1(E(J))\) is the group of the surgery
knot \(J\). Its meridian \(m_J\) and preferred
longitude \(\ell_J\) generate the peripheral subgroup 
\[
\langle m_J,\ell_J\rangle\cong\mathbb Z^2
\]
of \(K\). In the amalgam, \(m_J\) is identified
with the surface meridian \(\mu\), and we denote
the image of \(\ell_J\) by \(\ell\).
Let \(\rho':\pi\to G'\) and \(S'=\rho'(\pi)\) be the new
push-off and its image.  The retraction satisfies
\(r\circ\rho'=\rho\), so \(\rho'\) is injective.
The one-step calculation in \cref{prop:early-homology} gives
\[
 H_1(G';\mathbb Z)=\mathbb Z\langle[\mu]\rangle,
 \qquad S'\le[G',G'].
\]
The meridian centralizes \(S'\), and \(G'\) is torsion-free by
\cref{lem:tower-torsionfree}.

With the chosen basing paths, the push-off of \(d\)
is unchanged by surgery:
\[
\rho'(d)=\rho(d)=h.
\]
Thus \(h\) belongs to both \(S\) and \(S'\).
Since \(S\le[G,G]\) and \(S'\le[G',G']\), every
element of either surface subgroup is null-homologous
in the corresponding exterior group. On the other hand,
an element \(\mu^ah^b\in A\) has homology class
\(a[\mu]\), because \([h]=0\). Since \([\mu]\) has
infinite order in both groups, such an element can
belong to \(S\) or \(S'\) only if \(a=0\).
Consequently,
\begin{equation}\label{eq:S-cap-A}
S\cap A=S'\cap A=\langle h\rangle.
\end{equation}

We now consider the curve \(c\), which meets \(d\)
once. By \cref{cor:one-crossing-push-off},
\[
x:=\rho'(c)=u\ell^\epsilon,
\qquad
u=\rho(c),\qquad \epsilon\in\{\pm1\}.
\]
We claim that neither factor belongs to \(A\).

First, suppose \(u\in A\). Since \(u\in S\), the
equality above gives \(u=h^b\) for some \(b\in\mathbb Z\).
Injectivity of \(\rho\) then implies \(c=d^b\) in
\(\pi\). This is impossible: it would give
\([c]=b[d]\) in \(H_1(\Sigma_g;\mathbb Z)\), whereas
the algebraic intersection number of \(c\) and \(d\)
is \(\pm1\). Thus \(u\notin A\).

Next, \(\ell^{\pm1}\notin A\). Indeed, under the
projection \(B=K\times\langle h\rangle\to K\),
the subgroup \(A\) maps to \(\langle\mu\rangle\).
But no nonzero power of \(\ell\) belongs to
\(\langle\mu\rangle\), since
\(\langle\mu,\ell\rangle\cong\mathbb Z^2\).

Therefore \(x=u\ell^\epsilon\) is cyclically
reduced of syllable length two in the amalgam
\(G'=G*_A B\). It follows that \(x\) acts
hyperbolically on its Bass--Serre tree \(T\),
with translation length two.

The length two is important. Since the action of \(G'\)
preserves the two vertex classes \(G'/G\) and \(G'/B\),
every positive translation length in \(T\) is even.
Thus \(x\) has the smallest possible positive translation
length, and its action generates all translations of its
axis stabilizer. This will be used in
\cref{thm:later-clean-carrier} to establish the
centralizer--transporter property.

\begin{lemma}[Push-off of the surgery curve]
\label{lem:creation-stage-edge}
Let \(h=\rho(d)\) for the chosen orientation and basing.
After surgery, the push-off of the same curve, with any
orientation and basing, is conjugate in \(G'\) to
\(h\) or \(h^{-1}\). In particular, its cyclic
subgroup is conjugate to \(\langle h\rangle\).
\end{lemma}

\begin{proof}
Let \(d_+\) be the boundary circle of the surgery
annulus on its positive side, oriented in the same
direction as \(d\). Choose a path from \(p\) to
\(d_+\) outside the interior of the annulus, and
denote the resulting based element by
\(\delta_+\in\pi\).
Both the curve and its basing path lie outside the
interior of the surgery annulus, where the surface
and normal section are unchanged. Hence
\[
\rho'(\delta_+)=\rho(\delta_+),
\]
where \(G\) is identified with its subgroup in \(G'\).

Since \(\delta_+\) and \(d\) represent the same oriented
free homotopy class, there is an element \(\gamma\in\pi\)
such that
\[
d=\gamma\delta_+\gamma^{-1}.
\]
Applying the two push-off homomorphisms gives
\[
\begin{aligned}
h=\rho(d)
&=\rho(\gamma)\rho(\delta_+)\rho(\gamma)^{-1},\\
\rho'(d)
&=\rho'(\gamma)\rho(\delta_+)\rho'(\gamma)^{-1}.
\end{aligned}
\]
Therefore, setting
\[
a=\rho'(\gamma)\rho(\gamma)^{-1}\in G',
\]
we obtain
\[
\rho'(d)=aha^{-1}.
\]

Changing the basing path conjugates the corresponding
element of \(\pi\), while reversing the orientation
replaces it by a conjugate of its inverse. Thus, with
any orientation and basing, the new push-off is
conjugate in \(G'\) to \(h\) or \(h^{-1}\). In particular,
its cyclic subgroup is conjugate to \(\langle h\rangle\).
\end{proof}

We now specify the surface splitting.  Let
\[
\Sigma^\dagger
=
\overline{\Sigma_g\setminus
          \bigl(S^1_d\times(-1,1)\bigr)}.
\]
Because $d$ is nonseparating, $\Sigma^\dagger$ is connected;
it has genus $g-1$ and two boundary components.
Retain $q_\pm$, $\lambda_\pm$, $\alpha$, and $\tau$ from the proof
of \cref{cor:one-crossing-push-off}, and label the boundary components
containing $q_\pm$ by $D_\pm$.
Put
\[
R=\pi_1(\Sigma^\dagger,p).
\]
Let \(\bar\delta_\pm\) be the boundary loops based at
\(q_\pm\), oriented in the same \(S^1\)-direction
in the surgery annulus. Its product structure gives
\[
[\tau\bar\delta_+\tau^{-1}]
=
[\bar\delta_-]
\quad\text{in }\pi_1(\Sigma_g,q_-),
\]
where the equality is taken after regluing the annulus.
Define
\[
\delta_\pm
=
[\lambda_\pm\bar\delta_\pm\lambda_\pm^{-1}]
\in R.
\]
The two boundary circles of \(\Sigma^\dagger\) are
parallel copies of \(d\). With the chosen orientation
and basing path, the loop around \(D_+\) represents
the same element as \(d\). Thus
\[
\delta_+=d\qquad\text{in }\pi.
\]

Recall that \(c\) consists of an arc
\(\alpha:q_+\to q_-\) outside the surgery annulus
and an arc \(\tau:q_-\to q_+\) crossing it.
By our basing choice in
\cref{cor:one-crossing-push-off},
\(\lambda_-=\lambda_+\alpha\). Hence
\[
c
=[\lambda_+\alpha\tau\lambda_+^{-1}]
=[\lambda_-\tau\lambda_+^{-1}]
=:\tau_d
\qquad\text{in }\pi.
\]

Regluing the two boundary circles of \(\Sigma^\dagger\)
recovers \(\Sigma_g\). By van Kampen, 
 \(\pi=\pi_1(\Sigma_g,p)\) is an HNN extension
of \(R=\pi_1(\Sigma^\dagger,p)\). The stable letter
is \(\tau_d\), represented by the loop crossing the
annulus once. The annulus identifies the two based
boundary loops, giving the presentation
\begin{equation}\label{eq:surface-HNN}
\pi
=
\left\langle R,\tau_d\ \middle|\
\tau_d\delta_+\tau_d^{-1}=\delta_-
\right\rangle.
\end{equation}
Both boundary cyclic subgroups are proper in \(R\),
so the splitting is reduced. Let \(T_d\) denote
its Bass--Serre tree.

Since \(R=\pi_1(\Sigma^\dagger,p)\) is represented by loops
outside the surgery annulus, their push-offs are unchanged.
Identifying \(G\) with its subgroup in \(G'\), we have
\[
\rho'|_R=\rho|_R.
\]
In particular, since \(\delta_+=d\) in \(\pi\),
\begin{equation}\label{eq:boundary-images}
\rho'(\delta_+)=\rho(\delta_+)=h.
\end{equation}
Applying \(\rho'\) to the HNN relation
\(\tau_d\delta_+\tau_d^{-1}=\delta_-\), and using
\(\rho'(\tau_d)=x=u\ell^\epsilon\) and
\([\ell,h]=1\), gives
\[
\rho'(\delta_-)=\rho(\delta_-)
=xhx^{-1}=uhu^{-1}.
\]
\begin{proposition}[The surface tree inside the exterior tree]
\label{prop:surface-subtree}
Let \(T\) be the Bass--Serre tree of \(G'=G*_A B\).
There is an \(S'\)-equivariant simplicial embedding
\[
\Phi:\operatorname{sd}(T_d)\into T
\]
whose image \(\mathcal T\) is an \(S'\)-invariant subtree.
The original vertices \(\gamma R\) map to
\(\rho'(\gamma)G\), while the inserted midpoints
map to vertices \(sB\), with \(s\in S'\).

Exactly two edges of \(\mathcal T\) meet at each
such vertex \(sB\). Their setwise stabilizer
in \(G'\) is \(sAs^{-1}\), and
\[
S'\cap sAs^{-1}=s\langle h\rangle s^{-1}.
\]
\end{proposition}
\begin{proof}
Let $v_0$ be the vertex of $T_d$ stabilized by $R$, and choose the
edge $e_-$ joining $v_0$ to $\tau_dv_0$ whose stabilizer is
$\langle\delta_-\rangle$.  Subdivide $e_-$.  Map its two half-edges
to the ambient edge cosets
\[
uA,
\qquad
u\ell^\epsilon A.
\]
Their common endpoint is the $B$-vertex $uB$, since
$u\ell^\epsilon B=uB$.  Their $G$-endpoints are
$uG=G$ and
$u\ell^\epsilon G=xG$, exactly the images of $v_0$ and
$\tau_dv_0$.  Both ambient edges have stabilizer $uAu^{-1}$; the
surface edge stabilizer maps to
$\langle uhu^{-1}\rangle=\rho'(\langle\delta_-\rangle)$.
At the first endpoint this is the image of the boundary subgroup
$\langle\delta_-\rangle$, while at the second endpoint conjugating by
$x^{-1}$ gives
\[
x^{-1}uhu^{-1}x=h,
\]
the image of $\langle\delta_+\rangle$.

The adjacent edge $\tau_d^{-1}e_-$ has stabilizer
$\langle\delta_+\rangle$, which maps literally to $\langle h\rangle$.
Its subdivided image is obtained by translating the preceding path by
$x^{-1}$, and its two edge cosets are
\[
\ell^{-\epsilon}A,
\qquad
A.
\]
Thus the selected surface edge whose stabilizer is $\langle h\rangle$
maps to a two-edge path through the base $B$-vertex.  Extending the
construction $\rho'$-equivariantly defines a simplicial map
\[
\Phi:\operatorname{sd}(T_d)\longrightarrow T.
\]

We verify local injectivity at a lift of the surface vertex.  At
$v_0$, the incident half-edges have two families of ambient images,
represented by
\[
\rho(r)A
\quad\text{and}\quad
\rho(r)uA,
\qquad r\in R.
\]
In the first family,
$\rho(r_1)A=\rho(r_2)A$ implies
\[
\rho(r_2^{-1}r_1)\in S\cap A=\langle h\rangle
 =\rho(\langle\delta_+\rangle).
\]
Injectivity of $\rho$ therefore says that the two representatives
determine the same surface half-edge.  In the second family,
$\rho(r_1)uA=\rho(r_2)uA$ implies
\[
u^{-1}\rho(r_2^{-1}r_1)u
\in S\cap A=\langle h\rangle.
\]
Since $u=\rho(\tau_d)$ before surgery, injectivity of $\rho$ and the
HNN relation give
$r_2^{-1}r_1\in\langle\delta_-\rangle$, again exactly the equality
relation for those half-edges. Finally, suppose that a half-edge from the second family has the
same image as one from the first family:
\[
\rho(r_1)uA=\rho(r_2)A.
\]
Then
\[
\rho(r_2^{-1}r_1)u
=
\rho(r_2^{-1}r_1\tau_d)
\in S\cap A=\langle h\rangle.
\]
Injectivity of \(\rho\) would therefore place
\(r_2^{-1}r_1\tau_d\) in \(\langle\delta_+\rangle\le R\).
But \(r_2^{-1}r_1\tau_d\) is a reduced HNN word containing one
stable-letter syllable, so Britton's lemma says that it does not lie
in \(R\).  This excludes equality between the two families.

At an inserted barycentric vertex the two images are distinct because
$\ell^\epsilon\notin A$.  Hence $\Phi$ is locally injective.  A
locally injective simplicial map from a tree to a tree sends every
geodesic to a reduced path; two distinct vertices therefore cannot
have the same image, because their joining geodesic would map to a
nontrivial closed reduced path in a tree.  Thus $\Phi$ is an embedding.

It remains to compute the stabilizer of the two edges
at an inserted midpoint. By \(S'\)-equivariance,
it suffices to consider the vertex \(B\).
Its two incident edges in \(\mathcal T\) are
\(A\) and \(\ell^{-\epsilon}A\).
If \(\epsilon=1\), left multiplication by \(\ell\)
takes this pair to
\begin{equation}\label{eq:rooted-edge-pair}
A,\qquad \ell A.
\end{equation}
If \(\epsilon=-1\), the pair already has this form. The preferred longitude centralizes
$\mu$ in the knot peripheral subgroup and centralizes $h$ in the
direct product $B=K\times\langle h\rangle$; hence it centralizes all
of $A$.  Every element of $A$ therefore fixes both cosets $A$ and
$\ell A$.

Conversely, an element preserving the pair fixes its common vertex $B$
and hence belongs to $B$.  If it fixes the coset $A$, it belongs to $A$.  If
$b\in B$ exchanged the two cosets, then $b\in\ell A$, and centrality
of $\ell$ with $A$ would give
\[
b(\ell A)=\ell^2A.
\]
Exchange would force $\ell^2\in A$.  Projection
$B=K\times\langle h\rangle\to K$ would then give
$\ell^2\in\langle\mu\rangle$, contradicting
$\langle\mu,\ell\rangle\cong\mathbb Z^2$.  Thus the setwise pair
stabilizer is exactly $A$.

For the selected pair, \eqref{eq:S-cap-A} gives
\(A\cap S'=\langle h\rangle\), which is the image of the selected
surface-edge stabilizer \(\langle\delta_+\rangle\).
Conjugating by \(S'\) gives the assertion at every product vertex.
\end{proof}

\begin{theorem}[The inductive axis and conjugacy step]
\label{thm:later-clean-carrier}
In the one-step setting above,
\(\langle\mu,x\rangle\cong\mathbb Z^2\) has the
centralizer--transporter property in \(G'\) relative to \(\mu\).
Moreover, \((S',\langle x\rangle)\) satisfies the cyclic-subgroup
conjugacy condition in \(G'\).
\end{theorem}

\begin{proof}
Put \(L=\Ax_T(x)\).  At the base \(G\)-vertex on this axis, the
consecutive edges are \(A\) and \(uA\).
To compute their stabilizer intersection, write
\begin{equation}\label{eq:consecutive-intersection}
 \mu^r h^s=u\mu^{r'}h^{s'}u^{-1}.
\end{equation}
Abelianization gives \(r=r'\).  Since \(u\in S\) commutes with
\(\mu\), this reduces to \(h^s=uh^{s'}u^{-1}\).
Since \(\rho(\delta_+)=h\) and \(\rho(\delta_-)=uhu^{-1}\),
this is
\[
 \rho(\delta_+^s)=\rho(\delta_-^{s'}).
\]
Injectivity of \(\rho\) gives equality in \(\pi\), hence in its
embedded subgroup \(R=\pi_1(\Sigma^\dagger,p)\).
The surface \(\Sigma^\dagger\) has genus \(g-1\ge1\) and two
boundary components, so \cref{lem:free-boundary-words} gives
\(\langle\delta_+\rangle\cap\langle\delta_-\rangle=1\).
Thus \(s=s'=0\), and
\[
 A\cap uAu^{-1}=\langle\mu\rangle.
\]
The reverse inclusion uses \([\mu,u]=1\).

Every syllable of \(u\ell^\epsilon\) commutes with \(\mu\).
Hence \(\mu\) fixes the axis edges in one fundamental segment and
all their \(x\)-translates.  Together with the preceding intersection,
this proves
\[
 \Fix_{G'}^{\mathrm{pt}}(L)=\langle\mu\rangle.
\]
The action preserves the \(G\)- and \(B\)-vertex types, so every
positive translation length is even.  Since \(\|x\|_T=2\), its
translation generates the positive translation image of the axis
stabilizer.  The torsion and homology hypotheses were established in
the setup.  Therefore \cref{lem:clean-axis-criterion} gives
\(\Stab_{G'}(L)=\langle\mu,x\rangle\cong\mathbb Z^2\) and the
centralizer--transporter property.

Use the embedded invariant subtree \(\mathcal T\) from
\cref{prop:surface-subtree}.  Its \(B\)-type vertices form one
\(S'\)-orbit and have valence two.  Since \(x\in S'\), its axis
lies in \(\mathcal T\).  Choose \(v=s_vB\in L\), with \(s_v\in S'\).
The pair stabilizer at \(v\) is
\[
 s_vAs_v^{-1}=\langle\mu,s_vhs_v^{-1}\rangle,
 \qquad s_vhs_v^{-1}\in S'.
\]
The meridian centralizes \(S'\). Thus
\cref{lem:surface-subtree-conjugacy} applies and gives the
cyclic-subgroup conjugacy condition.
\end{proof}

\begin{corollary}[Induction along the ordered curve sequence]
\label{cor:carrier-induction}
Let
\[
d_1,\ldots,d_N\subset\Sigma_g
\]
be an ordered sequence of nonseparating curves, with \(d_1\) the
curve constructed in \cref{sec:initialcarriers}, and
\[
i(d_i,d_{i+1})=1
\qquad(1\le i<N).
\]
Choose an orientation and a basing path for each curve, and continue
to denote the resulting element of
\(\pi=\pi_1(\Sigma_g,p)\) by \(d_i\).  Let
$F_i$, $G_i$, and
\[
\rho_i:\pi_1(\Sigma_g,p)\longrightarrow G_i
\]
denote the surface, exterior group, and push-off homomorphism after the
first $i$ rim surgeries, starting with the surface
\(F_0\) of \cref{sec:seed}.
For $1\le i\le N$, put
\[
h_i=\rho_{i-1}(d_i),
\qquad
S_{i-1}=\rho_{i-1}(\pi_1(\Sigma_g,p)).
\]
Then $\rho_i$ is injective at every stage, and each subgroup
$\langle\mu,h_i\rangle\cong\mathbb Z^2$ has the
centralizer--transporter property in $G_{i-1}$ relative to $\mu$.
Moreover, each pair $(S_{i-1},\langle h_i\rangle)$ satisfies the
cyclic-subgroup conjugacy condition in $G_{i-1}$.
\end{corollary}

\begin{proof}
The initial push-off \(\rho_0\) is injective by \cref{sec:seed}.
For \(h_1\), the edge property follows from
\cref{lem:even-clean-axis,lem:odd-clean-axis}, and the cyclic-subgroup
condition from
\cref{prop:even-rooted-isolation,prop:odd-rooted-isolation}, in the
respective parity.  Changing orientation or basing replaces the initial
cyclic subgroup by an \(S_0\)-conjugate and preserves both properties.

Suppose the assertions hold before surgery \(i\), where
\(1\le i\le N\).  Its rank-two edge allows \cref{thm:one-step} to
give the injectivity of \(\rho_i\).
If \(i<N\), apply \cref{thm:later-clean-carrier} to the surgery
along \(d_i\) and the curve \(d_{i+1}\), which meets it once.
The homology and torsion hypotheses follow from
\cref{prop:early-homology,lem:tower-torsionfree}.
For the adapted representatives the theorem gives the two properties
for \(x\).  For the prescribed orientation and basing of \(d_{i+1}\),
we have
\[
h_{i+1}=s x^{\pm1}s^{-1}
\qquad\text{for some }s\in S_i.
\]
Since \(s\in S_i\) commutes with \(\mu\), conjugation by
\(s\) and inversion preserve both the centralizer--transporter
property and the cyclic-subgroup conjugacy condition.
This completes the induction step. For \(i=N\),
\cref{thm:one-step} also gives injectivity of \(\rho_N\).
\end{proof}

\section{Centralizers in the amalgam arising from rim surgery}
\label{sec:centralizers}

We now turn from constructing the exterior groups to
recognizing the surgeries from the resulting groups.
The amalgam decomposition identifies the subgroup
introduced by each surgery, but an automorphism of
the exterior group need not preserve this decomposition.
We therefore need to recognize these subgroups
intrinsically. The key idea is that surgery introduces new commuting
relations in the exterior group. In this section, we
study centralizers to understand these relations and
distinguish them from those already present before
surgery. In \cref{sec:recognition}, we use these results to show
that the meridian-marked exterior group determines,
up to conjugacy, the product subgroups introduced
by rim surgery.

Let $K=\pi_1(E(J))$ be a hyperbolic knot group,
with meridian $m$ and preferred longitude $\ell$.
Set
\[
P_K=\langle m,\ell\rangle\cong\Z^2.
\]
We use standard terminology from hyperbolic geometry:
a loxodromic isometry of $\mathbb H^3$ preserves a
geodesic axis and translates along it, while a
parabolic isometry fixes exactly one point at infinity.
For the hyperbolic knot group $K$, the nontrivial
parabolic elements are precisely those conjugate
into $P_K$. See \cite{RatcliffeHyperbolic} for 
background reading.

\begin{lemma}[Centralizers and transporters in a hyperbolic knot group]
\label{lem:knot-centralizers}
The group $K$ is torsion-free and
$Z(K)=\{1\}$. Every nontrivial element
is either loxodromic or parabolic, and the following conditions hold.
\begin{enumerate}[label=\textup{(\roman*)},leftmargin=2.3em]
\item If $k$ is loxodromic, then $C_K(k)=\langle r\rangle\cong\Z$,
where $r$ generates the maximal cyclic subgroup containing $k$.
\item If $k$ is parabolic, it lies in a unique maximal parabolic
subgroup $P$, and $C_K(k)=P\cong\Z^2$.
\item The subgroup $P_K$ is malnormal and self-normalizing.
\item For every $n\ne0$,
\[
 C_K(m^n)=P_K
 =\{q\in K:q^{-1}m^nq\in\langle m\rangle\}.
\]
\end{enumerate}
\end{lemma}

\begin{proof}
The complete hyperbolic structure realizes $K$ as a torsion-free,
finite-covolume discrete subgroup of $\operatorname{PSL}_2(\mathbb C)$.
Its nonidentity elements are therefore loxodromic or parabolic.

For loxodromic $k$, the centralizer preserves its oriented geodesic
axis.  The group of orientation-preserving isometries preserving that
oriented geodesic is $\R\times\operatorname{SO}(2)$.  Its projection
to $\R$ is proper, so the image of the closed discrete subgroup
$C_K(k)$ is closed.  This image is nonzero and countable, hence is a
nonzero discrete subgroup of $\R$.  The kernel is a discrete subgroup
of the compact rotation group, hence finite and therefore trivial.
Thus $C_K(k)$ is infinite cyclic, generated by its shortest positive
translation.

If $k$ is parabolic, every $g\in K$ commuting
with $k$ must fix the unique fixed point of
$k$ at infinity.  Its stabilizer in this finite-volume manifold group is
the rank-two abelian cusp subgroup $P$.  Hence $C_K(k)=P$.
These descriptions also prove centerlessness, since $K$ is nonabelian.

Malnormality and self-normalization is \cite[Corollary~2]{HarpeWeber}. If $q^{-1}m^nq\in\langle m\rangle$ with
$n\ne0$, then $q^{-1}P_Kq\cap P_K\ne1$, so $q\in P_K$.
Conversely, every element of $P_K$ centralizes $m^n$.  This proves
\textup{(iv)}.
\end{proof}

Recall from \cref{thm:one-step} that the exterior
group after one rim surgery has an amalgam decomposition.
We write it as
\begin{equation}
\label{eq:clean-extension}
H=G*_{A}B,
\qquad A=\langle\mu,h\rangle\cong\Z^2,
\qquad B=K\times\langle h\rangle.
\end{equation}
Here $G$ is the exterior group before surgery,
$\mu$ is its surface meridian, and $h=\rho_F(d)$
is the push-off image of the surgery curve.
The amalgamation identifies the knot meridian
$m\in K$ with $\mu$ and the generator $s$ of
the $S^1_d$-factor with $h$.
Put
\begin{equation}
\label{eq:D-definition}
D=P_K\times\langle h\rangle
 =\langle\mu,\ell,h\rangle\cong\Z^3.
\end{equation}
Let $T$ be the Bass--Serre tree of the amalgam
$H=G*_{A}B$. Its base edge $e$ joins the vertices
$G$ and $B$, with stabilizers $A$, $G$, and $B$,
respectively. For the calculations below, we make two assumptions:
\begin{enumerate}[label=\textup{(C\arabic*)},leftmargin=3.4em]
\item\label{cond:powerclean}
The subgroup $A$ has the centralizer--transporter
property in $G$ relative to $\mu$.
\item\label{cond:homology}
$H_1(G;\Z)=\Z\langle[\mu]\rangle$ and $[h]=0$
in $H_1(G;\Z)$.
\end{enumerate}

The first homology is unchanged by the surgery.
Indeed, $H_1(K;\Z)$ is generated by the knot meridian,
which is identified with $\mu$ in the amalgam.
The generator of the $S^1_d$-factor is identified
with $h\in G$. Thus
\[
H_1(H;\Z)\cong H_1(G;\Z)
=\Z\langle[\mu]\rangle,
\qquad [h]=0.
\]

We also recall the retraction from \cref{thm:one-step}.
Let $\phi_K:K\to\Z$ be the abelianization sending
the knot meridian to $1$. Then
\[
r:H\longrightarrow G,
\qquad r|_G=\operatorname{id}_G,
\qquad r(k,h^n)=\mu^{\phi_K(k)}h^n.
\]
This defines a homomorphism on $B$ because $\mu$
commutes with $h$, and it agrees with the identity
on $G$ along $A$. Hence $r$ is well defined and
fixes $G$ pointwise. In particular,
\[
G\cap[H,H]=[G,G].
\]

To compute centralizers of nonzero powers of the
meridian, we make one additional assumption.
Write $M_G=C_G(\mu)$, and assume:
\begin{enumerate}[label=\textup{(C\arabic*)},start=3,leftmargin=3.4em]
\item\label{cond:meridian-old}
$M_G$ is nonabelian, and
$C_G(\mu^n)=M_G$ for every nonzero integer $n$.
\end{enumerate}

\begin{lemma}[Transporter of a nonzero meridian power]
\label{lem:meridian-transporter}
Under Condition~\ref{cond:meridian-old}, for every $n\ne0$,
\begin{equation}
\label{eq:old-meridian-transporter}
 \Trans_G(\mu^n,A)=M_G.
\end{equation}
\end{lemma}

\begin{proof}
Suppose $g\in G$ satisfies $g^{-1}\mu^n g\in A$,
where $n\ne0$. Write
\[
g^{-1}\mu^n g=\mu^a h^b.
\]
Passing to $H_1(G;\Z)$ gives $a=n$, since
$[\mu]$ generates $H_1(G;\Z)$ and $[h]=0$.

Suppose $b\ne0$. By Condition~\ref{cond:powerclean},
the centralizer of $\mu^n h^b$ in $G$ is $A$,
which is abelian. On the other hand, since
$\mu^n h^b=g^{-1}\mu^n g$, its centralizer is
\[
C_G(\mu^n h^b)
=g^{-1}C_G(\mu^n)g
=g^{-1}M_Gg.
\]
This is nonabelian by Condition~\ref{cond:meridian-old}, which gives
a contradiction. Therefore $b=0$, so
$g^{-1}\mu^n g=\mu^n$. It follows that
$g\in C_G(\mu^n)=M_G$.

Conversely, every $g\in M_G$ commutes with $\mu$,
so $g^{-1}\mu^n g=\mu^n\in A$.
This proves the reverse inclusion.
\end{proof}

\begin{lemma}[Centralizer from a fixed edge]
\label{lem:fixed-edge-centralizer}
Let $H=G*_{A}B$ have injective edge maps, and let $1\ne x\in A$.
Suppose $A\le U\le G$ and $A\le V\le B$ satisfy
\[
 C_G(x)=\Trans_G(x,A)=U,
 \qquad C_B(x)=\Trans_B(x,A)=V.
\]
Then
\[
 C_H(x)=\langle U,V\rangle\cong U*_{A}V.
\]
An edge of $T$ is fixed by $x$ if and only if it
has the form $ce$ for some
$c\in C_H(x)=\langle U,V\rangle$.
Every vertex fixed by $x$ is an endpoint of a
fixed edge. Consequently, the quotient
\(\Fix_T(x)/C_H(x)\) consists of a
single edge, represented by $e$, joining the
images of the base vertices $G$ and $B$.
In $C_H(x)$, the vertex $G$ has stabilizer $U$,
the vertex $B$ has stabilizer $V$, and
the edge $e$ has stabilizer $A$.
\end{lemma}

\begin{proof}
The inclusions $U\le G$ and $V\le B$ induce a
homomorphism
\[
U*_{A}V\longrightarrow H=G*_{A}B.
\]
A reduced word in $U*_{A}V$ is also reduced in $H$:
every syllable in $U\setminus A$ belongs to
$G\setminus A$, and similarly for $V\setminus A$.
By
\cref{lem:subgroup-amalgam}, this homomorphism is
injective, with image $L=\langle U,V\rangle$.

By assumption, every element of $U$ and $V$
commutes with $x$. Hence $L\le C_H(x)$.
Since $x\in A$ fixes the base edge $e$, for every
$l\in L$ we have
\[
x(le)=l(xe)=le.
\]
Thus every $L$-translate of $e$ is fixed by $x$.

Let $f$ be any edge fixed by $x$. Since
$\Fix_T(x)$ is a subtree containing $e$ and $f$,
the unique path from $e$ to $f$ consists of fixed
edges. We show inductively that every edge along
this path is an $L$-translate of $e$.

Suppose an edge already reached is $qe$, where
$q\in L$. If the next edge meets $qe$ at its
$G$-endpoint, it has the form $qge$ for some
$g\in G$. Since $q$ commutes with $x$ and
$qge$ is fixed by $x$, we have
\[
g^{-1}xg\in A.
\]
Thus $g\in\Trans_G(x,A)=U$, so $qg\in L$.
If the next edge meets $qe$ at its $B$-endpoint,
the same argument uses $\Trans_B(x,A)=V$.
Therefore every edge fixed by $x$ is an
$L$-translate of $e$.

Let $c\in C_H(x)$. Since $x$ fixes $e$ and
commutes with $c$, it also fixes $ce$.
We have shown that every fixed edge is an
$L$-translate of $e$, so $ce=le$ for some
$l\in L$. Hence $l^{-1}c$ fixes $e$, giving
\[
l^{-1}c\in\Stab_H(e)=A\le L.
\]
Thus $c\in L$. Together with $L\le C_H(x)$,
this proves $C_H(x)=L$.

Every vertex fixed by $x$ is an endpoint of a
fixed edge, since the unique path from that
vertex to $e$ lies in $\Fix_T(x)$.
As all fixed edges are $C_H(x)$-translates of
$e$, the fixed vertices form exactly two orbits,
represented by the vertices $G$ and $B$. Therefore
$\Fix_T(x)/C_H(x)$ consists of a single edge
joining these two vertex orbits.

In $C_H(x)$, the stabilizer of the base
$G$-vertex is
\[
C_H(x)\cap G=C_G(x)=U,
\]
and the stabilizer of the base $B$-vertex is
\[
C_H(x)\cap B=C_B(x)=V.
\]
The stabilizer of the base edge $e$ is $A$.
\end{proof}

\begin{proposition}[Centralizer classification]
\label{prop:complete-centralizers}
Assume Conditions~\ref{cond:powerclean}--\ref{cond:meridian-old}.
We classify $C_H(x)$ for every nontrivial $x\in H$
according to its action on $T$: either $x$ acts
hyperbolically, fixes exactly one vertex, or fixes
an edge.

\begin{enumerate}[label=\textup{(\roman*)},leftmargin=2.6em]
\item \emph{Hyperbolic on $T$.}
There is an exact sequence
\begin{equation}
\label{eq:hyperbolic-centralizer-sequence}
1\longrightarrow E_x\longrightarrow C_H(x)
\longrightarrow\Z\longrightarrow1,
\end{equation}
where $E_x$ is free abelian of rank at most two.
In particular, $C_H(x)$ is metabelian.

\item \emph{Exactly one fixed $G$-vertex.}
If $\Fix_T(x)=\{gG\}$, then
\[
C_H(x)=C_{gGg^{-1}}(x).
\]

\item \emph{Exactly one fixed $B$-vertex.}
After conjugating, assume $\Fix_T(x)=\{B\}$
and write $x=(k,h^n)$. Then $k\ne1$ and
\[
C_H(x)=C_K(k)\times\langle h\rangle.
\]
If $k$ is loxodromic, this group is isomorphic
to $\Z^2$. If $k$ is parabolic, we may conjugate
further in $B$ so that $k\in P_K$.
Then $k\notin\langle\mu\rangle$ and
$C_H(x)=D\cong\Z^3$.

\item \emph{At least one fixed edge.}
After conjugating, write $x=\mu^ah^b\in A$,
where $(a,b)\ne(0,0)$. There are three cases:
\begin{enumerate}[label=\textup{(\alph*)},leftmargin=2.3em]
\item If $a=0$, then $b\ne0$ and
\begin{equation}
\label{eq:centralizer-h-power}
C_H(h^b)=B.
\end{equation}

\item If $ab\ne0$, then
\begin{equation}
\label{eq:centralizer-mixed}
C_H(\mu^ah^b)=D\cong\Z^3.
\end{equation}

\item If $b=0$, then $a\ne0$ and
\begin{equation}
\label{eq:centralizer-meridian-composite}
C_H(\mu^a)=M_G*_{A}D=C_H(\mu).
\end{equation}
\end{enumerate}
\end{enumerate}

Condition~\ref{cond:meridian-old} is needed only for
\textup{(iv)(c)}. Under Conditions~\ref{cond:powerclean}
and~\ref{cond:homology} alone, every nontrivial
$x\in[H,H]$ is covered by
\textup{(i)}--\textup{(iii)} or \textup{(iv)(a)}.
\end{proposition}

\begin{proof}
First, suppose $x$ acts hyperbolically on $T$, and let
$L=\Ax_T(x)$ be its axis. Since $H$ acts on $T$
by isometries, for every $c\in H$ we have
\(
\Ax_T(cxc^{-1})=cL.
\)
Thus, if $c\in C_H(x)$, then $cL=L$. Every isometry of $L$ is either
a translation or a reflection. A reflection conjugates
a nonzero translation to its inverse, so it cannot
commute with $x$. Therefore every element of
$C_H(x)$ acts on $L$ by translation.

Choose an orientation of $L$. For each
$c\in C_H(x)$, let $\theta(c)\in\Z$ be its
signed translation distance along $L$. This defines a homomorphism
$\theta:C_H(x)\to\Z$ with nonzero image,
since $\theta(x)\ne0$. Its kernel is
\[
E_x=C_H(x)\cap\Fix_H^{\mathrm{pt}}(L).
\]
Since $E_x$ fixes every edge of $L$, it lies
in an edge stabilizer, which is a conjugate
of $A\cong\Z^2$. Thus $E_x$ is free abelian
of rank at most two. Therefore we obtain the exact sequence
\eqref{eq:hyperbolic-centralizer-sequence}.
Since $C_H(x)/E_x$ is abelian, the commutator
subgroup of $C_H(x)$ lies in $E_x$. But $E_x$
is abelian so $C_H(x)$ is metabelian.

Now suppose $x$ fixes a vertex of $T$ but no edge, it
fixes exactly one vertex: otherwise, the unique
path between two fixed vertices would also be
fixed. Every element commuting with $x$ must
fix this vertex.  If the fixed vertex is $gG$, then
$C_H(x)=C_{gGg^{-1}}(x)$, proving \textup{(ii)}.
If the fixed vertex is $gB$, we may
replace $x$ by $g^{-1}xg$, which fixes $B$.
By writing $x=(k,h^n)$, we obtain
$C_H(x)=C_B(x)=C_K(k)\times\langle h\rangle$
as in \textup{(iii)}. If $k=1$, then $x=h^n\in A$ fixes the
edge $e$ joining $G$ and $B$ which is a contradiction. If $k\ne1$, the formulas follow from
\cref{lem:knot-centralizers}.

If $k$ is parabolic, we may assume $k\in P_K$.
Since $x$ fixes no edge, $k\notin\langle\mu\rangle$.
Moreover, suppose $q\in K$ satisfies
$q^{-1}kq\in\langle\mu\rangle$.
Since $k\in P_K$, and $P_K$ is malnormal, $q\in P_K$. But $P_K$ is abelian, we have
$q^{-1}kq=k\in\langle\mu\rangle$,
which is a contradiction.

We now prove \textup{(iv)}. Since $x$ fixes an
edge, after conjugating $x$ we may assume it
fixes the base edge $e$. Thus
\[
x=\mu^ah^b\in A,\qquad (a,b)\ne(0,0).
\]
Condition~\ref{cond:powerclean} and
\cref{lem:meridian-transporter} give
\[
C_G(x)=\Trans_G(x,A)=
\begin{cases}
A,&b\ne0,\\
M_G,&b=0.
\end{cases}
\]

In $B=K\times\langle h\rangle$, conjugation by
$(k,h^n)$ sends $x$ to $k^{-1}\mu^ak\,h^b$.
If $a=0$, then $x=h^b$ commutes with every
element of $B$.
If $a\ne0$, \cref{lem:knot-centralizers} gives
$C_K(\mu^a)=P_K$ and shows that
$k^{-1}\mu^ak\in\langle\mu\rangle$
exactly when $k\in P_K$. Hence
\[
C_B(x)=\Trans_B(x,A)=
\begin{cases}
B,&a=0,\\
D,&a\ne0.
\end{cases}
\]

Applying \cref{lem:fixed-edge-centralizer} with
$(U,V)=(A,B)$, $(A,D)$, and $(M_G,D)$
in the three respective cases proves
\textup{(iv)(a)}--\textup{(iv)(c)}. The fixed subtree in case \textup{(iv)(a)}
and its quotient are illustrated in
\cref{fig:centralizer-h-power-star}.

\begin{figure}[htbp]
  \centering
  \includegraphics[width=154mm]{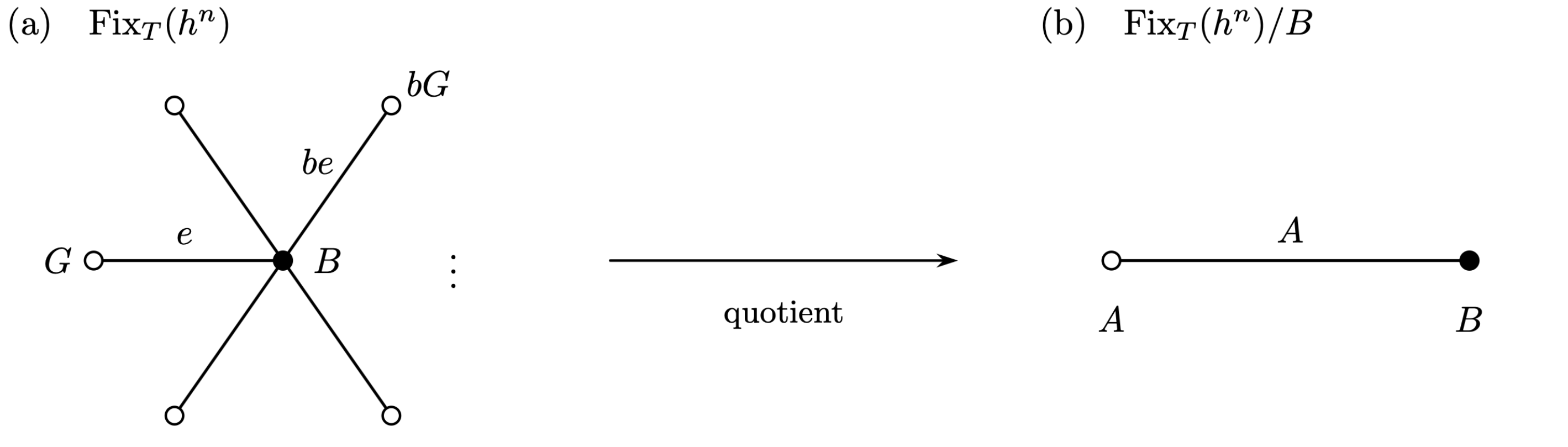}

\caption{For $n\ne0$, the fixed subtree
$\Fix_T(h^n)$ is a star centered at the base
$B$-vertex, with edges $be$ for $b\in B$. Only a few edges are shown. 
In~(b), the left and right vertices are represented
by $G$ and $B$, with stabilizers $A$ and $B$,
respectively. The edge has stabilizer $A$.
}
\label{fig:centralizer-h-power-star}

\end{figure}

Every nontrivial element either acts hyperbolically,
fixes exactly one vertex, or fixes an edge, so the
cases are exhaustive. The three cases in \textup{(iv)} are distinct
under conjugation. Indeed, elements in
\textup{(iv)(a)} have trivial class in
$H_1(H;\Z)$, whereas those in
\textup{(iv)(b)} and \textup{(iv)(c)} do not.
The latter two cases are distinguished by
their centralizers, which are abelian and
nonabelian, respectively.

Finally, suppose $1\ne x\in[H,H]$ fixes an edge.
After conjugating $x$, we can write $x=\mu^ah^b\in A$.
Since $[x]=a[\mu]=0$ in $H_1(H;\Z)$,
we have $a=0$ and $b\ne0$. Thus only
\textup{(iv)(a)} applies. 
\end{proof}

\begin{corollary}[Central quotients of hyperbolic centralizers]
\label{cor:hyperbolic-not-knot}
If $x$ is hyperbolic on $T$, then $C_H(x)/Z(C_H(x))$ is solvable
and is not a nonabelian hyperbolic knot group.
\end{corollary}

\begin{proof}
By \cref{prop:complete-centralizers}, $C_H(x)$
is metabelian. Hence $C_H(x)/Z(C_H(x))$
is solvable. Since a hyperbolic knot group
is not solvable \cite{RatcliffeHyperbolic},
the result follows.
\end{proof}

\begin{lemma}[Primitivity and the center of the product vertex group]
\label{lem:carrier-primitive}
The element $h$ is not a proper power in $G$ or in $B$, and
$Z(B)=\langle h\rangle$.
\end{lemma}

\begin{proof}
Suppose $u^n=h$ for some $u\in G$ and
$|n|>1$. Since $u$ commutes with $h=u^n$,
we have $u\in C_G(h)=A$. Write
$u=\mu^ph^q$ and compare the powers of $h$
in $u^n=h$ gives $nq=1$ which is impossible. Similarly, suppose $v^n=h$ for some
$v=(k,h^q)\in B$. Comparing the second
factors in $v^n=h$ again gives $nq=1$. Finally, \cref{lem:knot-centralizers}
gives $Z(K)=\{1\}$. Since
$B=K\times\langle h\rangle$, we have
\[
Z(B)=Z(K)\times\langle h\rangle
=\langle h\rangle.
\]
\end{proof}

We have shown that
\[
C_H(h^n)=B,\qquad B/Z(B)\cong K
\qquad(n\ne0).
\]
Thus the subgroup introduced by surgery appears as
a centralizer whose quotient by its center recovers
the knot group. To recognize this subgroup from
$H$, we must distinguish it from other centralizers.
Apart from those inherited from $G$, the only
remaining case is the meridian centralizer $C_H(\mu)$.The next proposition gives additional information
about this meridian centralizer. Although not needed
for the recognition argument, it describes how the
center and first Betti number of its central quotient
behave under rim surgery. We use it in
\cref{prop:tower-meridian-centralizers}.

For a group $Q$, write
\[
b_1(Q)=\dim_{\mathbb Q}H_1(Q;\mathbb Q).
\]
Recall that $M_G=C_G(\mu)$ is the meridian
centralizer in the exterior group before surgery.

\begin{proposition}[The meridian centralizer]
\label{prop:meridian-composite}
Assume Condition~\ref{cond:meridian-old} and
\[
 Z(M_G)=\langle\mu\rangle,
 \qquad b_1(M_G/\langle\mu\rangle)\ge g\ge3.
\]
Then $M_H:=C_H(\mu)=M_G*_{A}D$ satisfies
\begin{equation}
\label{eq:meridian-center}
 Z(M_H)=\langle\mu\rangle,
\end{equation}
\begin{equation}
\label{eq:meridian-quotient}
 M_H/\langle\mu\rangle
 \cong (M_G/\langle\mu\rangle)
       *_{\langle h\rangle}\langle h,\ell\rangle,
\end{equation}
and
\begin{equation}
\label{eq:b1-meridian-composite}
 b_1(M_H/\langle\mu\rangle)
 =b_1(M_G/\langle\mu\rangle)+1.
\end{equation}
In particular, $M_H/Z(M_H))$ is not
isomorphic to the fundamental group of
any knot complement in $S^3$.
\end{proposition}

\begin{proof}
By \eqref{eq:centralizer-meridian-composite},
\[
M_H=M_G*_{A}D.
\]
The edge group \(A\) is proper in \(D\) because
\(\ell\notin A\), and proper in \(M_G\) because
\(M_G\) is nonabelian. Let \(T\) be the Bass--Serre
tree of this amalgam. The action of \(M_H\) on \(T\)
is minimal: no vertex is fixed by all of \(M_H\),
and any nonempty invariant subtree containing an
edge must be all of \(T\), since \(M_H\) acts
transitively on the edges.

Let $z\in Z(M_H)$. Suppose first that $z$ is hyperbolic with
axis $L$. Since $z$ is central, every element of $M_G$ or $D$
preserves $L$ and acts on it by a translation. But every element
of either vertex group is elliptic, so its translation on $L$
must be trivial. As these vertex groups generate $M_H$, this
would force $z$ to fix $L$ pointwise which gives a contradiction.
Therefore $z$ is elliptic. The fixed subtree $\Fix_T(z)$ is thus nonempty.
Since $z$ is central, this subtree is $M_H$-invariant:
if $zv=v$, then $z(gv)=g(zv)=gv$ for every $g\in M_H$.
By minimality, $\Fix_T(z)=T$. In particular, $z$ fixes
the base edge whose stabilizer is $A$. Thus $z\in A$.
Since $z$ also centralizes $M_G$, we obtain
\[
z\in A\cap Z(M_G)=\langle\mu\rangle.
\]
Conversely, $\mu$ is central in both vertex groups, which
generate $M_H$. Hence $Z(M_H)=\langle\mu\rangle$.

Since \(\mu\) is central in both \(M_G\) and \(D\),
it is also central in their amalgam \(M_H\).
Quotienting \(M_H\) by \(\langle\mu\rangle\)
therefore amounts to imposing the relation
\(\mu=1\) in the amalgam presentation.

The edge group
\(A=\langle\mu,h\rangle\cong\mathbb Z^2\)
then becomes \(\langle h\rangle\cong\mathbb Z\).
This subgroup still embeds in both vertex-group
quotients, since the only elements of \(A\)
killed by either quotient are powers of \(\mu\).
Moreover,
\[
D/\langle\mu\rangle
=\langle h,\ell\rangle\cong\mathbb Z^2.
\]
Thus the quotient is the amalgam of
\(M_G/\langle\mu\rangle\) and
\(\langle h,\ell\rangle\) along their common
subgroup \(\langle h\rangle\), giving
\eqref{eq:meridian-quotient}.

We now compute the first Betti number of this quotient.
Put $Q=M_G/\langle\mu\rangle$. Since
$\langle h,\ell\rangle\cong\mathbb Z^2$, its rational
first homology has basis $[h]$ and $[\ell]$.
Abelianizing the amalgam identifies $[h]$ in this
second vertex group with $h$ in
$H_1(Q;\mathbb Q)$. Thus $[\ell]$ contributes one
new independent generator, giving
\[
H_1(M_H/\langle\mu\rangle;\mathbb Q)
\cong H_1(Q;\mathbb Q)\oplus\mathbb Q.
\]
This proves \eqref{eq:b1-meridian-composite}.
By assumption, this first Betti number is at least
$g+1\ge4$, whereas every classical knot group has
first Betti number one. Since
$Z(M_H)=\langle\mu\rangle$, the quotient
$M_H/Z(M_H)$ cannot be a knot group.
\end{proof}

\section{Product vertex groups determined by meridian-marked
centralizer pairs}\label{sec:recognition}

The successive rim surgeries give amalgam decompositions
\[
G_i=G_{i-1}*_{\langle\mu,h_i\rangle}B_i,
\qquad
B_i=K_i\times\langle h_i\rangle.
\]
These decompositions tell us which product vertex group
is introduced at each stage. However, an automorphism of the
exterior group need not preserve them. Our goal in this
section is to recognize the subgroups $B_i$ from the
exterior group and its distinguished meridian $\mu$ alone.
This will allow us to recover the surgery data without
assuming that the original amalgam decompositions are
preserved.

Centralizers and the centers of vertex groups have been used
to recognize group decompositions in earlier work
\cite[Lemmas~3.5--3.7]{CrokeKleiner},
\cite[Lemma~11 and Proposition~13]{CrispArtin}, and
\cite[Theorem~4.14]{FrigerioLafontSisto}.
These results do not apply directly here, since our
attaching subgroup has rank two, while the corresponding
peripheral subgroup has rank three. Instead, we identify
the vertex groups using centralizers of
null-homologous elements and their quotients by the center,
together with the distinguished meridian.

The centralizer calculations of \cref{sec:centralizers}
provide the key idea. 
For the product vertex group $B=K\times\langle h\rangle$
in such an amalgam, the centralizer calculations give
\[
B=C_H(h),\qquad Z(B)=\langle h\rangle,
\qquad (B/Z(B),\bar\mu)\cong(K,m).
\]
Thus the knot group, together with its meridian, can be
recovered from a suitable centralizer. The difficulty is
to distinguish these centralizers from others in the
exterior group. We do this by considering centralizers
of nontrivial null-homologous elements, together with
a distinguished meridian element.

We begin by verifying the necessary edge-group properties
and computing the meridian centralizers throughout the
construction. 
We then use centralizers and the distinguished meridian
to recognize the product vertex group introduced by
a single surgery, and that previously recognized groups
remain recognizable after further surgeries. 
Finally, \cref{thm:tower-spectrum} shows that if the
meridian-marked knot groups used in the vertical-boundary
and subsequent surgeries are pairwise nonisomorphic,
then every automorphism fixing the meridian preserves the
conjugacy class of each product vertex group $B_i$.
In the next section, we will use this result to recover
the exterior group before each surgery by taking
a quotient preserved by these automorphisms.

Throughout, an isomorphism of meridian-marked groups
sends the distinguished meridian to the distinguished
meridian. The product vertex groups themselves are
recognized only up to conjugacy.

To apply the preceding results to all successive surgeries,
we use a single index for the vertical-boundary surgeries
and the subsequent surgeries along $d_1,\ldots,d_N$.
The preliminary surgeries have already produced $A_0$
and are not included in this indexing. We retain the
existing notation for each surgery.

\begin{definition}[Notation for successive surgeries]
\label{def:clean-tower}
Let $v=1$ in even genus and $v=2$ in odd genus.
Starting with $\Gamma_0=A_0$, let $\Gamma_r$ be the
exterior group after the first $r$ surgeries following
the preliminary ones. Thus
\[
\begin{array}{ll}
\text{even genus:}&
\Gamma_1=G_0,\\
\text{odd genus:}&
\Gamma_1=G_0^{(1)},\quad \Gamma_2=G_0,
\end{array}
\]
and
\[
\Gamma_{v+i}=G_i\qquad(1\le i\le N).
\]

We keep the notation for the two kinds of surgeries.
For $1\le r\le v$, set
\[
\cB_r=C_r^V,\qquad
\cE_r=\langle\mu,w_r\rangle,\qquad
\eta_r=w_r,\qquad
\mathcal K_r=(K_r^V,m_r^V),
\]
where $w_1=w$ in even genus. For $1\le i\le N$, set
\[
\cB_{v+i}=B_i,\qquad
\cE_{v+i}=\langle\mu,h_i\rangle,\qquad
\eta_{v+i}=h_i,\qquad
\mathcal K_{v+i}=(K_i,m_i).
\]
Thus $\mathcal K_r$ denotes a meridian-marked knot group, and
\[
\Gamma_r=\Gamma_{r-1}*_{\cE_r}\cB_r
\qquad(1\le r\le v+N).
\]
All edge maps are injective, so we regard the earlier groups
$\Gamma_s$ and $\cB_s$ as subgroups of the later ones.
\end{definition}

Note that the preliminary knot groups $K_j^D$ are not included
in the collection $\mathcal K_r$.

\begin{lemma}
\label{lem:disk-meridian-powers}
For every $n\ne0$,
\[
 C_{A_0}(\mu^n)=C_{A_0}(\mu)=\langle\mu\rangle\times U_0.
\]
Moreover,
\[
 Z(C_{A_0}(\mu))=\langle\mu\rangle,
 \qquad b_1(C_{A_0}(\mu)/\langle\mu\rangle)=g.
\]
\end{lemma}

\begin{proof}
Recall that $A_0$ is the amalgam of the knot groups
$K_1^D,\ldots,K_g^D$ along their common meridian
subgroup $\langle\mu\rangle$. We view it as the
fundamental group of a star-shaped graph of groups,
with central vertex group $\langle\mu\rangle$,
leaf groups $K_j^D$, and edge groups
$\langle\mu\rangle$.

Let $T_D$ be its Bass--Serre tree, and let $e_j$
be the base edge joining the central vertex to
the $K_j^D$-vertex. Fix $n\ne0$ and put
$C=C_{A_0}(\mu^n)$. Every edge of type $j$ has the form $qe_j$ for
some $q\in A_0$. Such an edge is fixed by $\mu^n$
exactly when
\[
q^{-1}\mu^nq=\mu^k
\]
for some integer $k$. Since
$H_1(A_0;\mathbb Z)=\mathbb Z\langle[\mu]\rangle$,
abelianization gives $k=n$. Thus $q\in C$.
Conversely, every $C$-translate of $e_j$ is fixed
by $\mu^n$. 
Thus, for each $j=1,\ldots,g$, the edges
$qe_j$ fixed by $\mu^n$ are precisely those
with $q\in C$.

The central vertex is fixed by $\mu^n$.
If another vertex is fixed, the unique path
joining it to the central vertex is also fixed.
Thus every fixed vertex belongs to a fixed edge.
It follows that the quotient
$\Fix_{T_D}(\mu^n)/C$ is the same star-shaped graph,
with one central vertex and $g$ leaves.

The central vertex has stabilizer $\langle\mu\rangle$,
and each edge has the same stabilizer. By
\cref{lem:knot-centralizers}, the stabilizer
of the $j$-th leaf is
\[
C_{K_j^D}(\mu^n)=\langle\mu,\ell_j^D\rangle.
\]
Reconstructing $C$ from this graph of groups gives
\[
\begin{aligned}
C_{A_0}(\mu^n)
&=\langle\mu,\ell_1^D\rangle
  *_{\langle\mu\rangle}\cdots
  *_{\langle\mu\rangle}
  \langle\mu,\ell_g^D\rangle\\
&=\langle\mu\rangle\times U_0
\end{aligned}
\]
for every $n\ne0$.

Finally, $U_0$ is a free group of rank $g\ge3$
by \cref{prop:disk-group}, so it has trivial center
and $b_1(U_0)=g$. The remaining conclusions follow
from the direct-product decomposition above.
\end{proof}

\begin{lemma}[Centralizers and transporters of longitudes]
\label{lem:disk-leaf-power-clean}
For $c=\ell_j^D$ and $b\ne0$,
\[
 C_{A_0}(\mu^ac^b)=\langle\mu,c\rangle
 =\Trans_{A_0}(\mu^ac^b,\langle\mu,c\rangle).
\]
\end{lemma}

\begin{proof}
Put $P_j=\langle\mu,c\rangle$ and $x=\mu^ac^b$,
where $b\ne0$. By \cref{lem:knot-centralizers},
$C_{K_j^D}(x)=P_j$, and $P_j$ is malnormal in $K_j^D$.

We first show that no element of $K_j^D$ such that it conjugates $x$ to a power of $\mu$. Indeed, if
$k^{-1}xk=\mu^m$, then $m\ne0$ and
$k^{-1}P_jk\cap P_j\ne1$. By malnormality, $k\in P_j$.
But $P_j$ is abelian, so $k^{-1}xk=x$, contradicting
$b\ne0$.

Since $x\in K_j^D$, it fixes the base $K_j^D$-vertex
of $T_D$. Every edge meeting this vertex has the form
$ke_j$ for some $k\in K_j^D$. If $x$ fixed such an edge,
then $k^{-1}xk\in\langle\mu\rangle$, contradicting
what we just proved. If $x$ fixed another vertex, it would also fix
the path joining that vertex to the
$K_j^D$-vertex, including an edge incident to
the latter. This is impossible. Thus the base
$K_j^D$-vertex is the only vertex fixed by $x$. Every element commuting with $x$ must fix
that vertex and hence belong to $K_j^D$. Thus
\(
C_{A_0}(x)=C_{K_j^D}(x)=P_j.
\)

Now suppose $q^{-1}xq\in P_j$. This conjugate also
has a unique fixed vertex. It cannot be a nonzero
power of $\mu$, since such a power fixes an edge.
Thus $q^{-1}xq\in P_j\setminus\langle\mu\rangle$.
The preceding argument applies to every element
of this set, so its unique fixed vertex is again
the base $K_j^D$-vertex. Hence $q$ fixes that
vertex and lies in $K_j^D$.

Since both $x$ and $q^{-1}xq$ are nontrivial
elements of $P_j$, malnormality gives $q\in P_j$. Conversely, every element of $P_j$ commutes with $x$. Therefore
\(
\Trans_{A_0}(x,P_j)=C_{A_0}(x)=P_j,
\)
as desired.
\end{proof}

\begin{lemma}[Vertical-boundary edge groups]
\label{lem:seam-power-clean}
In even genus, $\langle\mu,w\rangle$ has the
centralizer--transporter property in $A_0$.

In odd genus, $\langle\mu,w_1\rangle$ has this
property in $A_0$. After the first vertical-boundary
surgery, $\langle\mu,w_2\rangle$ has the same property
in
\[
G_0^{(1)}=A_0*_{E_1}C_1^V,
\qquad E_1=\langle\mu,w_1\rangle.
\]
\end{lemma}

\begin{proof}
Let $1\ne u\in U_0$ be an element that is not a
proper power. We first prove the property for
$\langle\mu,u\rangle$ in $A_0$. By \cref{lem:disk-meridian-powers},
$C_{A_0}(\mu)=\langle\mu\rangle\times U_0$.
The quotient group $U_0$ acts on
$\Fix_{T_D}(\mu)$ as the Bass--Serre group of
a star with leaf groups $\langle\ell_j^D\rangle$
and trivial central and edge groups. Suppose $u$ is hyperbolic on this tree, with
axis $L$. By the calculation in
\cref{lem:disk-meridian-powers}, every edge of
$\Fix_{T_D}(\mu)$ has the form $qe_j$, where
$q$ commutes with $\mu$. Its stabilizer in
$A_0$ is therefore
$q\langle\mu\rangle q^{-1}=\langle\mu\rangle$.
Consequently,
\(
\Fix_{A_0}^{\mathrm{pt}}(L)=\langle\mu\rangle.
\)

By \cref{lem:reflection-exclusion}, every element
of $\Stab_{A_0}(L)$ commutes with $\mu$, and none of them is a 
reflection of $L$. Modulo $\langle\mu\rangle$, this
stabilizer acts faithfully by translations, so
its image is an infinite cyclic group. Since $u$ is not
a proper power, its translation
generates this image. Hence
\(
\Stab_{A_0}(L)=\langle\mu,u\rangle.
\)
The first alternative of
\cref{lem:clean-axis-criterion} now gives the
centralizer--transporter property.

If $u$ is elliptic, it is conjugate in $U_0$
to $(\ell_j^D)^n$ for some $n\ne0$. Since $u$
is not a proper power, $n=\pm1$.
The conjugating element belongs to $U_0$ and
therefore commutes with $\mu$.
Thus \cref{lem:disk-leaf-power-clean} gives
the same property.

By \cref{lem:free-boundary-words}, the boundary
elements $w$ in even genus and $w_1,w_2$ in odd
genus are nontrivial and not proper powers.
This proves the first half of the lemma.

It remains to consider the second vertical-boundary
surgery in odd genus. Put
\[
E_i=\langle\mu,w_i\rangle\qquad(i=1,2).
\]
We show that $E_2$ retains the property in
$G_0^{(1)}=A_0*_{E_1}C_1^V$. 
Fix $x=\mu^aw_2^b$ with $b\ne0$.
We claim that $q^{-1}xq\notin E_1$ for every
$q\in A_0$. Suppose otherwise that
\[
q^{-1}xq=\mu^{a'}w_1^{b'}
\]
for some $q\in A_0$ and $a',b'\in\mathbb Z$. Since $w_1,w_2\in U_0\subset[A_0,A_0]$,
taking homology classes in
$H_1(A_0;\mathbb Z)=\mathbb Z\langle[\mu]\rangle$
gives
\[
a[\mu]=[x]=[q^{-1}xq]=a'[\mu].
\]
Since $[\mu]$ has infinite order, $a'=a$.
 If $b'=0$,
the right side must be a nonzero power of
$\mu$, whose centralizer in $A_0$ is
nonabelian by \cref{lem:disk-meridian-powers}.
But $C_{A_0}(x)=E_2$ is abelian which is a contradiction.
Thus $b'\ne0$.

Since $b,b'\ne0$, the property already proved
in $A_0$ gives
\[
C_{A_0}(x)=E_2,
\qquad
C_{A_0}(q^{-1}xq)=E_1.
\]
But
\[
C_{A_0}(q^{-1}xq)
=q^{-1}C_{A_0}(x)q.
\]
Therefore $q^{-1}E_2q=E_1$. In particular, $q^{-1}\mu q\in E_1$.
Since its homology class is $[\mu]$, we can write
\[
q^{-1}\mu q=\mu w_1^s
\]
for some $s\in\mathbb Z$.
If $s\ne0$, the right side has centralizer
$E_1$ which is abelian, whereas the centralizer
of $\mu$ is nonabelian. Hence $s=0$, so
$q$ commutes with $\mu$. Passing to
$C_{A_0}(\mu)/\langle\mu\rangle=U_0$ now gives
\[
\bar q^{-1}w_2^b\bar q=w_1^{b'},
\]
contrary to the noncommensurability of $w_1$
and $w_2$. 
Therefore $q^{-1}xq\notin E_1$ for every $q\in A_0$.

Consider the Bass--Serre tree of
$G_0^{(1)}=A_0*_{E_1}C_1^V$.
The element $x\in A_0$ fixes the base
$A_0$-vertex. Every edge incident to this
vertex has the form $qE_1$, with $q\in A_0$.
If $x$ fixed such an edge, then
$q^{-1}xq\in E_1$.
Thus $x$ fixes no incident edge and hence
no other vertex. Every element commuting
with $x$ must therefore fix the base
$A_0$-vertex. It follows that
\[
C_{G_0^{(1)}}(x)=C_{A_0}(x)=E_2.
\]

Finally, suppose $y^{-1}xy\in E_2$.
Since $x$ fixes exactly one vertex, so does
its conjugate. The conjugate cannot be a
nonzero power of $\mu$, since such a power
fixes an edge. Hence
\[
y^{-1}xy\in E_2\setminus\langle\mu\rangle.
\]
The preceding argument applies to every
element of this set, so its unique fixed
vertex is the base $A_0$-vertex. Since the
unique fixed vertex of $y^{-1}xy$ is also
$y^{-1}A_0$, we have $y^{-1}A_0=A_0$.
Thus $y\in A_0$, and the transporter
property already proved in $A_0$ gives
$y\in E_2$.

Conversely, every element of $E_2$ commutes
with $x$. Therefore
\[
\Trans_{G_0^{(1)}}(x,E_2)
=C_{G_0^{(1)}}(x)=E_2,
\]
as required.
\end{proof}

The preceding lemmas, together with
\cref{cor:carrier-induction,prop:meridian-composite},
also give an independent calculation of the meridian
centralizers. In particular, their first Betti numbers
record the number of surgeries after the preliminary stage.

\begin{proposition}[Meridian centralizers throughout the construction]
\label{prop:tower-meridian-centralizers}
For every $0\le r\le v+N$ and every $n\ne0$,
\[
C_{\Gamma_r}(\mu^n)=C_{\Gamma_r}(\mu),
\qquad
Z(C_{\Gamma_r}(\mu))=\langle\mu\rangle,
\]
and
\[
b_1(C_{\Gamma_r}(\mu)/\langle\mu\rangle)=g+r.
\]

For each $1\le r\le v+N$, the surgery from
$\Gamma_{r-1}$ to $\Gamma_r$ gives
\[
C_{\Gamma_r}(\mu)
=C_{\Gamma_{r-1}}(\mu)*_{\cE_r}
C_{\cB_r}(\mu).
\]
After quotienting by $\langle\mu\rangle$, this becomes
\[
C_{\Gamma_r}(\mu)/\langle\mu\rangle
\cong
\bigl(C_{\Gamma_{r-1}}(\mu)/\langle\mu\rangle\bigr)
*_{\langle\eta_r\rangle}
\bigl(C_{\cB_r}(\mu)/\langle\mu\rangle\bigr).
\]
\end{proposition}

\begin{proof}
We proceed by induction on $r$. The case $r=0$
is \cref{lem:disk-meridian-powers}.

Suppose $1\le r\le v+N$ and the conclusions hold
for $\Gamma_{r-1}$. Recall that
\[
\Gamma_r=\Gamma_{r-1}*_{\cE_r}\cB_r.
\]
The edge group $\cE_r$ has the
centralizer--transporter property in $\Gamma_{r-1}$.
For the vertical-boundary surgeries, this follows
from \cref{lem:seam-power-clean}. In odd genus, the
second surgery uses the property in
$\Gamma_1=G_0^{(1)}$. For the surgeries along $d_i$,
the property follows from \cref{cor:carrier-induction}.
The required homological conditions hold at every
stage by \cref{prop:early-homology}.

Since $U_0$ remains embedded in $\Gamma_{r-1}$
and commutes with $\mu$, the group
$C_{\Gamma_{r-1}}(\mu)$ is nonabelian.
By induction,
\[
C_{\Gamma_{r-1}}(\mu^n)
=C_{\Gamma_{r-1}}(\mu)
\qquad(n\ne0).
\]
Thus \eqref{eq:centralizer-meridian-composite}
applies to the current surgery and gives
\[
C_{\Gamma_r}(\mu^n)
=C_{\Gamma_{r-1}}(\mu)*_{\cE_r}
C_{\cB_r}(\mu)
=C_{\Gamma_r}(\mu)
\qquad(n\ne0).
\]

The inductive hypothesis also gives
\[
Z(C_{\Gamma_{r-1}}(\mu))=\langle\mu\rangle,
\qquad
b_1(C_{\Gamma_{r-1}}(\mu)/\langle\mu\rangle)
=g+r-1\ge g.
\]
These are the remaining hypotheses of
\cref{prop:meridian-composite}, which then gives the center and quotient formula in the
statement, and shows that
\[
b_1(C_{\Gamma_r}(\mu)/\langle\mu\rangle)
=(g+r-1)+1=g+r.
\]
This completes the induction.
\end{proof}

We next prove the meridian-conjugacy statements needed for recognition.

\begin{lemma}[Meridian conjugacy inside the vertex groups]
\label{cor:meridian-reflection}
Recall the amalgam \eqref{eq:clean-extension}:
\[
H=G*_{A}B,
\qquad
A=\langle\mu,h\rangle\cong\mathbb Z^2,
\qquad
B=K\times\langle h\rangle,
\]
where $K$ is a hyperbolic knot group whose meridian
is identified with $\mu$. Assume
Conditions~\ref{cond:powerclean} and~\ref{cond:homology},
and suppose $C_G(\mu)$ is nonabelian. Then
\[
(H\mathbin{\cdot}\mu)\cap G=G\mathbin{\cdot}\mu,
\qquad
(H\mathbin{\cdot}\mu)\cap B=B\mathbin{\cdot}\mu,
\]
where $Q\mathbin{\cdot}\mu$ denotes the conjugacy
class of $\mu$ in $Q$.
\end{lemma}

\begin{proof}
If $u=q\mu q^{-1}\in G$ for some $q\in H$,
applying the retraction $r:H\to G$ gives
\(
u=r(q)\mu r(q)^{-1}.
\)
Thus $u$ is conjugate to $\mu$ in $G$.

Now suppose
\[
u=(k,h^n)=q\mu q^{-1}
\]
for some $k\in K$, $n\in\mathbb Z$, and $q\in H$.
Recall that $\phi_K:K\to\mathbb Z$ is the
abelianization sending the knot meridian $m$ to $1$.
Applying the retraction $r:H\to G$ gives
\[
\mu^{\phi_K(k)}h^n=r(q)\mu r(q)^{-1}.
\]
Taking homology classes in
$H_1(G;\mathbb Z)=\mathbb Z\langle[\mu]\rangle$
and using $[h]=0$, we obtain $\phi_K(k)=1$.
Hence
\(
\mu h^n=r(q)\mu r(q)^{-1}.
\)
If $n\ne0$, Condition~\ref{cond:powerclean}
gives $C_G(\mu h^n)=A$, which is abelian.
But the right side has centralizer
$r(q)C_G(\mu)r(q)^{-1}$, which is nonabelian.
Therefore $n=0$.

It remains to show that $k$ is conjugate to $m$
in $K$. Let $\phi:G\to\mathbb Z$ be the
abelianization sending $\mu$ to $1$.
The homomorphisms
\[
G\to K,\qquad g\mapsto m^{\phi(g)},
\]
and
\[
B\to K,\qquad (k',h^j)\mapsto k',
\]
agree on $A$: both send $\mu$ to $m$ and $h$
to $1$. They therefore define a homomorphism
$\psi:H\to K$. Applying $\psi$ to
$u=q\mu q^{-1}$ gives
\(
k=\psi(q)m\psi(q)^{-1}.
\)
Thus $u=(k,1)$ is conjugate to $\mu=(m,1)$
in $B$.

The reverse inclusions follow from
$G,B\le H$.
\end{proof}

\begin{definition}[Meridian-marked centralizer pairs]
\label{def:marked-centralizer}
Let $(Q,\mu)$ be a meridian-marked group and $(K,m)$ a
meridian-marked nonabelian hyperbolic knot group.
Pairs $(C,u)$ and $(C',u')$, with $C,C'\le Q$ and $u\in C$, $u'\in C'$,
are called \emph{simultaneously conjugate} if there is $q\in Q$ satisfies
\[
 C'=qCq^{-1},\qquad u'=quq^{-1}.
\]
A simultaneous conjugacy class \emph{realizes $(K,m)$} if
\begin{enumerate}[label=\textup{(\roman*)},leftmargin=2.3em]
\item $C=C_Q(x)$ for some $1\ne x\in[Q,Q]$;
\item $u\in C$ is conjugate in $Q$ to $\mu$;
\item $C$ is nonabelian;
\item $Z(C)$ is infinite cyclic;
\item $(C/Z(C),\bar u)\cong(K,m)$ as meridian-marked groups.
\end{enumerate}
The isomorphism in
\textup{(v)} sends $\bar u$ to $m$.
\end{definition}

These conditions are invariant under meridian-marked isomorphisms.
Indeed, an isomorphism $\Phi:(Q,\mu)\to(Q',\mu')$ carries $[Q,Q]$
onto $[Q',Q']$, sends $C_Q(x)$ to $C_{Q'}(\Phi(x))$, and carries
centers onto centers.  It also preserves meridian conjugacy and induces
the marked quotient isomorphism.  
Applying the same argument to $\Phi^{-1}$ shows that
$\Phi$ induces a bijection between the simultaneous
conjugacy classes realizing $(K,m)$ in $Q$ and those
in $Q'$.

\begin{theorem}[Recognition of a product vertex group]
\label{thm:terminal-recognition}
Let
\[
H=G*_{A}B,\qquad
A=\langle\mu,h\rangle\cong\mathbb Z^2,
\qquad B=K\times\langle h\rangle,
\]
where $K$ is a nonabelian hyperbolic knot group
with meridian $m$ identified with $\mu$.
Assume Conditions~\ref{cond:powerclean}
and~\ref{cond:homology}, and suppose
$C_G(\mu)$ is nonabelian.

If no simultaneous conjugacy class in $G$
realizes $(K,m)$, then the class of $(B,\mu)$
is the unique class in $H$ realizing $(K,m)$.

Moreover, let $(K',m')$ be a meridian-marked
nonabelian hyperbolic knot group not isomorphic
to $(K,m)$. If the class of $(C,u)$ in $H$
realizes $(K',m')$, then there exists $q\in H$
such that $qCq^{-1}\le G$, and the class of
$(qCq^{-1},quq^{-1})$ in $G$ also realizes
$(K',m')$.
\end{theorem}

\begin{proof}
Let $(C,u)$ represent a simultaneous conjugacy class
in $H$ realizing a meridian-marked nonabelian
hyperbolic knot group $(K',m')$.
By \cref{def:marked-centralizer}, we have
\[
C=C_H(x),\qquad 1\ne x\in[H,H],
\]
and $u$ is conjugate to $\mu$ in $H$. Apply \cref{prop:complete-centralizers} to $x$,
conjugating $x$, $C$, and $u$ by the same element
of $H$ when necessary. 
If $x$ is hyperbolic, \cref{cor:hyperbolic-not-knot}
shows that $C/Z(C)$ cannot be a nonabelian
hyperbolic knot group. If $x$ fixes only
a $B$-vertex, its centralizer is abelian.
Thus only two cases remain to consider.

Suppose first that $x$ fixes only the base
$G$-vertex. Then
\(
C=C_G(x)\le G,
\)
so $u\in G$. Since the retraction $r:H\to G$
fixes $x$ and $x\in[H,H]$, we also have
$x\in[G,G]$. By \cref{cor:meridian-reflection},
$u$ is conjugate to $\mu$ in $G$.
All the conditions of \cref{def:marked-centralizer}
therefore hold in $G$. Hence the simultaneous
conjugacy class of $(C,u)$ in $G$ realizes
$(K',m')$.

Now suppose that $x$ fixes an edge. After
conjugation, we may write
\[
x=\mu^ah^b\in A,\qquad (a,b)\ne(0,0).
\]
Since $x\in[H,H]$ and $[h]=0$, taking homology
classes gives $a=0$. Thus $b\ne0$, and
\cref{prop:complete-centralizers} gives
\(
C=C_H(h^b)=B.
\)
By \cref{cor:meridian-reflection}, there is
$q\in B$ such that $q^{-1}uq=\mu$.
Since $q$ preserves $B$ by conjugation,
$(B,u)$ and $(B,\mu)$ are simultaneously
conjugate.

By \cref{lem:carrier-primitive},
\[
Z(B)=\langle h\rangle,\qquad
(B/Z(B),\bar\mu)\cong(K,m).
\]
Moreover, $1\ne h\in[H,H]$ and $C_H(h)=B$.
Thus the simultaneous conjugacy class of
$(B,\mu)$ realizes $(K,m)$.

We have shown that every simultaneous conjugacy
class in $H$ satisfying \cref{def:marked-centralizer}
either comes from $G$ as described above or is
the class of $(B,\mu)$. Since no class in $G$
realizes $(K,m)$ by hypothesis, the latter is
the unique class in $H$ realizing $(K,m)$.
\end{proof}

\begin{proposition}[The preliminary group]
\label{prop:disk-spectrum}
No simultaneous conjugacy class of pairs in $A_0$
realizes a meridian-marked nonabelian hyperbolic
knot group in the sense of \cref{def:marked-centralizer}.
\end{proposition}

\begin{proof}
Let $1\ne x\in[A_0,A_0]$, and consider its action
on the Bass--Serre tree $T_D$ by $A_0$. Suppose first that $x$ is hyperbolic with axis $L$.
Every element of $C_{A_0}(x)$ preserves $L$ and
acts on it by translation. The kernel of this
action fixes $L$ pointwise, so it lies in an
edge stabilizer, which is infinite cyclic.
Thus $C_{A_0}(x)$ is metabelian. As in the proof of
\cref{cor:hyperbolic-not-knot}, its quotient by
its center is solvable and therefore cannot be
a nonabelian hyperbolic knot group.

Now suppose $x$ is elliptic. Since
$x\in[A_0,A_0]$ and every edge stabilizer
is a conjugate of $\langle\mu\rangle$,
$x$ cannot fix an edge. As in the proof of \cref{lem:disk-leaf-power-clean},
$x$ therefore fixes exactly one vertex, and
every element commuting with $x$ fixes that vertex.
It must be a $K_j^D$-vertex, since $x$ cannot
fix a vertex whose stabilizer is a conjugate
of $\langle\mu\rangle$.
After conjugating $x$, we may assume that
this is the base $K_j^D$-vertex. Hence
\(
C_{A_0}(x)=C_{K_j^D}(x),
\)
which is abelian by \cref{lem:knot-centralizers}.

In either case, $C_{A_0}(x)$ does not satisfy the conditions
of \cref{def:marked-centralizer}. 
\end{proof}

\begin{lemma}[Persistence of earlier centralizers]
\label{lem:old-block-not-enlarge}
Let $H=G*_A B$ be as in \eqref{eq:clean-extension},
and assume Conditions~\ref{cond:powerclean}
and~\ref{cond:homology}. Let
\[
B_{\mathrm{old}}=K_{\mathrm{old}}\times\langle z\rangle
\le G
\]
be a product vertex group from an earlier surgery,
where $K_{\mathrm{old}}$ is a nonabelian hyperbolic
knot group. Suppose $z\in[G,G]$ and
\[
C_G(z^n)=B_{\mathrm{old}}\qquad(n\ne0).
\]
Then no conjugate in $G$ of $z^n$ lies in $A$
for $n\ne0$. Moreover,
\[
C_H(z^n)=C_G(z^n)=B_{\mathrm{old}}
\qquad(n\ne0).
\]
\end{lemma}

\begin{proof}
Fix $n\ne0$. We first show that
$q^{-1}z^nq\notin A$ for every $q\in G$.
Suppose otherwise that
\(
q^{-1}z^nq=\mu^ah^b
\)
for some $q\in G$ and $a,b\in\mathbb Z$.
Since $z\in[G,G]$ and $[h]=0$, taking
homology classes in
$H_1(G;\mathbb Z)=\mathbb Z\langle[\mu]\rangle$
gives $a=0$. Also, $b\ne0$, since $z$ has
infinite order.

By Condition~\ref{cond:powerclean},
$C_G(h^b)=A$, which is abelian.
On the other hand,
\[
C_G(h^b)=q^{-1}C_G(z^n)q
=q^{-1}B_{\mathrm{old}}q,
\]
which is nonabelian. This contradiction proves
that $q^{-1}z^nq\notin A$ for every $q\in G$.

Now consider the Bass--Serre tree of
$H=G*_A B$. The element $z^n\in G$ fixes
the base $G$-vertex. Every edge meeting
this vertex has the form $qA$ for some $q\in G$.
If $z^n$ fixed such an edge, then
$q^{-1}z^nq\in A$, which we have ruled out.
Thus $z^n$ fixes no incident edge and hence
no other vertex.

Every element commuting with $z^n$ must
therefore fix the base $G$-vertex and belong
to $G$. Consequently,
\[
C_H(z^n)=C_G(z^n)=B_{\mathrm{old}}.
\]
Since $G$ embeds in $H$, the element $z$
remains of infinite order and belongs to
$[H,H]$.
\end{proof}

\begin{theorem}[Recognition of all product vertex groups]
\label{thm:tower-spectrum}
Use the notation of \cref{def:clean-tower}.
Assume that the meridian-marked knot groups
\[
\mathcal K_1,\ldots,\mathcal K_{v+N}
\]
used in the vertical-boundary surgeries and
the subsequent surgeries along $d_1,\ldots,d_N$
are pairwise nonisomorphic. For every $0\le r\le v+N$, the following hold.
\begin{enumerate}[label=\textup{(\roman*)},leftmargin=2.3em]

\item For every $1\le s\le r$, the center
$Z(\cB_s)$ is infinite cyclic and lies in
$[\Gamma_r,\Gamma_r]$. Moreover,
\[
C_{\Gamma_r}(z)=\cB_s
\qquad
(1\ne z\in Z(\cB_s)).
\]

\item Up to meridian-marked isomorphism, the
knot groups realized in $\Gamma_r$ are exactly
\[
\mathcal K_1,\ldots,\mathcal K_r.
\]
For each $1\le s\le r$, the class of
$(\cB_s,\mu)$ is the unique simultaneous
conjugacy class in $\Gamma_r$ that realizes
$\mathcal K_s$.

\item Every $\Phi\in\Aut(\Gamma_r)$ fixing
$\mu$ satisfies, for each $1\le s\le r$,
\[
\Phi(\cB_s)=q\cB_sq^{-1}
\]
for some $q\in C_{\Gamma_r}(\mu)$.
\end{enumerate}
\end{theorem}

\begin{proof}
For $r=0$, \cref{prop:disk-spectrum} gives
\textup{(ii)}. Suppose the conclusions hold for $\Gamma_{r-1}$,
and write
\[
\Gamma_r=\Gamma_{r-1}*_{\cE_r}\cB_r.
\]
The edge group $\cE_r$ satisfies the
centralizer--transporter property by
\cref{lem:seam-power-clean,cor:carrier-induction}.
The homological conditions follow from
\cref{prop:early-homology}. Moreover,
$C_{\Gamma_{r-1}}(\mu)$ is nonabelian,
since it contains $U_0$.

Let $z$ generate the infinite cyclic factor
of $\cB_r$. By construction, $z$ is nontrivial
and belongs to $[\Gamma_{r-1},\Gamma_{r-1}]$.
By \cref{prop:complete-centralizers,lem:carrier-primitive},
\[
C_{\Gamma_r}(z^n)=\cB_r\qquad(n\ne0),
\qquad Z(\cB_r)=\langle z\rangle.
\]
For $s<r$, the inductive hypothesis and
\cref{lem:old-block-not-enlarge} show that
\[
C_{\Gamma_r}(z')=\cB_s
\qquad(1\ne z'\in Z(\cB_s)).
\]
Since $\cB_s$ remains embedded in $\Gamma_r$,
its center is unchanged and still lies in
$[\Gamma_r,\Gamma_r]$. This proves \textup{(i)}.

By induction, $\Gamma_{r-1}$ realizes exactly
$\mathcal K_1,\ldots,\mathcal K_{r-1}$, none of
which is isomorphic to $\mathcal K_r$.
By \cref{thm:terminal-recognition}, the class
of $(\cB_r,\mu)$ is the unique class realizing
$\mathcal K_r$, and every class realizing a
different marked knot group has a representative
in $\Gamma_{r-1}$ satisfying
\cref{def:marked-centralizer}.

By \textup{(i)}, each earlier
$(\cB_s,\mu)$ still represents a class realizing
$\mathcal K_s$. 
By induction and \cref{thm:terminal-recognition},
each earlier knot group is still realized by exactly
one simultaneous conjugacy class in $\Gamma_r$,
and no knot groups other than
$\mathcal K_1,\ldots,\mathcal K_r$ are realized.
 This proves \textup{(ii)}.

Finally, let $\Phi\in\Aut(\Gamma_r)$ fix $\mu$.
Then $(\Phi(\cB_s),\mu)$ represents a class
realizing $\mathcal K_s$. By uniqueness,
there is $q\in\Gamma_r$ such that
\[
\Phi(\cB_s)=q\cB_sq^{-1},
\qquad \mu=q\mu q^{-1}.
\]
Thus $q\in C_{\Gamma_r}(\mu)$, proving
\textup{(iii)}.
\end{proof}

\section{Commutator quotients invariant under automorphisms fixing the meridian}\label{sec:collapse}

Fix $1\le i\le N$. Recall that surgery
along $d_i$ gives
\[
G_i=G_{i-1}*_{A_i}B_i,\qquad
A_i=\langle\mu,h_i\rangle,\qquad
B_i=K_i\times\langle h_i\rangle.
\]
An automorphism of $G_i$ is called
\emph{meridian-fixing} if it fixes $\mu$.

By \cref{thm:tower-spectrum}, every
meridian-fixing automorphism of $G_i$
preserves the conjugacy class of $B_i$.
We show that it also preserves the kernel
of the canonical retraction
$r_i:G_i\to G_{i-1}$.

The knot-group commutator quotient proved below is analogous to
the untying maps for satellite knot groups; compare
\cite[Proposition~3.4]{SilverWhittenKnotGroups}.

\begin{proposition}[The invariant commutator quotient]
\label{prop:characteristic-kernel}
The conjugacy class of \(B_i\) is invariant under meridian-fixing
automorphisms of \(G_i\), and the normal subgroup
\begin{equation}\label{eq:terminal-kernel}
 \cN_i
 =\normal{[B_i,B_i]}_{G_i}
 =\normal{[K_i,K_i]}_{G_i}
 =\ker r_i
\end{equation}
is invariant under every such automorphism.
The canonical retraction identifies the quotient by
\begin{equation}\label{eq:marked-collapse}
 G_i/\cN_i\xrightarrow{\cong}G_{i-1},\qquad
 x\cN_i\longmapsto r_i(x).
\end{equation}
The retraction $r_i$ fixes the meridian and
recovers the push-off homomorphism from the
preceding stage:
\[
r_i(\mu)=\mu,\qquad
r_i\circ\rho_i=\rho_{i-1}.
\]
\end{proposition}

\begin{proof}
Let \(\Phi\in\Aut(G_i)\) fix \(\mu\). Since $G_i=\Gamma_{v+i}$ and
$B_i=\cB_{v+i}$, \cref{thm:tower-spectrum}
\textup{(iii)} applies to $\Phi$ and gives
$g\in C_{G_i}(\mu)$ such that
\[
\Phi(B_i)=gB_ig^{-1}.
\]
Thus \(\Phi([B_i,B_i])=g[B_i,B_i]g^{-1}\), and
\[
 \Phi(\normal{[B_i,B_i]}_{G_i})
 =\normal{g[B_i,B_i]g^{-1}}_{G_i}
 =\normal{[B_i,B_i]}_{G_i}.
\]
The last equality holds because a subgroup
and any of its conjugates have the same
normal closure. Since $B_i=K_i\times\langle h_i\rangle$,
its commutator subgroup is
\[
[B_i,B_i]=[K_i,K_i].
\]
\cref{thm:one-step} then gives
\eqref{eq:terminal-kernel}.  Since \(r_i\) is surjective, it induces
\eqref{eq:marked-collapse}. 
Since $r_i$ is the identity on $G_{i-1}$,
it fixes $\mu$. The same theorem also gives
\(
r_i\circ\rho_i=\rho_{i-1}.
\)
\end{proof}

\begin{corollary}
\label{cor:descent-through-quotient}
Every meridian-fixing \(\Phi_i\in\Aut(G_i)\) descends uniquely to
\(\Phi_{i-1}\in\Aut(G_{i-1})\) satisfying
\[
 \Phi_{i-1}\circ r_i=r_i\circ\Phi_i,\qquad
 \Phi_{i-1}(\mu)=\mu.
\]
If \(\Phi_i\circ\rho_i=\rho_i\circ f_*\) for a
representative \(f_*\) of \(f\in\Mod(\Sigma_g)\), then
\[
 \Phi_{i-1}\circ\rho_{i-1}=\rho_{i-1}\circ f_*
\]
with the same \(f_*\).
\end{corollary}

\begin{proof}
Define the descended map by
\[
 \Phi_{i-1}(r_i(x))=r_i(\Phi_i(x))\qquad(x\in G_i).
\]
If \(r_i(x)=r_i(y)\), then \(xy^{-1}\in\cN_i\).
Invariance of \(\cN_i\) gives
\(r_i(\Phi_i(x))=r_i(\Phi_i(y))\), so the map is well defined.
It is a homomorphism, and applying the same construction to
\(\Phi_i^{-1}\) gives its inverse.  Surjectivity of \(r_i\) gives
uniqueness, and 
\(\Phi_{i-1}(\mu)=\mu\) by construction.
Finally,
\[
 \Phi_{i-1}\circ\rho_{i-1}
 =\Phi_{i-1}\circ r_i\circ\rho_i
 =r_i\circ\Phi_i\circ\rho_i
 =r_i\circ\rho_i\circ f_*
 =\rho_{i-1}\circ f_*.
\]
\end{proof}

\section{Labelled filling systems}\label{sec:filling}

We now construct the curves needed for the
successive rim surgeries. Starting with the
nonseparating curve $d_1$ chosen in
\cref{sec:initialcarriers}, we construct an ordered
family
\[
\mathcal D=(d_1,\ldots,d_N)
\]
that fills $\Sigma_g$ and satisfies
$i(d_i,d_{i+1})=1$, as required by
\cref{cor:carrier-induction}. We also require
that the only mapping class preserving every
$d_i$ as an unoriented curve be the identity.

We first construct a filling chain whose
labelled stabilizer is finite. We then add
curves one at a time to eliminate the remaining
symmetries, while keeping consecutive
intersection number one. We adapt the Alexander method
\cite[Proposition~2.8 and Lemma~2.9]{FarbMargalit}.
Finite rigid collections of nonseparating curves
are also known \cite[Theorem~1.1]{DePool}.

Throughout this section, $g\ge3$, curves are
considered up to unoriented isotopy, and
$i(c,d)$ denotes geometric intersection number.

\begin{definition}[Ordered filling systems]
\label{def:filling-system}
A finite collection of essential simple closed
curves \emph{fills} $\Sigma_g$ if it has representatives
in pairwise minimal position whose complement
is a union of disks. An \emph{ordered filling system} is a tuple
\[
\cD=(d_1,\ldots,d_N)
\]
of pairwise distinct unoriented isotopy classes
of essential simple closed curves that fills
$\Sigma_g$. Its stabilizer is
\[
\Stab_{\Mod(\Sigma_g)}(\cD)
=
\{\varphi\in\Mod(\Sigma_g):
\varphi(d_i)=d_i\text{ for every }i\}.
\]
The order of the curves specifies the order
of the rim surgeries.
\end{definition}

\begin{lemma}[Filling chain]
\label{lem:filling-chain}
For any nonseparating curve $d_1$ on $\Sigma_g$,
there is an ordered filling system
\[
\cD=(d_1,\ldots,d_{2g})
\]
consisting of nonseparating curves such that
\[
i(d_j,d_{j+1})=1\qquad(1\le j<2g),
\]
and
\[
i(d_j,d_k)=0\qquad(|j-k|>1).
\]
\end{lemma}

\begin{proof}
Take annuli $A_1,\ldots,A_n$ and plumb
$A_j$ to $A_{j+1}$ along a square for
$1\le j<n$, choosing disjoint plumbing
squares in each annulus. Let $N_n$ be the
resulting compact oriented surface, and
let $a_j$ be the core of $A_j$. The curves $a_j$ and $a_{j+1}$ intersect
once, while all other pairs are disjoint.
Since $N_n$ is obtained from $n$ annuli
by $n-1$ square plumbings,
\[
\chi(N_n)=1-n.
\]
Moreover, $N_n$ is a regular neighborhood
of the union of the core curves, so
$[a_1],\ldots,[a_n]$ form a basis of
$H_1(N_n;\mathbb Z)$. The plumbing construction is illustrated in
\cref{fig:filling-chain-plumbing}.

\begin{figure}[!htbp]
\centering
\includegraphics[width=85mm]{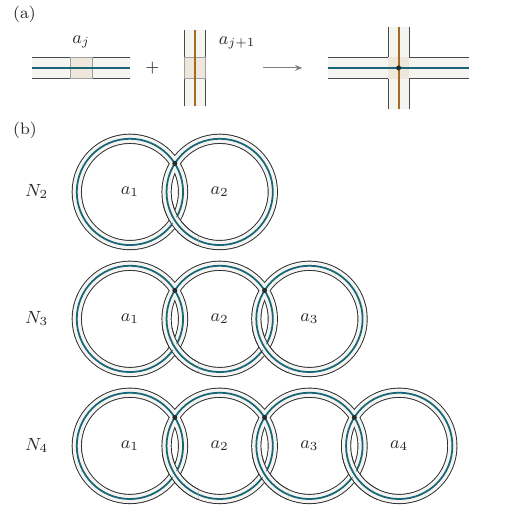}
\caption{%
(a) Square plumbing of two annuli.
(b) The resulting surfaces $N_2$, $N_3$,
and $N_4$, with respective genus and
boundary-component counts
$(1,1)$, $(1,2)$, and $(2,1)$.
Adjacent cores intersect once at the
marked points, while nonadjacent cores
are disjoint.}
\label{fig:filling-chain-plumbing}
\end{figure}

Orient the cores so that the algebraic intersection
number of $a_j$ and $a_{j+1}$ is $1$.
In the basis $[a_1],\ldots,[a_n]$ of
$H_1(N_n;\mathbb Z)$, the intersection form
has matrix
\[
M_n=
\begin{pmatrix}
0 & 1 & 0 & \cdots & 0\\
-1 & 0 & 1 & \ddots & \vdots\\
0 & -1 & 0 & \ddots & 0\\
\vdots & \ddots & \ddots & \ddots & 1\\
0 & \cdots & 0 & -1 & 0
\end{pmatrix}.
\]
Write $D_n=\det M_n$, with $D_0=1$.
Expanding along the first row, and then along
the first column of the remaining matrix,
gives $D_n=D_{n-2}$. Since $D_1=0$, we obtain
\[
D_{2k}=1,\qquad D_{2k+1}=0.
\]
Thus $M_n$ has rank $n$ when $n$ is even.
When $n$ is odd, its determinant vanishes,
but its upper-left $(n-1)\times(n-1)$ submatrix
is invertible, so its rank is $n-1$.

Recall that the radical of the intersection
form consists of classes having zero algebraic
intersection with every class. For a connected
oriented surface with $b$ boundary components,
the radical is generated by the boundary classes,
whose only relation is that their sum vanishes.
Indeed, capping the boundary components with
disks gives a closed surface with nondegenerate
intersection form, and the kernel of the induced
map on first homology is precisely the span of
the boundary classes. Hence the radical has
dimension $b-1$.

Applying this to $N_n$ over $\mathbb Q$, we have
\[
b-1=n-\operatorname{rank}(M_n).
\]
Therefore
\[
\#\partial N_n=
\begin{cases}
2,&n\text{ odd},\\
1,&n\text{ even}.
\end{cases}
\]
Taking $n=2g$, we see that $N_{2g}$ has
one boundary component. Since
$\chi(N_{2g})=1-2g$, it has genus $g$.

Cap its boundary with a disk to obtain
a closed surface $\Sigma_g$. Since $N_{2g}$
is a regular neighborhood of the union
of the core curves, their complement
in $\Sigma_g$ is a disk. Capping the
boundary does not change first homology,
so $[a_1],\ldots,[a_{2g}]$ remain a basis
of $H_1(\Sigma_g;\mathbb Z)$. Thus the
curves are pairwise distinct. 
Each $a_j$ is nonseparating, since a separating
curve cannot intersect another closed curve
exactly once. Their intersections
are also pairwise minimal, so they form
an ordered filling system.

Finally, cutting $\Sigma_g$ along either
$a_1$ or the prescribed curve $d_1$
gives a surface of genus $g-1$ with
two boundary components. By the
classification of surfaces, there is
an orientation-preserving homeomorphism
of $\Sigma_g$ carrying $a_1$ to $d_1$.
Applying this homeomorphism to the entire
chain gives the required ordered filling
system. The genus-three case is illustrated
in \cref{fig:filling-chain-genus-three}.

\begin{figure}[!htbp]
\centering
\includegraphics[width=85mm]{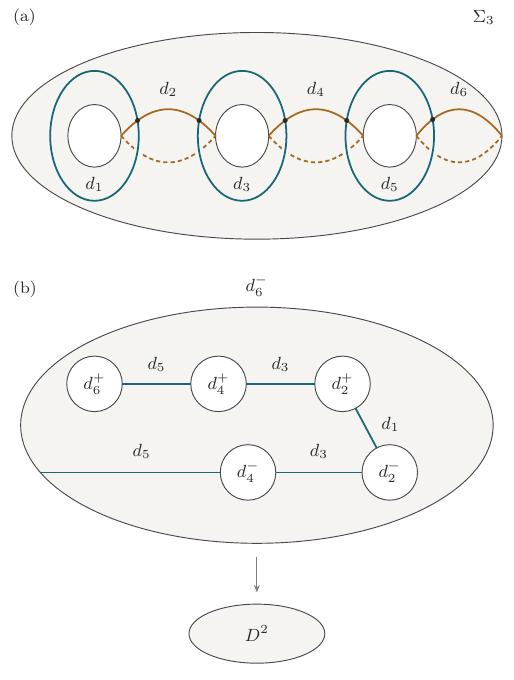}
\caption{%
An ordered filling system
$(d_1,\ldots,d_6)$ on $\Sigma_3$.
(a) The surface is obtained by doubling
a disk with three holes. The five dots mark the intersections of adjacent curves.
(b) Cutting along $d_2,d_4,d_6$ gives
a sphere with six boundary components.
The remaining five arcs connect these
components in a tree, so cutting along
them leaves a disk.}
\label{fig:filling-chain-genus-three}
\end{figure}
\end{proof}

\begin{lemma}
\label{lem:filling-finite-stabilizer}
Let $\cD=(d_1,\ldots,d_{2g})$ be an
ordered filling system as in
\cref{lem:filling-chain}. Then
$\Stab_{\Mod(\Sigma_g)}(\cD)$ is finite.
\end{lemma}

\begin{proof}
Choose representatives of the filling chain in
minimal position, and let $f$ represent an element
of $\Stab_{\Mod(\Sigma_g)}(\cD)$. We show
inductively that $f$ can be isotoped to preserve
every $d_j$ setwise.

Suppose $f$ already preserves
$d_1,\ldots,d_{j-1}$. Since $f(d_j)$ is
isotopic to $d_j$, any intersections between
them can be removed by bigon moves.
Both curves are disjoint from
$d_1,\ldots,d_{j-2}$ and meet $d_{j-1}$
exactly once. After arranging that their
intersections avoid the earlier curves,
consider a bigon between them.

The bigon is disjoint from
$d_1,\ldots,d_{j-2}$, since these curves
are essential and do not meet its boundary.
The curve $d_{j-1}$ either misses the bigon
or crosses it along a single arc.
We remove the bigon by sliding $f(d_j)$
across it. If $d_{j-1}$ crosses the bigon,
we move its intersection with $f(d_j)$
along that arc. In either case, the isotopy
preserves all earlier curves setwise.

After removing all bigons, $f(d_j)$ and
$d_j$ are disjoint and isotopic, so they
bound an annulus $A$. The curves $d_1,\ldots,d_{j-2}$ are disjoint
from $A$. Otherwise, one of them would
lie inside $A$ and hence be isotopic to
$d_j$, contrary to the distinctness of
the chain curves.

Since $d_{j-1}$ intersects each boundary
component of $A$ exactly once, it crosses
$A$ along a single arc joining them.
We can therefore slide $f(d_j)$ across
$A$ until it agrees with $d_j$, moving
its intersection with $d_{j-1}$ along
this arc. This isotopy preserves all
earlier curves setwise. For $j=1$, there are no earlier curves
to preserve. Repeating the argument for
$j=1,\ldots,2g$, we obtain a representative
of $f$ that preserves every curve in
$\cD$ setwise. These isotopies are
illustrated in
\cref{fig:filling-chain-isotopies}.

\begin{figure}[!htbp]
\centering
\includegraphics[width=85mm]{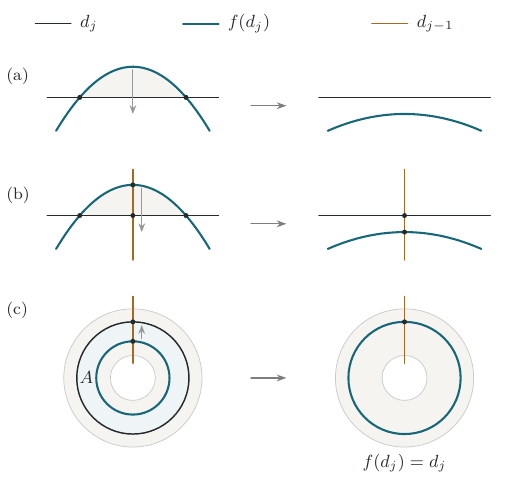}
\caption{%
Isotopies used to preserve the filling
chain. (a) Removing a bigon between
$f(d_j)$ and $d_j$.
(b) When $d_{j-1}$ crosses the bigon,
the intersection with $f(d_j)$ moves
along $d_{j-1}$.
(c) Sliding $f(d_j)$ across the annulus
$A$ until it agrees with $d_j$.}
\label{fig:filling-chain-isotopies}
\end{figure}

Regard
\(
\Gamma=\bigcup_{j=1}^{2g}d_j
\)
as a graph whose vertices are the
intersection points of the curves and whose
edges are the arcs between them. Since the
chain fills $\Sigma_g$, cutting along $\Gamma$
gives a disk. After the preceding isotopies, $f$ preserves
$\Gamma$ and therefore permutes its finitely
many vertices and edges. We show that this
action determines the mapping class of $f$.

Suppose two homeomorphisms $f_1$ and $f_2$
induce the same action on the vertices and
oriented edges of $\Gamma$.
Then $h=f_2^{-1}f_1$ fixes every vertex
and preserves each edge with its orientation. The restriction of $h$ to each edge is an
orientation-preserving homeomorphism fixing
its endpoints, so it can be isotoped to
the identity. These isotopies extend to
the surface, allowing us to assume that
$h$ fixes $\Gamma$ pointwise. 
Since $h$ fixes $\Gamma$ pointwise and
preserves the orientation of $\Sigma_g$,
it cannot exchange the two sides of any
edge. Thus, after cutting along $\Gamma$,
it induces a homeomorphism of the disk
that fixes its boundary pointwise. By the Alexander trick, this homeomorphism
is isotopic to the identity relative to the
boundary. Regluing the disk gives an isotopy
to the identity on $\Sigma_g$.

Thus distinct mapping classes in the
stabilizer induce distinct automorphisms
of the finite graph $\Gamma$. Since this
graph has only finitely many automorphisms,
the stabilizer is finite.
\end{proof}

\begin{convention}[Boundary isotopies and spanning arcs]
Let $S=\Sigma_{h,2}$ have boundary components
$B_+$ and $B_-$. We consider orientation-preserving
homeomorphisms of $S$, which may exchange the two
boundary components. Isotopies are not required
to fix the boundary pointwise.

A \emph{spanning arc} is a properly embedded arc
joining $B_+$ to $B_-$. We consider these arcs
up to unoriented isotopy, allowing their endpoints
to move along the boundary. Let $\mathcal X(S)$
denote the set of their isotopy classes.

Under this convention, a Dehn twist about a
boundary-parallel curve is isotopic to the identity.
Indeed, in a collar
$(\R/\Z)\times[0,1]$, with boundary at $r=0$,
the maps
\[
H_t(x,r)=(x+t(1-r),r),
\qquad 0\le t\le1,
\]
extended by the identity outside the collar,
give such an isotopy. The boundary circle
rotates during the isotopy.
\end{convention}

\begin{lemma}[Spanning arcs detecting mapping classes]
\label{lem:rigid-cross-arcs}
\label{lem:cross-arc-faithful}
Let $S=\Sigma_{h,2}$ with $h\ge2$.
There exist $2h+1$ pairwise disjoint,
pairwise nonisotopic spanning arcs
\[
\alpha_0,\ldots,\alpha_{2h}
\]
such that cutting $S$ along them gives a disk.
Set
\[
\cA=\{[\alpha_0],\ldots,[\alpha_{2h}]\}
\subset\mathcal X(S).
\]
If an orientation-preserving mapping class of
$S$ fixes every element of $\cA$, then it is
the identity. Consequently, the action of
the boundary-setwise mapping class group
on $\mathcal X(S)$ is faithful.
\end{lemma}

\begin{proof}
Put $n=2h+1\ge5$. Start with two disjoint
disks $D_+$ and $D_-$, and join them by
$n$ bands $e_0,\ldots,e_{n-1}$.
Arrange the ends of the bands around
$\partial D_+$ in the cyclic order
\[
0,1,\ldots,n-1,
\]
and around $\partial D_-$ in the order
\[
0,-2,-4,\ldots,-2(n-1),
\]
with indices taken modulo $n$.
Since $n$ is odd, every label appears
exactly once in the second list.
Attach the bands so that the resulting
surface $N$ is oriented.

We determine the boundary of $N$ by
following it through the bands.
Starting at band $e_i$ on $D_+$,
we follow one side of the band to $D_-$,
continue to band $e_{i-2}$, and return
to $D_+$. The next band there is
$e_{i-1}$. Thus successive returns
to $D_+$ encounter the labels
\[
i,i-1,i-2,\ldots
\]
modulo $n$. This traversal visits every
band end on both disks, so $N$ has
exactly one boundary component.
Since $N$ consists of two disks joined
by $n$ bands,
\[
\chi(N)=2-n=1-2h.
\]
It follows that $N$ has genus $h$.

Cap the boundary of $N$ with a disk,
then remove a small open disk from
the interior of each $D_\pm$.
The resulting surface is
$S\cong\Sigma_{h,2}$, with boundary
components $B_+$ and $B_-$.
Each band gives a properly embedded
arc $\alpha_i$ joining these boundary
components. The arcs are disjoint. 
The construction for $h=2$ is shown in
\cref{fig:spanning-arcs-construction}.

\begin{figure}[!htbp]
\centering
\includegraphics[width=85mm]{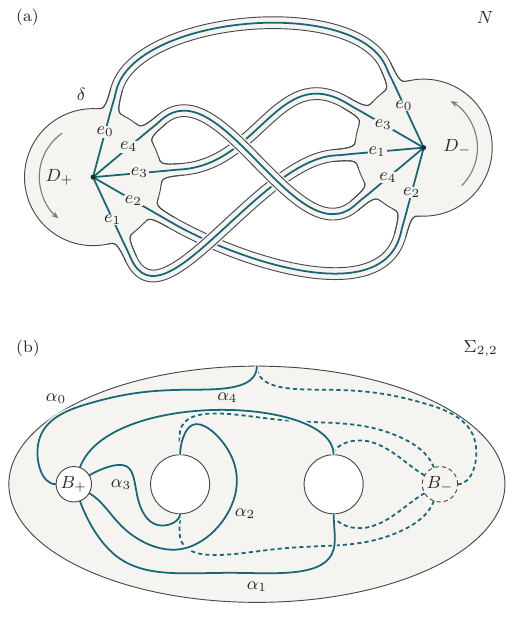}
\caption{%
Five spanning arcs for $h=2$.
(a) Five bands join $D_+$ and $D_-$ in the
cyclic orders $(0,1,2,3,4)$ and
$(0,3,1,4,2)$, producing a genus-two
surface $N$ with one boundary component.
(b) Capping this boundary and puncturing
the two disks gives $\Sigma_{2,2}$.}
\label{fig:spanning-arcs-construction}
\end{figure}

Cutting $S$ along all the $\alpha_i$ leaves
a connected surface, since every remaining
piece meets the disk used to cap $\partial N$.
Its Euler characteristic is
\[
\chi(S)+n=-2h+(2h+1)=1,
\]
so it is a disk. In addition, the arcs are pairwise nonisotopic.
If two were isotopic, they would bound
a rectangle whose other two sides lie in
$B_+$ and $B_-$. Cutting along these two
arcs would separate the rectangle from
the rest of $S$, contradicting the
connectedness. The cut-open disk and the cyclic orders of
the arc endpoints are illustrated in
\cref{fig:spanning-arcs-cutting}.

\begin{figure}[!htbp]
\centering
\includegraphics[width=85mm]{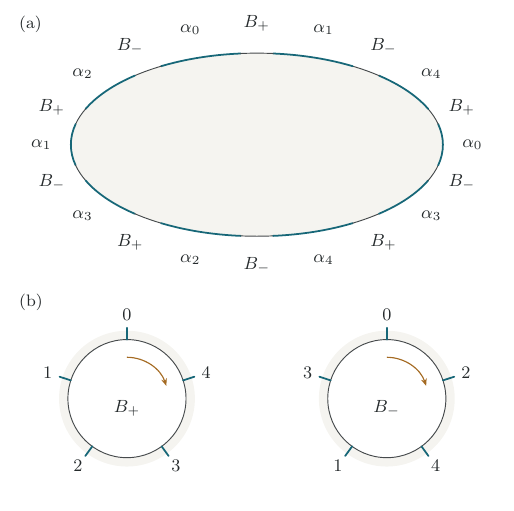}
\caption{%
The five spanning arcs on $\Sigma_{2,2}$.
(a) Cutting along all five arcs gives a disk.
Each arc appears twice on its boundary,
with intervening intervals from $B_+$
and $B_-$.
(b) The cyclic orders of the arc endpoints
in the induced boundary orientations.}
\label{fig:spanning-arcs-cutting}
\end{figure}

Let $f$ represent a mapping class fixing every
$[\alpha_i]$. As in the proof of
\cref{lem:filling-finite-stabilizer}, we isotope
$f$ one arc at a time. Suppose $f$ already preserves
$\alpha_0,\ldots,\alpha_{i-1}$.
Since $f(\alpha_i)$ is isotopic to $\alpha_i$,
we remove their intersections by bigon and
half-bigon moves. A half-bigon is a disk
bounded by portions of the two arcs and an
interval of $\partial S$. We remove it by
sliding an endpoint of $f(\alpha_i)$ along
that boundary interval. These moves preserve the earlier arcs.
Indeed, those arcs are disjoint from both
$f(\alpha_i)$ and $\alpha_i$, and each joins
$B_+$ to $B_-$. They therefore cannot enter
an usual bigon or a half-bigon whose
boundary meets only one of $B_+$ and $B_-$. Once $f(\alpha_i)$ and $\alpha_i$ are disjoint,
they bound a rectangle whose other two sides
lie in $\partial S$. No earlier arc can lie
in this rectangle, since it would then be isotopic
to $\alpha_i$. Sliding across the rectangle
makes $f(\alpha_i)=\alpha_i$ without moving
the earlier arcs. By continuing inductively, we isotope $f$ to
preserve every $\alpha_i$ setwise.

Now suppose $f$ exchanges $B_+$ and $B_-$.
Since $f$ preserves every $\alpha_i$ setwise,
it must send the endpoint of $\alpha_i$ on
$B_+$ to its endpoint on $B_-$. As $f$
preserves the orientation of $S$, it must
also preserve the cyclic order of these
endpoints. However, in the induced boundary orientations,
the next label after $i$ is $i-1$ on $B_+$
and $i+2$ on $B_-$, as shown in
\cref{fig:spanning-arcs-cutting}.
This would require
\[
i-1\equiv i+2\pmod n,
\]
so $n$ divides $3$, contrary to $n\ge5$.
Thus $f$ preserves each boundary component.

Since $f$ also preserves every $\alpha_i$
setwise, it fixes their endpoints. We can
isotope $f$ to fix the arcs pointwise, and
then fix the boundary intervals between
their endpoints. As in the proof of
\cref{lem:filling-finite-stabilizer},
orientation preservation ensures that
$f$ does not exchange the two sides of
any arc.

Cutting $S$ along the arcs gives a disk,
and $f$ induces a homeomorphism fixing its
boundary pointwise. By the Alexander trick,
this homeomorphism is isotopic to the
identity relative to the boundary.
Regluing gives an isotopy of $f$ to the
identity on $S$. Thus any mapping class fixing every
$[\alpha_i]$ is the identity, as required.
\end{proof}

\begin{lemma}[Curves and spanning arcs]
\label{lem:cross-arc-cutting}
Let $c$ be a nonseparating curve on $\Sigma_g$,
and let
\[
S=\Sigma_g\setminus\operatorname{int}\nu c
\cong\Sigma_{g-1,2},
\]
with boundary components $B_+$ and $B_-$.

Cutting a simple closed curve $e$ with
$i(c,e)=1$ gives a spanning arc in $S$.
This defines a well-defined surjection
\[
\kappa_c:
\{[e]:i(c,e)=1\}
\longrightarrow\mathcal X(S),
\qquad
[e]\longmapsto[e\cap S].
\]

Every mapping class preserving the unoriented
isotopy class of $c$ induces a well-defined
mapping class of $S$, and $\kappa_c$
respects these actions. In particular,
exchanging the two sides of $c$ exchanges
$B_+$ and $B_-$.
\end{lemma}

\begin{proof}
A simple closed curve $e$ with $i(c,e)=1$
can be arranged to meet $c$ once, so
cutting along $c$ gives a spanning arc.
To see that its isotopy class depends only
on $[e]$, let $e_0$ and $e_1$ be isotopic
curves, each meeting $c$ once. Arrange
their intersections with $c$ to be distinct.
As in the proof of
\cref{lem:filling-finite-stabilizer},
we remove the bigons between $e_0$ and
$e_1$ while preserving $c$ setwise.
Indeed, $c$ either misses a bigon or
crosses it in a single arc. In the latter
case, its intersection with $e_0$ moves
along $c$ during the bigon removal.
Once $e_0$ and $e_1$ are disjoint, they
bound an annulus crossed by $c$ in a
single arc. Sliding across this annulus
identifies the two curves while preserving
$c$ setwise. Since these isotopies preserve $c$ setwise,
$e_0$ and $e_1$ determine isotopic spanning
arcs in $S$. Thus $\kappa_c$ is well defined.

Conversely, let $\alpha$ be any spanning
arc in $S$. Join its endpoints by an
embedded arc in $\nu c$ that crosses $c$
once. Together they form a simple closed
curve meeting $c$ once, and hence an
essential curve. Thus $\kappa_c$ is
surjective.

Finally, let $f$ represent a mapping class
preserving the unoriented isotopy class
of $c$. 
After an isotopy, we may assume
$f(c)=c$ and that $f$ preserves
the neighborhood $\nu c$.
It therefore restricts to a homeomorphism
$\widehat f$ of $S$. If $e$ meets $c$ once, then
$f(e)$ also meets $c$ once, and
\(
f(e)\cap S=\widehat f(e\cap S).
\)
Consequently,
\(
\kappa_c(f([e]))
=\widehat f(\kappa_c([e])).
\)
This holds even when $f$ exchanges
the two sides of $c$, in which case
$\widehat f$ exchanges $B_+$ and $B_-$.

To see that restriction is well defined, suppose $f$ preserves $c$
and is isotopic to the identity on $\Sigma_g$.
Then $f$ fixes every curve class $[e]$
with $i(c,e)=1$. Since $\kappa_c$ is
surjective and commutes with the action of
$f$, the restriction $\widehat f$ fixes
every spanning-arc class. By
\cref{lem:rigid-cross-arcs}, $\widehat f$
is isotopic to the identity on $S$.
\end{proof}

\begin{lemma}
\label{lem:dual-curve-detection}
Let $g\ge3$ and let $c$ be a nonseparating
curve on $\Sigma_g$.
If a mapping class $\varphi\in\Mod(\Sigma_g)$
fixes $c$ and every simple closed curve $e$
with $i(c,e)=1$, then $\varphi$ is the identity.
\end{lemma}

\begin{proof}
Choose a representative $f$ preserving $c$
and its neighborhood $\nu c$, and let
$\widehat f$ be its restriction to
$S=\Sigma_{g-1,2}$. Since $f$ fixes every
curve class meeting $c$ once,
\cref{lem:cross-arc-cutting} implies that
$\widehat f$ fixes every spanning-arc class.
By \cref{lem:rigid-cross-arcs},
$\widehat f$ is isotopic to the identity
under our boundary-setwise convention.
In particular, $f$ does not exchange
the two sides of $c$.

We now show that $f$ is isotopic to a
power of the Dehn twist $T_c$ about $c$.
 Choose an isotopy $H_t$
from $\widehat f$ to the identity on $S$,
preserving each boundary component setwise. Identify both boundary circles with
$\R/\Z$, and choose lifts
$\widetilde H_t^\pm:\R\to\R$ of their
restrictions under $H_t$, with the same
initial lift. Since $H_1$ is the identity,
\[
\widetilde H_1^\pm(x)=x+m_\pm,
\qquad m_\pm\in\Z.
\]
Define
\[
n=m_+-m_-,
\qquad
q_t=\widetilde H_t^+
\circ(\widetilde H_t^-)^{-1}.
\]
Then $q_0(x)=x$, $q_1(x)=x+n$, and
$q_t\widetilde H_t^-=\widetilde H_t^+$.
On a collar of $B_-$, with boundary
at $r=0$, define
\[
C_t(x,r)=
\bigl(rx+(1-r)q_t(x),r\bigr)
\pmod{\Z},
\]
and extend $C_t$ by the identity
outside the collar. Since $q_t$ is
increasing and commutes with integer
translations, $C_t$ is a homeomorphism.
Moreover, $C_0=\operatorname{id}$,
and the restrictions of $C_tH_t$ to
$B_+$ and $B_-$ agree. 
The maps $C_t\circ H_t$ therefore give
an isotopy of $\Sigma_g$ from $f$ to
$C_1$, where
\[
C_1(x,r)=(x-nr,r)\pmod{\Z}
\]
on the collar and $C_1$ is the identity
elsewhere. Thus $f$ is isotopic to a
power of $T_c$.

Finally, write $\varphi=T_c^k$ and choose
a curve $e$ with $i(c,e)=1$. Since
$c$ and $e$ intersect once, the Dehn
twist formula on first homology gives
\[
[T_c^k(e)]=[e]\pm k[c],
\]
where the sign depends on the orientations.
Thus the algebraic intersection number
of $T_c^k(e)$ and $e$ is $\pm k$.
But $\varphi$ fixes the isotopy class
of $e$, hence $k=0$ and
$\varphi=1$.
\end{proof}

\begin{theorem}[An ordered filling system with trivial stabilizer]
\label{thm:filling-sequence}
The nonseparating curve $d_1$ constructed in
\cref{sec:initialcarriers} extends to an
ordered filling system
\[
\cD=(d_1,\ldots,d_N)
\]
of nonseparating curves on $\Sigma_g$ such that
\begin{equation}\label{eq:consecutive-one}
i(d_i,d_{i+1})=1
\qquad(1\le i<N),
\end{equation}
and
\begin{equation}\label{eq:trivial-labelled-stabilizer}
\Stab_{\Mod(\Sigma_g)}(\cD)=1.
\end{equation}
\end{theorem}

\begin{proof}
By \cref{lem:filling-chain}, extend $d_1$
to a filling chain. Its stabilizer
$\mathcal H_0$ is finite by
\cref{lem:filling-finite-stabilizer}. Suppose the current stabilizer
$\mathcal H_k$ is nontrivial.
Choose $1\ne\varphi\in\mathcal H_k$
and let $c$ be the last curve.
By \cref{lem:dual-curve-detection},
there is a curve $e$ with
$i(c,e)=1$ and $\varphi(e)\ne e$.
Since $\varphi$ fixes every curve
already in the tuple, $e$ is distinct
from all of them. It is nonseparating
because it meets $c$ once.
Appending $e$ preserves
\eqref{eq:consecutive-one} and gives
the proper subgroup
\[
\mathcal H_{k+1}
=\{\psi\in\mathcal H_k:\psi(e)=e\}.
\]
Since $\mathcal H_0$ is finite,
repeating this step eventually gives
a trivial stabilizer.

The initial chain already fills $\Sigma_g$,
and adding curves only subdivides its
complementary disks. Thus the final tuple
also fills $\Sigma_g$.
\end{proof}

\section{Recovering a surgery curve from the exterior group}
\label{sec:slope-recovery}

By \cref{thm:tower-spectrum}, every
meridian-fixing automorphism $\Phi_i$ of
$G_i$ preserves the conjugacy class of
$B_i$. By
\cref{prop:characteristic-kernel,cor:descent-through-quotient},
$\Phi_i$ induces an automorphism
$\Phi_{i-1}$ of $G_{i-1}$. If
\[
\Phi_i\circ\rho_i=\rho_i\circ f_*
\]
for some $f\in\Mod(\Sigma_g)$, then
\[
\Phi_{i-1}\circ\rho_{i-1}
=\rho_{i-1}\circ f_*.
\]
Before passing to $\Phi_{i-1}$, we now
show that $f$ fixes the surgery curve
$d_i$. Repeating this step will show
that $f$ fixes every curve in the
ordered filling system of
\cref{thm:filling-sequence}, and hence
$f$ is the identity.

The key point is the cyclic-subgroup conjugacy
condition: injectivity of the push-off alone does
not turn conjugacy in the exterior group into
conjugacy in the surface group. For comparison,
local knot decorations have also been used to confine symmetries of spatial graphs
\cite[\S2]{FlapanMellorNaimi}. That construction
does not supply the surface-subgroup conjugacy we need here.

Fix $1\le i\le N$, and write
\[
\pi=\pi_1(\Sigma_g,p),
\qquad
h_i=\rho_{i-1}(d_i),
\qquad
S_{i-1}=\rho_{i-1}(\pi).
\]
As before, $d_i$ denotes both an unoriented
curve and an element of $\pi$ determined
by an orientation and a basing path.
The element $h_i$ is its push-off before
the $i$-th surgery. The product vertex group
\[
B_i=K_i\times\langle h_i\rangle
\]
has center $Z(B_i)=\langle h_i\rangle$.
By \cref{lem:creation-stage-edge},
the cyclic subgroups
$\rho_i(\langle d_i\rangle)$ and
$\langle h_i\rangle$ are conjugate in
$G_i$, although $\rho_i(d_i)$ need
not equal $h_i$. We will use the
retraction $r_i:G_i\to G_{i-1}$,
the cyclic-subgroup conjugacy condition
of \cref{cor:carrier-induction}, and
injectivity of $\rho_{i-1}$ to recover
the unoriented curve $d_i$.

\begin{proposition}[Recovery of the $i$-th surgery curve]
\label{prop:terminal-slope-recovery}
Let $\Phi_i\in\Aut(G_i)$ satisfy
\[
\Phi_i(\mu)=\mu,
\qquad
\Phi_i\circ\rho_i=\rho_i\circ f_*
\]
for some $f\in\Mod(\Sigma_g)$.
If $\Phi_i$ preserves the conjugacy class
of $B_i$, then
\[
f(d_i)=d_i
\]
as an unoriented isotopy class.
\end{proposition}

\begin{proof}
Since $\Phi_i$ preserves the conjugacy class
of $B_i$, choose $g\in G_i$ such that
\[
\Phi_i(B_i)=gB_ig^{-1}.
\]
Automorphisms preserve centers, and
$Z(B_i)=\langle h_i\rangle$. Hence
\begin{equation}
\label{eq:center-under-Phi}
\Phi_i(\langle h_i\rangle)
=g\langle h_i\rangle g^{-1}.
\end{equation}
By \cref{lem:creation-stage-edge}, there
is $a_i\in G_i$ such that
\[
\rho_i(\langle d_i\rangle)
=a_i\langle h_i\rangle a_i^{-1}.
\]
Applying $\Phi_i$ and using
$\Phi_i\circ\rho_i=\rho_i\circ f_*$,
we obtain
\[
\rho_i(\langle f_*(d_i)\rangle)
=\Phi_i(a_i)g\langle h_i\rangle
g^{-1}\Phi_i(a_i)^{-1}.
\]

Now apply the retraction
$r_i:G_i\to G_{i-1}$.
Since $r_i\circ\rho_i=\rho_{i-1}$
and $r_i(h_i)=h_i$, putting
$q=r_i(\Phi_i(a_i)g)$ gives
\begin{equation}
\label{eq:retracted-slope}
\rho_{i-1}(\langle f_*(d_i)\rangle)
=q\langle h_i\rangle q^{-1}.
\end{equation}
The left side lies in $S_{i-1}$.
By the cyclic-subgroup conjugacy condition
of \cref{cor:carrier-induction}, there
is $s\in S_{i-1}$ such that
\[
q\langle h_i\rangle q^{-1}
=s\langle h_i\rangle s^{-1}.
\]

Choose $u\in\pi$ with
$\rho_{i-1}(u)=s$.
Since $h_i=\rho_{i-1}(d_i)$
and $\rho_{i-1}$ is injective,
\eqref{eq:retracted-slope} gives
\[
\langle f_*(d_i)\rangle
=u\langle d_i\rangle u^{-1}.
\]
Therefore
\[
f_*(d_i)=ud_i^{\pm1}u^{-1}.
\]
Thus $f(d_i)$ and $d_i$ are freely
homotopic as unoriented curves, and
hence isotopic.
\end{proof}

\section{Construction and rigidity of the surfaces}
\label{sec:mainproof}

We now bring together the geometric construction
and the algebraic results of the preceding sections.
The preliminary and vertical-boundary surgeries
produce a surface $F_0$ with injective push-off
homomorphism, while \cref{thm:filling-sequence}
provides an ordered filling system
$\cD=(d_1,\ldots,d_N)$ with trivial stabilizer.
Successive rim surgeries along these curves
give our final surface $F_N$.

To prove rigidity, we reverse the construction
one stage at a time. We first recognize the product
vertex group $B_i$, then recover the surgery
curve $d_i$, and finally descend to $G_{i-1}$,
preserving the same action on the surface group.
Repeating this argument shows that every
extendable mapping class fixes all the curves
in $\cD$, and hence is the identity.

\begin{theorem}[Construction by successive rim surgeries]
\label{thm:construction}
Fix $g\ge3$, and let $\epsilon_g=1$ if $g$
is even and $\epsilon_g=2$ if $g$ is odd.
Choose an ordered filling system
\[
\cD=(d_1,\ldots,d_N)
\]
as in \cref{thm:filling-sequence}, beginning
with the curve $d_1$ constructed in
\cref{sec:initialcarriers}. Thus the curves
are nonseparating, satisfy
\[
i(d_i,d_{i+1})=1\qquad(1\le i<N),
\]
and have trivial stabilizer:
\[
\Stab_{\Mod(\Sigma_g)}(\cD)=1.
\]

Choose hyperbolic knots
$J_1^D,\ldots,J_g^D$ for the preliminary
meridian-disk surgeries,
$J_1^V,\ldots,J_{\epsilon_g}^V$ for the
vertical-boundary surgeries, and
$J_1,\ldots,J_N$ for the subsequent
ordered surgeries. Assume that the
meridian-marked knot groups of the
$J_\nu^V$ and $J_i$ are pairwise
nonisomorphic (No such condition is needed on the preliminary knots $J_j^D$).

The preliminary and vertical-boundary
surgeries produce $F_0\subset S^4$.
Successive ordinary untwisted rim surgeries
along $d_1,\ldots,d_N$, using $J_i$
at stage $i$, produce surfaces
$F_1,\ldots,F_N$.

Write
\[
\pi=\pi_1(\Sigma_g,p),\qquad
G_i=\pi_1(S^4\setminus\operatorname{int}\nu F_i,*),
\qquad
\rho_i:\pi\longrightarrow G_i
\]
for the exterior groups and push-off
homomorphisms, and let $\mu$ be the
common positive meridian. For
$1\le i\le N$, put
\[
K_i=\pi_1(E(J_i)),\qquad
h_i=\rho_{i-1}(d_i),
\]
\[
A_i=\langle\mu,h_i\rangle,\qquad
B_i=K_i\times\langle h_i\rangle.
\]
Then the following hold.

\begin{enumerate}[label=\textup{(\alph*)},leftmargin=2.4em]

\item Each push-off homomorphism
$\rho_i$ is injective, for
$0\le i\le N$.

\item For every $1\le i\le N$,
\[
A_i\cong\mathbb Z^2,\qquad
G_i=G_{i-1}*_{A_i}B_i,\qquad
Z(B_i)=\langle h_i\rangle.
\]
The canonical retraction
$r_i:G_i\to G_{i-1}$ satisfies
\[
r_i\circ\rho_i=\rho_{i-1}.
\]

\item Every meridian-fixing automorphism
of $G_i$ preserves the conjugacy class
of $B_i$ and the normal subgroup
\[
\ker r_i
=\normal{[B_i,B_i]}_{G_i}.
\]
Consequently, every such automorphism
descends uniquely to a meridian-fixing
automorphism of $G_{i-1}$.

\end{enumerate}
\end{theorem}

\begin{proof}
The constructions of \cref{sec:disk,sec:seed}
give $F_0$, and \cref{thm:filling-sequence}
supplies $\cD$. By \cref{cor:carrier-induction},
the push-off homomorphisms remain injective
and $A_i\cong\mathbb Z^2$ at every stage.
\Cref{thm:one-step} gives the amalgam
decompositions and canonical retractions,
while \cref{lem:carrier-primitive} gives
$Z(B_i)=\langle h_i\rangle$.
This proves \textup{(a)} and \textup{(b)}.

By the pairwise nonisomorphism hypothesis,
\cref{thm:tower-spectrum} shows that every
meridian-fixing automorphism of $G_i$
preserves the conjugacy class of $B_i$.
The remaining assertions in \textup{(c)}
follow from
\cref{prop:characteristic-kernel,cor:descent-through-quotient}.
\end{proof}

Let $F_{\cD}=F_N$. We now show that the
final exterior group, together with its
meridian and push-off homomorphism, detects
every curve in $\cD$.

\begin{theorem}[Rigidity from the final push-off homomorphism]
\label{thm:peripheral-recognition}
Let $f\in\Mod(\Sigma_g)$, and suppose
$\Phi_N\in\Aut(G_N)$ satisfies
\[
\Phi_N(\mu)=\mu,
\qquad
\Phi_N\circ\rho_N=\rho_N\circ f_*.
\]
Then
\[
f(d_i)=d_i
\qquad(1\le i\le N)
\]
as unoriented isotopy classes. Consequently,
$f=1$.
\end{theorem}

\begin{proof}

We proceed by downward induction,
starting with the given $\Phi_N$.
At stage $i$, we have an automorphism
$\Phi_i\in\Aut(G_i)$ satisfying
\begin{equation}
\label{eq:reverse-peripheral-equation}
\Phi_i(\mu)=\mu,
\qquad
\Phi_i\circ\rho_i=\rho_i\circ f_*.
\end{equation}

Suppose $\Phi_i$ satisfies
\eqref{eq:reverse-peripheral-equation}.
By \cref{thm:construction}\textup{(c)},
it preserves the conjugacy class of $B_i$.
Hence \cref{prop:terminal-slope-recovery}
gives $f(d_i)=d_i$.

If $i>1$, \cref{cor:descent-through-quotient}
gives $\Phi_{i-1}\in\Aut(G_{i-1})$
satisfying
\begin{equation}
\label{eq:reverse-collapse-commutation}
\Phi_{i-1}\circ r_i=r_i\circ\Phi_i
\end{equation}
and \eqref{eq:reverse-peripheral-equation}
at stage $i-1$, with the same $f_*$. Thus $f$ fixes every $d_i$ individually.
By \eqref{eq:trivial-labelled-stabilizer},
\(
f\in\Stab_{\Mod(\Sigma_g)}(\cD)=1.
\)
\end{proof}

It remains to show that every extendable
mapping class gives an automorphism of
$G_N$ satisfying the hypotheses of
\cref{thm:peripheral-recognition}.

\begin{proposition}[Triviality of the extendable subgroups]
\label{prop:each-constructed-surface-rigid}
Every final surface $F=F_N$ constructed in
\cref{thm:construction} satisfies
\[
E_{\mathrm{TOP}}^+(F)=E^+(F)=1.
\]
\end{proposition}

\begin{proof}
Let $\Psi:(S^4,F)\to(S^4,F)$ be a
homeomorphism preserving the orientations
of both $S^4$ and $F$, and let
$f\in\Mod(\Sigma_g)$ be the induced
mapping class. By
\cref{prop:peripheral-naturality},
the induced outer automorphism of $G_N$
has a representative $\Phi$ satisfying
$\Phi(\mu)=\mu$ and
\[
\Phi(\rho_N(\gamma))
=q\,\rho_N(f_*(\gamma))
\mu^{\theta(\gamma)}q^{-1}
\qquad(\gamma\in\pi),
\]
where $q\in\langle\rho_N(\pi),\mu\rangle$
and $\theta:\pi\to\mathbb Z$. By \cref{prop:early-homology},
\[
H_1(G_N;\mathbb Z)
=\mathbb Z\langle[\mu]\rangle,
\qquad
\rho_N(\pi)\subset[G_N,G_N].
\]
Since $\mu$ centralizes $\rho_N(\pi)$,
\cref{cor:rooted-normalization} gives
$\Phi_N\in\Aut(G_N)$ with
\[
\Phi_N(\mu)=\mu,
\qquad
\Phi_N\circ\rho_N=\rho_N\circ f_*.
\]
By \cref{thm:peripheral-recognition},
$f=1$. Hence $E_{\mathrm{TOP}}^+(F)=1$,
and therefore $E^+(F)=1$.
\end{proof}

\section{Alexander modules, knot choices, and homology}
\label{sec:homology}

Having established rigidity, we now compute
the Alexander module and choose the 
knots needed for rim surgery to complete the proofs of the main results.
We then extend the construction to simply
connected ambient $4$-manifolds and conclude
with computations of the group and exterior homology.

\subsection{Alexander-module calculations}

Put $\Lambda=\Z[t,t^{-1}]$.
Let $F\subset X$ be a smoothly embedded
closed oriented surface in a closed
oriented smooth $4$-manifold, and write
\[
E_F=X\setminus\operatorname{int}\nu F.
\]
Suppose
\[
H_1(E_F;\Z)=\Z\langle[\mu]\rangle,
\]
where $\mu$ is the positive meridian.
This holds for every surface exterior
in our construction by
\cref{prop:early-homology}.

Let $\phi_F:\pi_1(E_F)\to\Z$ be the
unique homomorphism sending $\mu$ to $1$.
The subgroup $\ker\phi_F$ determines
an infinite cyclic cover
\[
\widetilde E_F\longrightarrow E_F.
\]
Let $T$ be the deck transformation
corresponding to $\mu$. The
\textit{first Alexander module} of $F$ is
\[
\cA_1(F)=H_1(\widetilde E_F;\Z),
\]
with the left $\Lambda$-module structure
defined by
\[
t\cdot x=T_*(x).
\]

We also use a group-theoretic description.
For a meridian-marked group $(G,\mu)$
with a homomorphism $\phi:G\to\Z$
satisfying $\phi(\mu)=1$, put
$N=\ker\phi$ and define
\[
\cA(G)=N/[N,N],
\qquad
t\cdot[x]=[\mu x\mu^{-1}].
\]
Since $N$ is normal in $G$, conjugation
by $\mu$ induces an automorphism of
$N/[N,N]$. If $g\in G$ also satisfies
$\phi(g)=1$, then $g=k\mu$ for some
$k\in N$. Conjugation by $k$ acts
trivially on $N/[N,N]$, so $g$ and
$\mu$ induce the same action.

For $G=\pi_1(E_F)$ and $\phi=\phi_F$,
we have
\[
\pi_1(\widetilde E_F)\cong N.
\]
The positive deck transformation
corresponds to conjugation by $\mu$.
Thus
\[
\cA_1(F)\cong\cA(\pi_1(E_F))
\]
as left $\Lambda$-modules.

For a classical knot $J$, write
$\cA(J)$ for its Alexander module,
presented by $tV-V^T$ for a
Seifert matrix $V$.

The preliminary meridian-disk surgeries and the
subsequent rim surgeries have different amalgam
decompositions, but both add the Alexander module
of the surgery knot as a direct summand. This is
a Mayer--Vietoris calculation in the meridional
cover, analogous to satellite-knot module
calculations \cite[Theorem~2]{LivingstonMelvin}.
In the rank-two case we also account for the
extra circle in the new product vertex group.

\begin{proposition}[Alexander modules under one surgery]
\label{prop:alexander-one-step}
Let $(G,\mu)$ be a meridian-marked group
with a homomorphism
$\phi_G:G\to\Z$ satisfying
$\phi_G(\mu)=1$. Let $K$ be the group
of a classical knot $J$, with meridian $m$.
In both cases below, the amalgamation
identifies $m$ with $\mu$, and the
homomorphisms to $\Z$ agree with
$\phi_G$ on $G$ and with the
meridional abelianization on $K$.

\begin{enumerate}[label=\textup{(\roman*)},leftmargin=2.3em]
\item Suppose
\[
G^D=G*_{\langle\mu\rangle}K.
\]
Then
\[
\cA(G^D)\cong\cA(G)\oplus\cA(K).
\]

\item Suppose $h\in G$ satisfies
$\phi_G(h)=0$ and
$\langle\mu,h\rangle\cong\Z^2$.
Form the amalgam
\[
G'=G*_{\langle\mu,h\rangle}
\bigl(K\times\langle z\rangle\bigr),
\]
where $z$ generates an infinite cyclic
factor and is identified with $h$.
Then
\[
\cA(G')\cong\cA(G)\oplus\cA(K).
\]
\end{enumerate}
\end{proposition}

\begin{proof}
Choose classifying spaces for the vertex and edge groups and join them
by mapping cylinders for the injective edge maps.  The universal cover
is a tree of contractible spaces, so the resulting graph of spaces is a
classifying space for the amalgam.  The meridional homomorphism is
surjective on every vertex and edge group; hence all corresponding
meridional covers are connected.  The lifted graph of spaces has the
same one-edge underlying graph and computes $H_1(\ker\phi;\Z)$.  Put
\[
Q=\Lambda/(t-1).
\]
Their zeroth homology is $Q$, and in both cases the degree-zero
difference map is
\[
Q\longrightarrow Q\oplus Q,
\qquad q\longmapsto(q,-q),
\]
which is injective.

For~\textup{(i)}, the restriction to the edge group
$\langle\mu\rangle$ is an isomorphism onto $\Z$.  The lifted edge
space is therefore its universal cover $\mathbb R$, so its first
homology vanishes.  The Mayer--Vietoris segment
\[
0\longrightarrow\cA(G)\oplus\cA(K)
\longrightarrow\cA(G^D)
\longrightarrow Q\longrightarrow Q\oplus Q
\]
and the preceding injectivity give
$\cA(G^D)\cong\cA(G)\oplus\cA(K)$.

For~\textup{(ii)}, the lifted edge and product-vertex spaces satisfy
\[
\widetilde A=\mathbb R\times S^1_h,
\qquad
H_1(\widetilde A;\Z)=Q\langle[h]\rangle,
\]
and, by K\"unneth,
\[
H_1(\widetilde B;\Z)
\cong
\cA(K)\oplus Q\langle[z]\rangle.
\]
Thus the degree-one difference map is
\[
\delta_1:Q\langle[h]\rangle
\longrightarrow
\cA(G)\oplus\cA(K)\oplus Q\langle[z]\rangle,
\qquad
\delta_1(q[h])=(q[h],0,-q[z]).
\]
This is well defined because $[\mu,h]=1$ gives
$(t-1)[h]=0$ in $\cA(G)$.  Since the edge map identifies $h$ with
the positively oriented external circle $z$, the
$Q\langle[z]\rangle$-component is multiplication by $-1$ and hence an
isomorphism.  Thus $\delta_1$ is injective; the degree-zero injectivity
then identifies $\cA(G')$ with its cokernel.  The map
\[
\cA(G)\oplus\cA(K)\oplus Q
\longrightarrow
\cA(G)\oplus\cA(K),
\qquad
(a,k,q)\longmapsto(a+q[h],k),
\]
is surjective and has kernel
$\{(-q[h],0,q):q\in Q\}=\operatorname{im}\delta_1$.  It therefore
induces the required isomorphism.  The old class $[h]$ is absorbed by
the external-circle summand; it is not killed and imposes no quotient
on $\cA(G)$.
\end{proof}

Let
\[
\epsilon_g=
\begin{cases}
1,&g\text{ even},\\
2,&g\text{ odd}.
\end{cases}
\]

\begin{corollary}[Total Alexander module]
\label{cor:total-alexander}
For the final surface $F_N\subset S^4$,
\[
\cA_1(F_N)
\cong
\bigoplus_{j=1}^g\cA(J_j^D)
\oplus
\bigoplus_{\nu=1}^{\epsilon_g}\cA(J_\nu^V)
\oplus
\bigoplus_{i=1}^N\cA(J_i).
\]
In particular, $\cA_1(F_N)$ is a finitely generated torsion
$\Lambda$-module.
\end{corollary}

\begin{proof}
The exterior group of the standard unknotted surface is
$\langle\mu\rangle\cong\Z$, so its Alexander module is zero. By \cref{prop:alexander-one-step}\textup{(i)}, each of the
$g$ preliminary surgeries adds the Alexander module of
its knot as a direct summand. By
\cref{prop:alexander-one-step}\textup{(ii)}, the same is
true for the $\epsilon_g$ vertical-boundary surgeries
and the $N$ ordered surgeries. Applying these formulas
successively gives the displayed direct sum, using the
identification of $\cA_1(F_N)$ with the Alexander module
of its exterior group.

Each knot Alexander module has a finite square
presentation $tV-V^T$ with nonzero determinant, so it
is a finitely generated torsion $\Lambda$-module.
The same holds for their finite direct sum.
\end{proof}

\subsection{Knot choices and proofs of the main results}

To complete the construction, we need hyperbolic
knots with prescribed Alexander modules, as well
as infinitely many hyperbolic knots with trivial
Alexander modules and distinct knot groups.
An alternative supply follows from the hyperbolic
covering knots of
\cite[Theorem~2.2]{SilverWhittenHyperbolic},
which preserve the Alexander invariant and allow
arbitrarily large volume. Below we use Friedl's
and Kalfagianni's results directly.

\begin{proposition}[Hyperbolic knots with prescribed Alexander modules]
\label{prop:hyperbolic-alexander-supply}
\begin{enumerate}[label=\textup{(\roman*)},leftmargin=2.3em]
\item For every classical knot $K$, there exists a
hyperbolic knot $K^\sharp$ such that
\[
\cA(K^\sharp)\cong\cA(K)
\]
as $\Lambda$-modules.

\item There exists a sequence of hyperbolic knots
$L_1,L_2,\ldots$ with trivial Alexander modules
and pairwise nonisomorphic knot groups.
\end{enumerate}
\end{proposition}

\begin{proof}
For~\textup{(ii)}, Kalfagianni
\cite[Theorem~1.2]{KalfagianniVolume} supplies hyperbolic knots with
$\Delta(t)=1$ and arbitrarily large volume.  Choose them with strictly
increasing volumes.  If $V$ is a Seifert matrix, the presentation
matrix $tV-V^T$ then has unit determinant and is invertible over
$\Lambda$, so the Alexander module is zero.  Mostow--Prasad rigidity
makes isomorphic finite-volume hyperbolic knot groups have equal
volume. Thus the underlying groups, and therefore the
meridian-marked groups, are pairwise nonisomorphic.

For~\textup{(i)}, the zero-module case follows from~\textup{(ii)}.
Otherwise let $V$ be a Seifert matrix for $K$.  Friedl
\cite[Theorem~1.1]{FriedlSeifert} supplies a hyperbolic knot with a
Seifert matrix $S$ that is $S$-equivalent to $V$.  A unimodular
congruence changes $tV-V^T$ by invertible row and column operations.
For the two elementary enlargements
\[
V_1=
\begin{pmatrix}
V&0&0\\
v^T&0&0\\
0&1&0
\end{pmatrix},
\qquad
V_2=
\begin{pmatrix}
V&v&0\\
0&0&1\\
0&0&0
\end{pmatrix},
\]
their Alexander presentation matrices are
\[
tV_1-V_1^T=
\begin{pmatrix}
tV-V^T&-v&0\\
tv^T&0&-1\\
0&t&0
\end{pmatrix},
\qquad
tV_2-V_2^T=
\begin{pmatrix}
tV-V^T&tv&0\\
-v^T&0&t\\
0&-1&0
\end{pmatrix}.
\]
For $V_1$, add $t^{-1}v_i$ times the last row to the $i$-th old row
and then $tv_i$ times the last column to the $i$-th old column; for
$V_2$, use $tv_i$ times the last row and then $t^{-1}v_i$ times the
last column.  These operations give, respectively,
\[
(tV-V^T)\oplus
\begin{pmatrix}
0&-1\\
t&0
\end{pmatrix},
\qquad
(tV-V^T)\oplus
\begin{pmatrix}
0&t\\
-1&0
\end{pmatrix}.
\]
Both displayed $2\times2$ blocks are invertible.  Elementary
reductions reverse these operations, so $S$-equivalence preserves the
full cokernel of $tV-V^T$, not only its determinant.  The resulting
knot $K^\sharp$ is hyperbolic and hence nontrivial, with
$\cA(K^\sharp)\cong\cA(K)$.
\end{proof}

\begin{proof}[Proof of \cref{thm:main}]
Fix $g\ge3$ and a knot $K\subset S^3$.
By \cref{prop:hyperbolic-alexander-supply},
choose a hyperbolic knot $K^\sharp$ with
$\cA(K^\sharp)\cong\cA(K)$, and a sequence
of hyperbolic knots $L_1,L_2,\ldots$ with
trivial Alexander modules and pairwise
nonisomorphic knot groups. Use $K^\sharp$ for the first preliminary
meridian-disk surgery and knots from the
sequence $(L_r)$ for all remaining surgeries.
Choose distinct members of this sequence
for the vertical-boundary and ordered
surgeries, so their meridian-marked knot
groups are pairwise nonisomorphic, as required
by \cref{thm:construction}.

Fix one of the ordered surgeries, say the
$i_0$-th, and keep all other knots fixed.
At this stage, use a different unused knot
from $(L_r)$ for each $n\ge1$. This gives
a sequence of surfaces $F_{g,n}$ with exterior
groups $G_{N,n}$ and positive meridians $\mu_n$. Each surface satisfies the hypotheses of
\cref{thm:construction}. Hence
\cref{prop:each-constructed-surface-rigid} gives
\[
E_{\mathrm{TOP}}^+(F_{g,n})
=E^+(F_{g,n})=1.
\]
Since all surgery knots except $K^\sharp$
have trivial Alexander modules,
\cref{cor:total-alexander} gives
\[
\cA_1(F_{g,n})
\cong\cA(K^\sharp)
\cong\cA(K).
\]

It remains to show that these surfaces are
pairwise topologically inequivalent.
Suppose an orientation-preserving homeomorphism
of oriented pairs carries $F_{g,n}$ to
$F_{g,m}$. By \cref{prop:peripheral-naturality},
it induces an isomorphism of meridian-marked
exterior groups
\[
(G_{N,n},\mu_n)\cong(G_{N,m},\mu_m).
\]
The conditions in \cref{def:marked-centralizer}
are preserved by such an isomorphism.
By \cref{thm:tower-spectrum}, these conditions
recognize exactly the meridian-marked knot
groups used in the vertical-boundary and
ordered surgeries, up to isomorphism.

If $n\ne m$, the knot group used at the $i_0$-th
surgery for $F_{g,n}$ occurs in its recognized
collection but not in that of $F_{g,m}$.
This contradicts the isomorphism above.
Hence the surfaces are pairwise topologically
inequivalent.
\end{proof}

\begin{proof}[Proof of \cref{cor:simply-connected-ambient}]
Fix $X$, $g$, and $K$, and take the surfaces
$F_{g,n}\subset S^4$ constructed in the proof of
\cref{thm:main}. The entire construction takes place
in a fixed smooth $4$-ball $B\subset S^4$.
Identify $B$ with a
$4$-ball in $X$, and regard the same surfaces as
embedded in $X$.

Let $E_{B,n}=B\setminus\operatorname{int}\nu F_{g,n}$
and $C=X\setminus\operatorname{int}B$.
Since
\[
X=C\cup_{S^3}B
\]
and both $B$ and $S^3$ are simply connected,
van Kampen gives $\pi_1(C)=\pi_1(X)=1$. Now the surface exteriors decompose as
\[
E_{X,n}=C\cup_{S^3}E_{B,n},
\qquad
E_{S^4,n}
=(S^4\setminus\operatorname{int}B)
\cup_{S^3}E_{B,n}.
\]
Since $C$ and $S^4\setminus\operatorname{int}B$
are simply connected, applying van Kampen
to these two decompositions gives
inclusion-induced isomorphisms
\[
\pi_1(E_{X,n})
\cong\pi_1(E_{B,n})
\cong\pi_1(E_{S^4,n}).
\]
These isomorphisms preserve the positive meridian
and the push-off homomorphism, since the surface,
its normal section, and all pushed-off loops
lie in $B$. They therefore preserve the
Alexander modules.
By \cref{thm:main},
\[
\cA_1(F_{g,n}\subset X)
\cong\cA_1(F_{g,n}\subset S^4)
\cong\cA(K).
\]

Now let $f$ be a mapping class induced by an
orientation-preserving homeomorphism of
$(X,F_{g,n})$ that preserves the surface orientation.
The proof of \cref{prop:peripheral-naturality}
applies equally to $X$. Using
\cref{prop:early-homology,cor:rooted-normalization},
we obtain an automorphism $\Phi$ of the exterior
group satisfying
\[
\Phi(\mu_n)=\mu_n,
\qquad
\Phi\circ\rho_{N,n}=\rho_{N,n}\circ f_*.
\]
By \cref{thm:peripheral-recognition}, $f=1$.
Hence
\[
E_{X,\mathrm{TOP}}^+(F_{g,n})
=E_X^+(F_{g,n})=1.
\]

Finally, suppose a homeomorphism of oriented
pairs in $X$ carries $F_{g,n}$ to $F_{g,m}$.
It induces an isomorphism of their
meridian-marked exterior groups.
If $n\ne m$, the knot group used at the varying
ordered surgery for $F_{g,n}$ occurs in its
recognized collection but not in that of
$F_{g,m}$, by \cref{thm:tower-spectrum}.
This contradicts the isomorphism.
Thus the surfaces are pairwise topologically
inequivalent.
\end{proof}

\subsection{Group and exterior homology}

This is ordinary group homology of classifying spaces, not the
Alexander module of a meridional cover.

\begin{proposition}[Group homology of the exterior groups]
\label{prop:tower-homology}
For the preliminary group
\[
A_0=
K_1^D*_{\langle\mu\rangle}\cdots
*_{\langle\mu\rangle}K_g^D,
\]
one has
\[
H_1(A_0;\Z)\cong\Z\langle[\mu]\rangle,
\qquad
H_2(A_0;\Z)=0.
\]
For every $0\le i\le N$,
\begin{equation}
\label{eq:tower-H1-H2}
H_1(G_i;\Z)\cong\Z\langle[\mu]\rangle,
\qquad
H_2(G_i;\Z)=0.
\end{equation}
Moreover, if
$\rho:\pi\to A_0$ is the push-off homomorphism after the preliminary
surgeries, then
\[
\rho(\pi)\subset[A_0,A_0],
\]
and
\begin{equation}
\label{eq:surface-in-commutator}
\rho_i(\pi)\subset[G_i,G_i]
\qquad(0\le i\le N).
\end{equation}
\end{proposition}

\begin{proof}
The first-homology and commutator-subgroup assertions are exactly
\cref{prop:early-homology}; it remains to prove the two $H_2$ claims.
Every hyperbolic knot exterior is aspherical and has the homology of
$S^1$, so for each knot group $K$ used here,
\[
H_1(K;\Z)=\Z\langle[m]\rangle,
\qquad
H_2(K;\Z)=0.
\]

For $A_0$, use the star-shaped graph of groups with central vertex
$\langle\mu\rangle$, leaf groups $K_j^D$, and edge groups
$\langle\mu\rangle$.  The central and leaf vertex groups have
vanishing second homology, so the degree-two part of Mayer--Vietoris is
\[
0\longrightarrow H_2(A_0;\Z)
\longrightarrow\Z^g
\xrightarrow{\delta_1}\Z\oplus\Z^g,
\]
where
\[
\delta_1(n_1,\ldots,n_g)
=
\bigl(-\textstyle\sum_jn_j,n_1,\ldots,n_g\bigr).
\]
The leaf coordinates make $\delta_1$ injective, and hence
$H_2(A_0;\Z)=0$.

Now consider one rank-two attachment
\[
G'=G*_{A}B,
\qquad
A=\langle\mu,h\rangle\cong\Z^2,
\qquad
B=K\times\langle z\rangle,
\qquad z=h.
\]
At every such stage,
\[
[h]=0\in H_1(G;\Z)
\]
by \cref{prop:early-homology}.  Hence the degree-one edge map is
\[
[\mu]\longmapsto([\mu],-[m]),
\qquad
[h]\longmapsto(0,-[z]),
\]
and is injective because $[m],[z]$ form a basis of $H_1(B;\Z)$.
Moreover,
\[
H_2(B;\Z)
\cong
H_1(K;\Z)\otimes H_1(\langle z\rangle;\Z)
\cong\Z\langle[m]\wedge[z]\rangle,
\]
while
\[
H_2(A;\Z)\cong\Z\langle[\mu]\wedge[h]\rangle.
\]
The $H_2(B)$-component of
\[
\delta_2:H_2(A;\Z)
\longrightarrow H_2(G;\Z)\oplus H_2(B;\Z)
\]
sends the displayed generator isomorphically, up to sign, to
$[m]\wedge[z]$.  In the exact segment
\[
H_2(A)\xrightarrow{\delta_2}H_2(G)\oplus H_2(B)
\longrightarrow H_2(G')
\longrightarrow H_1(A)\xrightarrow{\delta_1}H_1(G)\oplus H_1(B),
\]
injectivity of $\delta_1$ gives
\[
H_2(G';\Z)\cong\operatorname{coker}\delta_2
\cong H_2(G;\Z).
\]
Starting from $A_0$ and applying this at the vertical-boundary and then
the ordered rank-two stages proves $H_2(G_i;\Z)=0$ for every $i$.
\end{proof}

The singular homology of the four-dimensional exterior is different
from the preceding group homology.

\begin{proposition}[Homology of the surface exterior]
\label{prop:complement-homology}
Let $F=F_N\subset S^4$ be any final genus-$g$ surface in the
construction, and put
\[
E_F=S^4\setminus\operatorname{int}\nu F.
\]
Then
\[
H_0(E_F;\Z)=\Z,
\qquad
H_1(E_F;\Z)=\Z,
\qquad
H_2(E_F;\Z)=\Z^{2g},
\]
\[
H_3(E_F;\Z)=0,
\qquad
H_4(E_F;\Z)=0.
\]
Consequently
\[
\chi(E_F)=2g.
\]
These statements are compatible with
$H_2(\pi_1(E_F);\Z)=0$ from \cref{prop:tower-homology}.
\end{proposition}

\begin{proof}
Radial projection in the punctured normal disk fibers, extended by the
identity outside the normal neighborhood, gives
\[
S^4\setminus F\simeq E_F.
\]
Alexander duality gives
\[
\widetilde H_i(S^4\setminus F;\Z)
\cong
\widetilde H^{3-i}(F;\Z).
\]
For a connected closed oriented genus-$g$ surface this yields
\[
\widetilde H_0=0,
\qquad
H_1\cong H^2(F;\Z)\cong\Z,
\qquad
H_2\cong H^1(F;\Z)\cong\Z^{2g},
\qquad
H_3\cong\widetilde H^0(F;\Z)=0.
\]
Since $E_F$ is a compact connected $4$-manifold with nonempty
boundary, $H_4(E_F;\Z)=0$, and the displayed groups give
$\chi(E_F)=1-1+2g=2g$.

Writing $G_F=\pi_1(E_F)$, the group $H_2(G_F;\Z)$ is the homology of
a classifying space, whereas $H_2(E_F;\Z)$ is singular homology of
the actual exterior; no asphericity of $E_F$ is asserted.  The Hopf
exact sequence
\[
\pi_2(E_F)\longrightarrow H_2(E_F;\Z)
\longrightarrow H_2(G_F;\Z)\longrightarrow0
\]
therefore explains the compatibility of
$H_2(E_F;\Z)\cong\Z^{2g}$ with $H_2(G_F;\Z)=0$.
\end{proof}

\end{document}